\documentclass[12pt,reqno]{amsart}
\usepackage[left=1in,right=1in,top=1in,bottom=1in]{geometry}
\usepackage{color}
\usepackage{mathtools}
\usepackage{amsmath,amsfonts,amsthm}
\usepackage{comment}
\usepackage{physics}
\usepackage{amssymb}
\usepackage{amsaddr}
\usepackage{graphicx}
\usepackage{caption}
\usepackage{subcaption}
\usepackage{enumitem}
\usepackage[english]{babel}
\usepackage[utf8]{inputenc}
\usepackage{color,soul}
\usepackage[colorlinks=true, allcolors=blue]{hyperref}
\usepackage{tikz}    
\usepackage[backend=biber,maxbibnames=99]{biblatex}
\usepackage{extarrows}
\usepackage{csquotes}
\usetikzlibrary{shapes.geometric}
\usepackage{pgfplots}
\usepackage{tkz-euclide}
\usetikzlibrary{decorations.pathreplacing}
\pgfplotsset{compat=1.18}

\newtheorem{thm}{Theorem}[section]
\newtheorem{cor}[thm]{Corollary}
\newtheorem{lem}[thm]{Lemma}
\newtheorem{prop}[thm]{Proposition}
\newtheorem{exmp}[thm]{Example}
\newtheorem{defn}[thm]{Definition}
\newtheorem{rem}[thm]{Remark}

\newtheorem{hypothesis}{Hypothesis}

\theoremstyle{definition}

\numberwithin{equation}{section}

\newcommand{\eps}{\varepsilon}

\newcommand{\cB}{\mathcal{B}}
\newcommand{\cX}{\mathcal{X}}
\newcommand{\C}{\mathbb{C}}

\newcommand{\N}{\mathbb{N}}

\newcommand{\R}{\mathbb{R}}
\newcommand{\Z}{\mathbb{Z}}

\newcommand{\operator}[1]{\Lambda_{g,A,V,\cX}}
\newcommand{\opnum}[1]{\Lambda_{g_{#1}, A_{#1}, V_{#1}, \cX_{#1}}}
\newcommand{\magneticnum}[1]{\mathcal{L}_{g_{#1},A_{#1},V_{#1}}}
\newcommand{\magnetic}{\mathcal{L}_{g,A,V}}

\def\hat{\widehat}
\def\tilde{\widetilde}
\def \bfo {\begin {eqnarray*} }
\def \efo {\end {eqnarray*} }
\def \ba {\begin {eqnarray*} }
\def \ea {\end {eqnarray*} }
\def \beq {\begin {eqnarray}}
\def \eeq {\end {eqnarray}}

\def \det {\hbox{det}}

\def \p {\partial}

\def\hat{\widehat}
\def\tilde{\widetilde}
\def \bfo {\begin {eqnarray*} }
\def \efo {\end {eqnarray*} }
\def \ba {\begin {eqnarray*} }
\def \ea {\end {eqnarray*} }
\def \beq {\begin {eqnarray}}
\def \eeq {\end {eqnarray}}

\def \det {\hbox{det}}

\def \p {\partial}

\newcommand{\teemu}[1]{\footnote{Teemu: #1}}
\newcommand{\andrew}[1]{\footnote{Andrew: #1}}

\author{Teemu Saksala}
\address{Department of Mathematics, North Carolina State University, Raleigh, NC, USA   (\tt{tssaksal@ncsu.edu})}

\author{Andrew Shedlock}
\address{Department of Mathematics, North Carolina State University, Raleigh, NC, USA
(\tt{ashedlo@ncsu.edu})}


\title[A Hyperbolic Inverse Problem on Complete Riemannian Manifolds]{An Inverse Problem for symmetric hyperbolic Partial Differential Operators on Complete Riemannian Manifolds}
\date{\today}

\begin{document}

\begin{abstract}
    We show that a complete Riemannian manifold, as well as time independent smooth lower order terms appearing in a first order symmetric perturbation of a Riemannian wave operator can be uniquely recovered, up to the natural obstructions, from a local source-to-solution map of the respective hyperbolic initial value problem. Our proofs are based on an adaptation of the classical Boundary Control method (BC-method) originally developed by Belishev and Kurylev. The BC-method reduces the PDE-based problem to a purely geometric problem involving the so-called travel time data. For each point in the manifold the travel time data contains the distance function from this point to any point in a fixed \textit{a priori} known compact observation set. It is well known that this geometric problem is solvable. The main novelty of this paper lies in our strategy to recover the lower order terms via a further adaptation of the BC-method. 
\end{abstract}

\maketitle

\section{Introduction}\label{sec:introduction}

In this paper we show that a complete Riemannian manifold, as well as time independent smooth lower order terms appearing in a first order symmetric perturbation of a Riemannian wave operator can be uniquely recovered, up to the natural obstructions, from a local source-to-solution map related to the respective hyperbolic initial value problem.
This type of hyperbolic inverse problem is the mathematical underpin for many imaging techniques appearing in physics, engineering, geology, and medical imaging.
For instance, in Seismic Exploration one uses the seismic energy to probe beneath the surface of the Earth, usually as an aid in searching for economic deposits of oil, gas, or minerals, but also for engineering, archeological, and scientific studies. Typically,  this is done by setting up explosive charges or seismic vibrators in some area of the surface and trying to receive the echoes of the reflections of the seismic waves on some measurement area.  
%
% \subsection{Physical relevance}
% Inverse problems with partial boundary data are encountered in mathematical physics and in various applications. For example in medical imaging and in the geophysical imaging of the Earth, measurements can usually be done only for a part of the boundary. Often it is not possible to observe fields on the same area where sources are controlled. For example in oil exploration, explosives are used as sources and hence it is difficult to measure waves near the sources. Also, many inverse scattering problems, such as the transmission problems on a line, are equivalent to disjoint partial data problems.
% \teemu{From \cite{lassas2014inverse}}
%
% This type of hyperbolic inverse problem is the mathematical underpin for many imaging techniques appearing in physics, engineering, geology, and medical imaging. 
% For instance, in geophysics, we may have only partial data about the earth's subsurface structure, and we may need to use techniques such as seismic tomography to estimate the distribution of materials and their properties.
% \teemu{From \cite{lassas2024disjoint}}

% \subsection{Optimality of the method}
Our proofs are based on an adaptation of the classical Boundary Control method (BC-method) originally developed by Belishev and Belishev-Kurylev \cite{belishev_russian, belishev1992reconstruction}. The BC-method reduces the PDE-based problem to a purely geometric problem involving the so-called travel time data. For each point in the manifold this data contains the distance function from this point to any point in a fixed \textit{a priori} known compact observation set. It is well known that this geometric problem is solvable. The main novelty of this paper lies in our strategy to recover the lower order terms via a further adaptation of the BC-method. However, we also present a streamlined version for the proof of the global recovery of  geometry.

As we do not pose any geometric assumptions, including compactness, we cannot solely rely on spectral theoretic arguments. Furthermore, we do not make any \textit{a priori} assumptions for the size or the supports of the lower order terms, in particular the lower order terms are allowed to be unbounded. To recover the lower order terms we provide a higher order approximate controllability result, which is based on Tataru's unique continuation principle for hyperbolic operators with coefficients that depend on time real analytically \cite{tataru1995unique}. This implies that for any point in space-time there is a wave that we can send from our observation set and which does not vanish at this given point. Due to the limitations of the BC-method, the non-symmetric and time-dependent  hyperbolic operators remain outside the scope of the current paper as well as sending and receiving the waves in disjoint sets.

\subsection{Problem Setting and the Main result}

Let $(N,g)$ be a complete and connected, but not necessarily compact, smooth Riemannian manifold. Also, we let $A$ and $V$ to be a smooth real co-vector field and a smooth real-valued function on $N$ respectively. Then, we define the \textit{Magnetic-Schr\"odinger Operator} $\mathcal{L}_{g,A,V}$ which acts on the space $C^\infty(N)$ of smooth complex valued functions according to the formula
\begin{align}\label{def:L}
    \mathcal{L}_{g,A,V}(u) := -\Delta_gu -2i\langle A,du\rangle_g + (id^*A + |A|_g^2 + V)u.
\end{align}
In above $\Delta_g$ is the Laplace-Beltrami operator of $(N,g)$, $i$ is the imaginary unit, $d$ stands for the exterior derivative while $d^\ast$ is its formal $L^2$-adjoint, $\langle \cdot,\cdot\rangle_g$ and $|\cdot|_g$ are the Riemannian inner product and the respective norm acting on co-vectors. 
We let $dV_g$ to be the standard smooth density associated with $(N,g)$. Note that the smooth density $dV_g$ is the absolute value of the Riemannian volume form in local coordinates and is used since we do not assume that $N$ is oriented. Thus, on the space $C_0^\infty(N)$ of compactly supported smooth complex valued functions we have the inner product $(\cdot,\cdot)_g$ defined by
\begin{align}\label{inner_product}
    (u,v)_g :&= \int_N u\overline{v}dV_g, \quad \text{for all $u,v\in C_0^\infty(N)$.}
\end{align}
By \eqref{def:L} the Magnetic-Schr\"odinger operator is symmetric on $C_0^\infty(N)$ with respect to the inner product $(\cdot,\cdot)_g$, that is $(\mathcal{L}_{g,A,V}u,v)_g = (u,\mathcal{L}_{g,A,V}v)_g$ for all $u,v\in C_0^\infty(N)$ since $A$ and $V$ are real. 

In this paper we study the following hyperbolic initial value problem 
\begin{align} \label{eq:cauchy_problem_infinite_time}
    \begin{cases}
        (\partial_t^2 + \mathcal{L}_{g,A,V})u(t,x) = f(t,x), \quad \text{ for $(t,x)$ in } (0,\infty) \times N
        \\
        u(0,x) =\partial_tu(0,x) = 0, \quad \text{ for $x$ in } N.
    \end{cases}
\end{align}
Here the interior source (forcing function) $f$ is assumed to be smooth and compactly supported in a set $(0,\infty) \times \cX$ where the observation set $\cX \subset N$ is open and fixed through out the paper. 

In Section \ref{Sec:Well-Poss} of this paper we show that problem (\ref{eq:cauchy_problem_infinite_time}) always carries a unique smooth solution. Thus, the \textit{local source-to-solution map}
\begin{equation}
    \label{eq:source_to_sol_map}
    \operator{}\colon C^\infty_0((0,\infty)\times \cX)\to C^\infty((0,\infty)\times \cX), \quad \operator{}(f)=u^f|_{(0,\infty)\times \cX}, 
\end{equation}
is well defined. Here, for a given source function $f$ the function $u^f$ is the unique solution to problem \eqref{eq:cauchy_problem_infinite_time}. 

We also show in  Section \ref{Sec:Well-Poss} that if $\kappa$ is a smooth complex valued but unitary function on $N$ that equals 1 in the observation set $\cX$, then $\mathcal{L}_{g,A+i\kappa^{-1}d\kappa,V}$ is a Magnetic-Schr\"odinger operator while the associated source-to-solution map $\Lambda_{{g},{A}+{i\kappa^{-1}d\kappa,{V}, \cX}}$ satisfies the equation 
\[
\operator{}(f) 
    = \Lambda_{{g},{A}+{i\kappa^{-1}d\kappa,{V}, \cX}}(f), \quad \text{for all $f \in C_0^\infty((0,\infty)\times \cX)$}.
\]
Hence, the local source-to-solution map does not uniquely determine the co-vector field $A$.

% Since the initial value problem \eqref{eq:cauchy_problem_infinite_time} is invariant under Riemannian isometries 
% of $(N,g)$, which equal to the identity map on the observation set $\cX$, the best we can hope for is to determine a Riemannian manifold from \eqref{eq:source_to_sol_map} up to a Riemannian isometry. Therefore, 
We set the following notation of equivalence between local source-to-solution operators of two Riemannian manifolds with respective Magnetic-Schr\"odinger operators. 

\begin{hypothesis}\label{hyp:same_data}
    For $i \in \{1,2\}$ let $(N_i,g_i)$ to be a complete and connected smooth Riemannian manifold. Let $\cX_i \subset N_i$ be a nonempty and open set called the observation set.  Let $A_i$ and $V_i$ to be a smooth real co-vector field and a smooth real-valued function on $N_i$ respectively. Suppose that there is a diffeomorphism $\phi:\cX_1\to \cX_2$ such that  
    \begin{align*}
        \tilde{\phi}^*\opnum{2}(f) &= \opnum{1}(\tilde{\phi}^*f), \quad \text{ for all } f\in C_0^\infty((0,\infty)\times \cX _2)
    \end{align*}
    where $\tilde{\phi}(t,x) = (t,\phi(x))$, and  $\tilde \phi ^\ast f:=f\circ \tilde \phi$.
\end{hypothesis}

In particular if Hypothesis (\ref{hyp:same_data}) is satisfied then the equation
\[
u_1^{f\circ \tilde \phi}=u_2^f\circ \tilde \phi \text{ on } (0,\infty) \times \cX, 
\]
holds for all $f\in C_0^\infty((0,\infty)\times \cX _2)$. Here $u_1^{f\circ \tilde \phi}$ and $u_2^f$ are the solutions to the Cauchy problem \eqref{eq:cauchy_problem_infinite_time} with the respective sources and $\magneticnum{}=\magneticnum{i}$. 
% In this case we say that the local source-to-solution operators $\opnum{1}$ and $\opnum{2}$ are equivalent. 
% The pullback $\tilde \phi^\ast$ is given by the equation $\tilde \phi ^\ast f:=f\circ \tilde \phi$. 
We want to emphasize that due to the coordinate invariance of the initial value problem \eqref{eq:cauchy_problem_infinite_time} the Hypothesis \ref{hyp:same_data} holds true if $\cX_i=N_i$ and $\phi$ is a Riemannian isometry with $A_1 = \phi^*A_2$ and $V_1 = \phi^*V_2$. Here $\phi^*A_2$ stands for the pullback of a co-vector field $A_2$.

Our main theorem is as follows.

\begin{thm}
\label{thm:main_thm}
Let $(N_i,g_i)$ for $i \in \{1,2\}$ be a complete and connected Riemannian manifold with an open observation set $\cX_i \subset N_i$. Let $A_i$ and $V_i$ be a smooth real co-vector field and a smooth real-valued function on $N_i$ respectively. If the Hypothesis \ref{hyp:same_data} is true, then there exists a unique diffeomorphism $\Phi:N_1\to N_2$ and a smooth map $\kappa:N_1\to\mathbb{C}$ such that the following are also true:
         \begin{itemize}
            \item $\Phi|_{\cX_1} = \phi|_{\cX_1}$
            \item $\kappa= 1$ on $\cX_1$ and $|\kappa| = 1$ on $N_1$
            \item $g_1 = \Phi^*g_2$  
            \item $A_1 = \Phi^*A_2 + i\kappa^{-1}d\kappa$  
            \item $V_1 = \Phi^*V_2$.
        \end{itemize}
\end{thm}

\begin{rem}
    If the manifolds $N_1$ and $N_2$, as in Theorem \ref{thm:main_thm}, are closed it is sufficient to know that Hypothesis \ref{hyp:same_data} is true in the finite space-time cylinder $(0,T)\times N_i$ where
    $$ T>2\max\{\text{diam}_{g_1}(N_1), \text{diam}_{g_2}(N_2)\}.$$
\end{rem}

The version of Theorem \ref{thm:main_thm} without lower order terms was originally given in \cite{source-to-solution}. In this paper one of our goals is to carefully argue why Hypothesis \ref{hyp:same_data} allows us to carry over all the geometric properties that can be recovered from the local source-to-solution map of the manifold $(N_1,g_1)$ to the manifold $(N_2,g_2)$. In particular, we do not need to shrink the measurement sets $\cX_i$, nor pose any other assumptions in addition to the openness. 
Furthermore, we show that if Hypothesis \ref{hyp:same_data} is valid then there is a unique Riemannian isometry $\Phi$ between $(N_1,g_1)$ and $(N_2,g_2)$ such that $\Phi|_{\cX_1}=\phi$.
By doing so we hope to streamline and extend the scope of \cite[Section 5]{source-to-solution}.

\subsection{Previous Literature} We survey some known results similar or related to Theorem \ref{thm:main_thm}.
In such coefficient determination problems,
it is typical to use
the Dirichlet-to-Neumann map (DN-map) to model the data.
Apart from immediate applications in seismic explorations and medical ultrasound imaging, this is reasonable since
several other types of data can be reduced to the 
DN-case. For instance, in \cite{Nachman1988a}
an inverse scattering problem is solved via a reduction to 
the inverse conductivity problem in \cite{Sylvester1987}, and the latter uses the DN-map as data. 
In the present paper, however,
we do not perform a reduction to the DN-map but adapt techniques originally developed in that case to the case of local source-to-solution map $\operator{}$.

The approach that we use in this current paper is a modification of the Boundary Control method (BC-method). This method was first developed by Belishev for the acoustic wave equation on $\R^n$ with an isotropic wave speed
\cite{belishev_russian}. A geometric version of the method, suitable when the wave speed is given by a Riemannian metric tensor as in the present paper, was introduced by Belishev and Kurylev \cite{belishev1992reconstruction}.
We refer to \cite{lassas_inverse} for a thorough review of the related literature. 
Local reconstruction of the geometry from the local source-to-solution map $\operator{}$ has been studied as a part of iterative schemes, see for instance \cite{Isozaki2010,kurylev2018inverse}. In the present paper we give a global uniqueness proof that does not rely on an iterative scheme. Our main result is a significant extension of \cite{source-to-solution} where an analogous problem was studied for the Riemannian wave equation in the absence of the lower order terms $A$ and $V$ which have no boundedness or integrability constraints. 

Recently, \cite{lassas2024disjoint} extended \cite{source-to-solution} to the direction in which solutions of the initial value problem \eqref{eq:cauchy_problem_infinite_time} are observed in a different set than where the sources are supported. This result was achieved in the absence of the lower order terms and requires compactness of the manifold as well as an additional spectral bound condition for the Laplace spectrum. 
In the DN-case a partial data result was earlier given in \cite{lassas2014inverse, kian2019unique}. Of these two papers the first one shows that the DN-map, with Dirichlet and Neumann data measured in some disjoint relatively open subsets of the boundary, uniquely recovers a compact Riemannian manifold with boundary. The result was obtained under the Hassell–Tao condition for eigenvalues and eigenfunctions of the respective Dirichlet Laplacian. This condition is valid for instance if the respective hyperbolic equation is exactly controllable. 
The second paper \cite{kian2019unique} considers a disjoint data problem for the DN-map associated with a wave type operator whose spatial part is a first order perturbation of the Laplacian.  The authors show that this data determines uniquely, up the natural gauge invariance, the lower order terms in a neighborhood of the set where the Neumann data is measured. This was obtained by assuming that this set is strictly convex and that the wave equation is exactly controllable from the set where the Dirichlet data are supported. The paper also contains a global result under a convex foliation condition. We leave the problem of recovering the lower order terms, that live in a closed manifold, by observing the solutions of the initial value problem \eqref{eq:cauchy_problem_infinite_time} in a different set than where the sources are supported to a future development.
% Our current paper extends many of the methods presented in \cite{kian2019unique} to the case when interior measurements are made on a non-compact manifold.

In the present work we restrict our attention to the unique solvability of the inverse problem but note that several variants of the Boundary Control method have been studied computationally \cite{belishev1999dynamical,Hoop2016,kabanikhin2004direct,pestov2010numerical, oksanen2024linearized}. Related stability questions have been investigated for instance in \cite{anderson2004boundary,korpela2016regularization,Liu2012, burago2020quantitative, Bosi2022Reconstruction}. 

Due to their dependency on the unique continuation principle \cite{tataru1995unique} our methods are not applicable for problems where the coefficients of the hyperbolic operator also depend on time. The closest adaptations of the BC-method to this line of studies were obtained recently in \cite{Alexakis_Feiz_Oksanen_22,Alexakis_Feiz_Oksanen_23} where the authors show that time dependent zeroth order coefficients can be recovered from the DN-map if the known globally hyperbolic Lorentzian structure satisfies certain curvature bounds. In the ultra-static case, like in the current paper, many papers including but not limited to
\cite{Kian_Oksanen, liu2023partial, liu2025recovery, feizmohammadi2021recovery, bellassoued2011stability} rely on the construct of so called Geometric Optic Solutions which can be used to reduce the unique recovery of the time-dependent or stable recovery of the time-independent lower order terms to an integral geometric problem.

In this paper we use the BC-method to reduce the PDE-based problem to a purely geometric problem in which the aim is to recover the Riemannian manifold with or with out boundary from a collection of distance functions related to a dense set of interior point sources. This data is called the Travel Time Data (TTD) and it was first proven in \cite{kurylev_boundary_distance_map, lassas_inverse} that the TTD determine a smooth compact Riemannian manifold with boundary. This uniqueness result was extended for compact Finsler manfolds in \cite{de2019inverse}, which corresponds to the measuring of travel times of polarized elastic waves,  that cannot be modeled in Riemannian formalism.
In particular, \cite{de2019inverse} shows that any Finsler metric arising from elasticity is determined by the TTD.  
In \cite{pavlechko2022Uniqueness} it was proved that the restriction of the travel time functions on any fixed relatively open subset of a strictly convex boundary determine the Riemannian manifold.
Furthermore, it is known that the travel time data always determines a compact Riemannian manifold with boundary at least H\"older stably \cite{katsuda2007stability}.
However, with better geometries the recovery of the manifold from its TTD can be some times made Lipschitz stable  \cite{ilmavirta2023three, ilmavirta2024lipschitzstabilitytraveltime}.
Finally, \cite{de2021stable} introduced a new physically more relevant geometric inverse problem for TTD with only finitely many source points, constructed a discrete approximation of a Riemanniann manifold and provided quantitative and fully data-driven error estimates for the approximation.
% The proof builds heavily on the simplicity,  which implies that the travel time map, sending each interior point to its distance function to the boundary points, is a metric isometry.  
%In conjunction with the Myers-Steenrod theorem: \textit{Every metric isometry between Riemannian manifolds is a Riemannian isomery i.e.  a smooth map that preserves the metric tensor}, this implies that the Gromov-Hausdorff distance of two simple Riemannian manifolds is bounded from above by the Hausdorff distance of their TTD. 

When the travel time measurements are done in an open measurement set, as in the current paper, the first uniqueness result was presented in \cite{source-to-solution}. 
Recently, in \cite{ilmavirta2024lipschitzstabilitytraveltime},  with collaborators, the authors of the present paper showed that the TTD determine a closed Riemannian manifold Lipschitz stably under certain geometric assumptions. 
% This paper also extended the scope of \cite{ilmavirta2023three} beyond simple manifolds allowing non-trivial topologies and interior conjugate points.  

% The recovery of a Riemannian manifold from its TTD has been used to solve many other geometric inverse problems.  For instance, \textit{Broken Scattering Relation} (BSR) records all the possible entering and exiting directions and the total travel times of broken geodesics, a curve concatenated by two geodesic segments staring at the boundary and meeting at some interior point.
% It was shown in \cite{kurylev2010rigidity} that this relation determines a compact Riemannian manifold via recovering the TTD.
% With my collaborators the we solved the Finslerian version of this problem in \cite{de2020foliated}.  In \cite{ilmavirta2023three} we showed that BSR determines any simple surface.  This result was not included in \cite{kurylev2010rigidity} as the method presented there in which was used to recover the TTD from BSR cannot handle the two dimensional case.

\subsection{Outline of the proof of Theorem \ref{thm:main_thm}}

Before proving Theorem \ref{thm:main_thm} we show that the local source-to-solution operator is well-defined (See Theorem \ref{thm:uniqueness_of_smooth_solutions}) and that the prescribed gauge to uniqueness is valid (See Proposition \ref{prop:existence_of_gauge}). We devote Section \ref{Sec:Well-Poss} for these goals.
% we prove that the source-to-solution operator is well-defined by showing that the Cauchy problem $\ref{eq:cauchy_problem_infinite_time}$ has a unique smooth solution $u$ provided that $f\in C_0^\infty((0,\infty)\times N)$. 
In this section we also build several important tools essential for the proof of the main theorem. Among these tools are a higher order approximate controllability result (See Theorem \ref{thm:approximate_controllability}), which allows us to approximate Sobolev functions of arbitrary degree by waves sent from the observation set, and 
prove the existence of a wave that does not vanish at a given point but is sent from the observation set
(See Corollary \ref{cor:nonvanishing_condition}).
We also recall the Blagovestchenskii identity (see Theorem \ref{thm:blagovestchenskii}), that helps us to obtain information about the state of the waves outside the cylinder $(0,\infty)\times \cX$ while we only send and measure waves in this set. 
% Finally, in Section \ref{Sec:Well-Poss} we prove that the lower order terms have an obstruction to uniqueness. 

The rest of the paper is split into two Sections: Recovery of the Geomety \ref{sec:geometry}, and the Recovery of the Lower Order Terms \ref{sec:lower_order_terms}.  
% The proof of Theorem \ref{thm:main_thm} can be broken down into determining $(N,g)$ from $(\cX,\operator{})$ and then determining $A$ and $V$ up to the gauge from $(N,g)$ and $(\cX,\operator{})$
In Section \ref{sec:geometry} we show that two manifolds that satisfy Hypothesis \ref{hyp:same_data} are indeed Riemannian isometric. This is done by showing the that local source-to-solution operators recover the travel time data (See Theorem \ref{thm:travel_time_data}). 
% proving that condition on the metric balls on $(N,g)$ satisfy a condition if and only wave starting in one region can be approximated by waves in another region. Once we are able to determine certain metric balls centered in $\cX$ we use the metric balls to determine the distance from any point in $N$ to any point in $\cX$. Once we know the distance between points in $N$ to any point in $\cX$ we will know the travel time data. 
We then show that the travel time data determines the topological, smooth and Riemannian structures (See Theorem \ref{thm:isometry_between_manifolds}). 
% In particular we show to use the travel time data and a diffeomorphism $\phi:\cX_1\to\cX_2$ to construct the Riemannian isometry from $(N_1,g_1)$ to $(N_2,g_2)$.

We begin Section \ref{sec:lower_order_terms} by showing that after recovering the Riemannian structure we can study the recovery of the lower order terms on \textit{a priori} known complete and connected Riemannian manifold $(N,g)$ with observation set $\cX$.
% we use the Riemannian isometry from $(N_1,g_1)$ to 
% $(N_2,g_2)$ to pullback 
% \\
% $(N_2,g_2,A_2,V_2)$ to $N_1$ and define a new source-to-solution operator on $(N_1,g_1)$ which is a change of coordinates of $\opnum{2}$ and is equal to $\opnum{1}$ on $C_0^\infty((0,\infty)\times \cX_1)$. Now that we know $(N,g)$, 
To establish the promised relationships between the lower order terms, we build special controllable sets (See Proposition \ref{prop:summary}) such that we can use a Lebesgue Differentiation theorem for the Blagovestchenskii identity and determine  pointwise products of two waves. 
% the Blagovestchenskii identity in order prove that the inner products between solutions for $\operator{1}$ and $\operator{2}$ are equal on a family of controllable sets. These special controllable sets are such that we can use a Lebesgue Differentiation Theorem at any point on $N$ to determine a pointwise product relationship between 
That is, if Hypothesis \ref{hyp:same_data} holds, we prove that for all $T>0$ and $f,h \in C^\infty_0((0,T)\times \cX))$ we have that $u_1^f(T,\cdot)\overline{u^h_1(T,\cdot)} =u_2^f(T,\cdot)\overline{u^h_2(T,\cdot)}$ in $N$ (See 
 Lemma \ref{lem:product_equal}). 
From here, we apply the higher order approximate controllability to build a smooth function $\kappa$ such that $u_1^f(T,\cdot) = \kappa(\cdot) u_2^f(T,\cdot)$ for any function $f\in C_0^\infty((0,T)\times N)$ (See Lemma \ref{lem:big_result}). Finally, we show that this identity implies the gauge for the lower order terms (See Theorem \ref{thm:gauge}).

\subsection*{Acknowledgements}
T.S. and A.S. were supported by the National Science Foundation grant DMS-2204997.
We thank Matti Lassas and Lauri Oksanen for numerous discussions and hosting us on our visit to University of Helsinki in the spring of 2024.

\section{Preliminary Material}\label{Sec:Well-Poss}
Let $(N,g)$ be a complete and connected smooth Riemannian manifold. Also, let $A$ and $V$ to be a smooth real co-vector field and a smooth real-valued function on $N$ respectively. Finally, let $ T \in (0,\infty]$. In this section we study the existence, uniqueness, and some key properties of the solutions to the following hyperbolic Cauchy problem 
\begin{align} \label{eq:cauchy_problem}
    \begin{cases}
        (\partial_t^2 + \mathcal{L}_{g,A,V})u = f, \quad \text{ in } (0,T) \times N
        \\
        u(0,\cdot
        ) = u_0, \,\partial_tu(0,\cdot) = u_1.
    \end{cases}
\end{align}
% We a assume that the Magnetic-Schr\"odinger Operator $\mathcal{L}_{g,A,V}$, as in formula \eqref{def:L},  has smooth coefficients. 

First in Subsection \ref{sub:change_coordinates} we derive a change of coordinates formula for the Magnetic-Schr\"odinger operator. Next in Subsection \ref{sub:definitions} we set  some definitions needed to study the Cauchy problem. These include the definition of distributional solutions. Then in Subsection \ref{sub:closed_manifold} we use Spectral Theory to prove existence and uniqueness of distributional solutions on a closed manifold and for finite time interval. In Subsection \ref{sub:complete_manifold} we use results from Subsection \ref{sub:closed_manifold} to construct smooth solutions on a complete manifold for infinite time interval under the assumption that the data $u_0,u_1$ and $f$ consists of smooth compactly supported functions. Next in Subsection \ref{sub:approx_controllability} we prove several important properties for the solutions of the Cauchy problem which will be needed to prove the inverse problem stated in Theorem \ref{thm:main_thm}. These contain a higher order approximate controllability result and the Blagovestchenskii Identity. Finally, in Subsection \ref{sub:multiplicative_gauge} we prove that the gauge introduced in Theorem \ref{thm:main_thm} is natural for our inverse problem.

\subsection{Change of Coordinates} \label{sub:change_coordinates}
In this subsection we prove a change of coordinates formula for the Magnetic-Schr\"odinger operator which will be used to prove the well-posedness of the Cauchy problem \eqref{eq:cauchy_problem} in a complete Riemannian manifold via connecting the problem to a sequence of similar problems on closed manifolds. See Subsection \ref{sub:complete_manifold}. The results in the current subsection are well known in differential geometry and we include them for the readers convenience. 

\begin{lem} \label{lem:laplace_isometry}
    Let $(N_1,g_1)$ and $(N_2,g_2)$ be Riemannian manifolds with Laplace-Beltrami operators $\Delta_1$ and $\Delta_2$ respectively. If $\phi:N_1\to N_2$ is a diffeomorphism, then $\phi:(N_1,g_1)\to (N_2,g_2)$ is a Riemannian isometry if and only if
    \begin{align*}
        \Delta_1 (w\circ \phi) = [\Delta_2 (w)]\circ \phi \hspace{4mm}\text{for all $w\in C_0^\infty(N_2)$.}
    \end{align*}
\end{lem}

\begin{proof}
    For the backward direction and will work entirely in local coordinates. Let $V\subset N_1$ be open and be so small that we can find local coordinates $(x^1,\ldots, x^n)$ on $V$ while $(y^1, \ldots, y^n)$ are the respective local coordinates on $\phi(V)\subset N_2$. Now in these local coordinates, $g_1(x) = g_{1,ij}(x)dx^idx^j, g_2 = g_{2,ij}(y)dy^idy^j$ and $\phi = (\phi^1, \ldots,\phi^n)$.
    %where the inverse of the symmetric matrices $G = \{g_{ij}\}$ and $H = \{h_{ij}\}$ are the matrices $ G^{-1} = \{g^{ij}_1\}$ and $H^{-1} = \{h^{ij}\}$. 
    We let $\sqrt{g_1} := \sqrt{\text{det}(g_1)}$ and $\sqrt{g_2} := \sqrt{\text{det}(g_2)}$. Moreover, if $w\in C^\infty_0(\phi(V))$, then $w\circ\phi \in C^\infty_0(V)$. 
    
    By a direct computation in these local coordinates we have that
    \begin{align} 
        \Delta_1(w\circ \phi) 
        % &= \frac{1}{\sqrt{g_1}} \frac{\partial}{\partial x^j}\Big( g^{ij}_1\sqrt{g_1} \frac{\partial}{\partial x^i}(w\circ\phi)\Big) \nonumber\\
        % &= g^{ij}_1 \frac{\partial^2}{\partial x^i\partial x^j}(w\circ\phi) + \frac{\partial g^{ij}_1}{\partial x^j} \frac{\partial}{\partial x^i}(w\circ\phi) + \frac{g^{ij}_1}{\sqrt{g_1}} \frac{\partial \sqrt{g_1}}{\partial x^j}\frac{\partial}{\partial x^i}(w\circ\phi) \nonumber
       % \intertext{By the chain rule, in these local coordinates we have}
        %\frac{\partial}{\partial x^i}(w\circ\phi) &= \frac{\partial \phi^r}{\partial x^i} \Big(\frac{\partial w}{\partial y^r}\circ \phi\Big) \nonumber\\
        %\frac{\partial^2}{\partial x^i\partial x^j}(w\circ\phi) &= \frac{\partial \phi^r}{\partial x^i}\frac{\partial \phi^s}{\partial x^j} \Big(\frac{\partial^2 w}{\partial y^r\partial y^s}\circ\phi\Big) + \frac{\partial^2 \phi^r}{\partial x^i\partial x^j} \Big(\frac{\partial  w}{\partial y^r}\circ\phi\Big) \nonumber
        % \Delta_1(w\circ\phi) 
        &= 
        g^{ij}_1\frac{\partial \phi^r}{\partial x^i}\frac{\partial \phi^s}{\partial x^j} \Big(\frac{\partial^2 w}{\partial y^r\partial y^s}\circ\phi\Big) + \Big(g^{ij}_1 \frac{\partial^2\phi^r}{\partial x^i\partial x^j} + \frac{g^{ij}_1}{\sqrt{g_1}}\frac{\partial\sqrt{g_1}}{\partial x^j}\frac{\partial\phi^r}{\partial x^j} + \frac{\partial g^{ij}_1}{\partial x^j}\frac{\partial\phi^r}{\partial x^i}\Big)\Big(\frac{\partial w}{\partial y^r}\circ \phi\Big) \label{laplace_g}.
    \end{align}
    While 
    \begin{align}
        % \Delta_2(w) &= g^{rs}_2 \frac{\partial^2 w}{\partial y^r\partial y^s} + \frac{\partial g^{rs}_2}{\partial y^s}\frac{\partial w}{\partial y^r} + \frac{g^{rs}_2}{\sqrt{g_2}} \frac{\partial \sqrt{g_2}}{\partial y^s}\frac{\partial w}{\partial y^r}\nonumber
        % \intertext{and}
        [\Delta_2(w)]\circ\phi &= (g^{rs}_2\circ\phi) \Big(\frac{\partial^2 w}{\partial y^r\partial y^s} \circ\phi\Big)+ \Big[\Big(\frac{\partial g^{rs}_2}{\partial y^s} + \frac{g^{rs}_2}{\sqrt{g_2}} \frac{\partial \sqrt{g_2}}{\partial y^s}\Big)\circ\phi\Big]\Big(\frac{\partial w}{\partial y^r}\circ\phi\Big) \label{laplace_h}.    
    \end{align}

    Now for any point $p\in V$ and any indices $r_0, s_0\in \{1,\ldots,n\}$, there is a function $w\in C^\infty(\phi(V))$  so that in our local coordinates all of its first and second order partial derivatives, except $\frac{\partial^2 w}{\partial y^{r_0}\partial y^{s_0}} \circ\phi=1$, vanish at $p$.
    % we have that
    % \begin{align*}
    %     \Big(\frac{\partial w}{\partial y^r} \circ \phi\Big)(p) &= 0 \hspace{4mm}\text{for all $r\in\{1,\ldots,n\}$,}\\
    %     \Big(\frac{\partial^2 w}{\partial y^r\partial y^s} \circ\phi\Big) &= \begin{cases}
    %         1 \hspace{12mm}\text{if}\hspace{1mm} (r,s) = (r_0,s_0)\\
    %         0 \hspace{12mm}\text{if}\hspace{1mm} (r,s) \not=(r_0,s_0).
    %     \end{cases}    
    % \end{align*}
    By assumption we have that $\Delta_1(w\circ\phi) = [\Delta_2(w)]\circ\phi$, so for this specific function $w$ evaluated at $p$ we are able to conclude from $(\ref{laplace_g})$ and $(\ref{laplace_h})$ that 
    \begin{align}
        (g_2^{r_0s_0}\circ\phi)(p) &= g^{ij}_1(p)\frac{\partial \phi^{r_0}(p)}{\partial x^i}\frac{\partial \phi^{s_0}(p)}{\partial x^j} \label{coefficients_equal}.
    \end{align}
\begin{comment}
    This argument holds for all $r,s\in\{1,\ldots,n\}$. 
    We let $H^{-1}\circ\phi, G^{-1}, d\phi_p$ be the matrices with coefficients $\{(g^{rs}_2\circ\phi)(p)\}, \{g^{ij}_1(p)\}, \{\frac{\partial\phi^r(p)}{\partial x^i}\}$ respectively. The equality $(\ref{coefficients_equal})$ can be rewritten as the matrix product $H^{-1}\circ\phi = d\phi_p G^{-1} d\phi_p^T$ where $d\phi_p^T$ is the transpose of $d\phi_p$ and $H^{-1}\circ\phi$ is the pointwise matrix inverse of $H\circ\phi$. It then follows by properties of the inverses that $G = d\phi_p^T (H\circ\phi)d\phi_p$. If we look at the coefficients, then we have that
    $$
        g_{ij}(p) = (h_{rs}\circ\phi)(p) \frac{\partial\phi^r(p)}{\partial x^i}\frac{\partial \phi^s(p)}{\partial x^j}
    $$
    The equality above holds for all $p\in V$ from which we conclude that in these local coordinates we have 
    \begin{align}
        g_{ij} = (h_{rs}\circ\phi) \frac{\partial\phi^r}{\partial x^i}\frac{\partial \phi^s}{\partial x^j} \label{local_isometry}
    \end{align}
\end{comment}
    The equality $(\ref{coefficients_equal})$ is equivalent to $g_1 = \phi^*g_2$ in local coordinates. 
    
    The forward direction uses \eqref{laplace_g}--\eqref{coefficients_equal} to show that $\Delta_1(w\circ\phi) = [\Delta_2(w)]\circ\phi$ if $g_1 = \phi^* g_2$.
\end{proof}
The effect of the change of coordinates for the Magnetic-Schr\"odinger operator is described in the following lemma.

\begin{lem} \label{lem:magnetic_isometry}
    Let $(N_1,g_1)$ and $(N_2,g_2)$ be Riemannian manifolds with Magnetic-Schr\"odinger Operators $\magneticnum{1}$ and $\magneticnum{2}$ where all functions and covectors are smooth on their domains. Let $\phi:N_1\to N_2$ be a diffeomorphism. We have that $(g_1,A_1,V_1) = (\phi^*g_2,\phi^*A_2,\phi^*V_2)$ if and only if 
    \begin{equation}
    \label{eq:change_of_coordinates_for_L}
    \magneticnum{1}(w\circ \phi) = [\magneticnum{2}(w)]\circ \phi \hspace{4mm}\text{for all $w\in C^\infty_0(N_2)$.}
    \end{equation}
    In particular, if \eqref{eq:change_of_coordinates_for_L} holds then $\phi:(N_1,g_1)\to (N_2, g_2)$ is a Riemannian isometry. 
\end{lem}

\begin{proof}
    The forward direction follows from the definition of the pullback of tensors and by working in local coordinates as was done in Lemma \ref{lem:laplace_isometry}. Therefore we focus on proving the backward direction and assume that the equation \eqref{eq:change_of_coordinates_for_L} is valid. Due to \eqref{def:L} for any $w\in C^\infty_0(N_2)$ have that
    \begin{align*}
        \magneticnum{1}(w\circ \phi) &= -\Delta_1(w\circ \phi) - 2i\langle A_1,d_1(w\circ\phi)\rangle_1 + (id_1^*A_1 + |A_1|_1^2 + V_1)(w\circ \phi)
    \end{align*}
    and
    \begin{align*}
        [\magneticnum{2}(w)]\circ \phi &= [-\Delta_2(w)]\circ \phi - 2i\langle A_2,d_2w \rangle_2\circ \phi + [(id_2^*A_2 + |A_2|_2^2 + V_2)\circ\phi](w\circ \phi).
    \end{align*}
        First consider the function $w\in C^\infty_0(N_2)$ which is constant on some open bounded set $\phi(U) \subset N_2$. Then on $U$ we have the pointwise equality
    \begin{align}
id_1^*A_1 + |A_1|_1^2 + V_1 = \magneticnum{1}(w\circ \phi) &=  [\magneticnum{2}(w)]\circ \phi = (id_2^*A_2 + |A_2|_2^2 + V_2)\circ\phi. \nonumber
        \intertext{By assumption $A_j$ and $V_j$ are real-valued, so setting the real parts equal to each other and imaginary parts equal to each other, we have that}
        |A_1|_1^2 + V_1 = (|A_2|_2^2 + V_2)\circ \phi &\hspace{4mm}\text{and}\hspace{4mm}
        d_1^*A = (d_2^*A_2)\circ \phi. \label{eq:real_and_complex_part}
        \intertext{Since the equality (\ref{eq:real_and_complex_part}) holds pointwise on $N_1$, so for any $w\in C^\infty_0(N_2)$ we have that}
        -\Delta_1(w\circ \phi) - 2i\langle A_1, d_1(w\circ \phi)\rangle_1 &= [-\Delta_2w]\circ \phi - 2i\langle A_2,d_2w\rangle_2\circ \phi. \nonumber
        \intertext{Next take $w\in C_0^\infty(N_2)$ to be real valued, then since $A_j$ are real-valued covector fields we obtain}
        \Delta_1(w\circ \phi) = [\Delta_2(w)]\circ \phi&\hspace{4mm}\text{and}\hspace{4mm}
        \langle A_1, d_1(w\circ \phi)\rangle_1 = \langle A_2, d_2w\rangle_2\circ \phi. \label{eq:isometry}
    \end{align}
\begin{comment}
    We now briefly summarize what we have shown. If we have the commuting property of the Schr\"odinger operators and $\phi$, then for all $w\in C^\infty(N_2)$ we have the following pointwise equalities
    \begin{align}
        \Delta_1(w\circ \phi) &= [\Delta_2(w)]\circ \phi \label{isometry}\\
        \langle A, d_1(w\circ \phi)\rangle_1 &= \langle A_2, d_2w\rangle_2\circ \phi \label{AB}\\
        d_1^* A &= (d_2^*A_2)\circ \phi \nonumber\\
        |A|_1^2 + V_1 &= (|A_2|_2^2 + V_2)\circ \phi \label{0-order}
    \end{align}
\end{comment}
    By linearity of the Magnetic-Schr\"odinger operators, the equality $(\ref{eq:isometry})$ can be extended to all $w\in C_0^\infty(N_2)$. Then by Lemma \ref{lem:laplace_isometry}, equality (\ref{eq:isometry}) implies that $g_1 = \phi^*g_2$.
    
    We note that the equality $A_1 = \phi^*A_2$ can be established by working in local coordinates, using the second part of the equality $(\ref{eq:isometry})$ as well as $g_1 = \phi^*g_2$. In this proof we need to also incorporate the fact that the pullback and the exterior derivative commute \cite[Proposition 14.26]{lee_smooth}.
    Furthermore, the equality $A_1 = \phi^*A_2$ yields $|A_1|_1 = |A_2|_2 \circ\phi$. Hence, $V_1 = V_2\circ \phi = \phi^*V_2$ follows from the equality $(\ref{eq:real_and_complex_part})$.
 \begin{comment}   
    and that $B^{\#_2}_{\phi(p)} = d\phi_pA_p^{\#_1}$, then if $g_p$ and $h_{\phi(p)}$ are $g$ and $h$ evaluated at $p$ and $\phi(p)$ respectively, then by the isometry property we have
    \begin{align*}
        |A|_1^2(p) &= |A^{\#_1}|_1^2(p) = g_p(A^{\#_1}_p, A_p^{\#_1})\\
        &= h_{\phi(p)} (d\phi_pA_{p}^{\#_1}, d\phi_p A_p^{\#_1})\\
        &= h_{\phi(p)}(B_p^{\#_2}, B_p^{\#_2})\\
        &= |B^{\#_2}|_2^2(\phi(p)) = (|B|_2^2\circ\phi)(p)
    \end{align*}
    This shows that $|A|_1^2 = |B|_2^2\circ \phi$ pointwise, so by (\ref{0-order}) we must have that $V_1 = V_2\circ \phi$ as desired. 
\end{comment}
\end{proof}

\subsection{Definitions for the Cauchy Problem} \label{sub:definitions}

Let $C_0^\infty(N)$ be the set of compactly supported smooth complex valued functions with the topology induced by seminorms so that $C_0^\infty(N)$ is a well-defined topological vector space. We let $\mathcal{D}'(N)$ be the set of distributions on $C_0^\infty(N)$ with the weak-$*$ topology so that $\mathcal{D}'(N)$ is a well-defined topological vector space and the topological dual of $C_0^\infty(N)$.

As $\mathcal{D}'(N)$ is a well-defined topological space then $C([0,T];\mathcal{D}'(N))$ is a well-defined space of continuous functions from $[0,T]$ to $\mathcal{D}'(N)$. Explicitly we mean that if $u\in C([0,T];\mathcal{D}'(N))$ and if $\{t_n\} \subset [0,T]$ which converges to $t\in[0,T]$, then for any $\psi\in C_0^\infty(N)$ we have that $\langle u(t_n),\psi\rangle \to \langle u(t),\psi\rangle$ where $\langle u(t),\cdot\rangle$ is the distributional pairing on $(N,g)$. Moreover, if $u(t)\in L^2(N,g)$ then $\langle u(t),\psi\rangle = \int_N u(t,x)\overline{\psi(x)}dV_g$.

\begin{comment}
Since $\mathcal{D}'(N)$ is a topological vector space, we use the standard definition of differentiable at a point. 

Let $u:[0,T] \to \mathcal{D}'(N)$ be given so that the difference quotient 
$$
    \frac{u(t)-u(s)}{t-s} \in\mathcal{D}'(N)
$$
for all $t,s\in[0,T]$. We say that $u$ is differentiable at $t$ if there exists a distribution $u'(t) \in \mathcal{D}'(N)$ such that
for any $\{t_n\}\subset [0,T]$ which converges to $t$ we have that
$$
    u'(t) = \lim_{n\to\infty} \frac{u(t)-u(t_n)}{t-t_n}
$$
where the quotients $\{\frac{u(t)-u(t_n)}{t-t_n}\}$ converge in the topology on $\mathcal{D}'(N)$. Explicitly, we mean that for any $\psi \in C_0^\infty
(N)$ that
$$
    \langle u'(t),\psi\rangle =\lim_{n\to\infty} \Big\langle\frac{u(t) - u(t_n)}{t-t_n},\psi \Big\rangle.
$$
From this definition it is clear that if $u$ is a distribution which is differentiable at $t$ then for all $\psi\in C_0^\infty(N)$ we have that $t\to \langle u(t),\psi\rangle$ is a differentiable function and
$$
    \langle u'(t),\psi\rangle = \frac{d}{dt}\langle u(t),\psi\rangle.
$$
\end{comment}

If a distribution $u\in C([0,T];\mathcal{D}'(N))$ is differentiable at all $t\in [0,T]$ and 
\\
$u' \in C([0,T];\mathcal{D}'(N))$, we shall say that $u$ is a continuously differentiable distribution and denote this set by $C^1([0,T];\mathcal{D}'(N))$. We define the spaces $C^k([0,T];\mathcal{D}'(N))$ inductively for all $k\in \mathbb{N}$. 
% One can also replace $\mathcal{D}'(N)$ in the definition above with any topological vector space such as the Sobolev space $H^k(N,g)$. 
We shall also let $L^2(0,T;\mathcal{D}'(N))$ be the set of distributions such that $t\to \langle u(t),\psi(t,\cdot)\rangle$ is an element of $L^2([0,T])$ for all $\psi\in C^\infty_0([0,T]\times N)$.
\begin{defn}
    We say that a distribution $u \in C^{1}([0,T];\mathcal{D}'(N))$ is a distributional solution to the Cauchy Problem (\ref{eq:cauchy_problem}) with initial conditions $u_0,u_1\in\mathcal{D}'(N)$ and source $f\in L^2((0,T); \mathcal{D}'(N))$ if
    \begin{align}
        \int_{0}^T\langle u(t), (\partial_t^2+\mathcal{L}_{g,A,V})w(t,\cdot)\rangle dt&= \int_{0}^T\langle f(t),w(t,\cdot)\rangle dt +\langle u_0,\partial_tw(0,\cdot)\rangle-\langle u_1,w(0,\cdot)\rangle
        \label{eq:distributional_solution}
    \end{align}
    for all $w\in C^\infty_0([0,T)\times N)$. 
\end{defn}

\begin{rem}
    Note that we require $u\in C^1([0,T];\mathcal{D}'(N))$ so that $u(0) = u_0$ and $\partial_tu(0) = u_1$ are well-defined. Note that (\ref{eq:distributional_solution}) is satisfied for any smooth solution to (\ref{eq:cauchy_problem}).
\end{rem}

Next we recall the definition of Sobolev spaces on $(N,g)$ for all nonnegative integers. 

\begin{defn} \label{def:sobolev_def}
    Let $(N,g)$ be a Riemannian manifold and for any $k \in\{0,1,2,\ldots\}$.
    \begin{align*}
        \mathcal{H}^{k}(N,g) :&= \Big\{ u\in C^\infty(N)\Big|(\forall j\in \{0,1,\ldots,k\})\,\, \int_{N}   |\nabla^j_g u|^2_g dV_g < \infty \Big\}.
    \end{align*}
    Here $\nabla_g$ is the Levi-Civita connection. We then define an inner product on $\mathcal{H}^{k}(N,g)$ by
    \begin{align*}
        \langle u,v\rangle_k &= \sum_{j=0}^k \int_{N} \langle \nabla^ju,\overline{\nabla^j v}\rangle_gdV_g.
    \end{align*}
    % It follows that $(\mathcal{H}^{k}(N,g),\langle\cdot,\cdot\rangle_k)$ is a complex inner product space. 
    Finally, we define $H^{k}(N,g)$ to be the completion of $\mathcal{H}^{k}(N,g)$ with respect to $\norm{\cdot}_{H^k(N,g)}$. When $k = 0$ we let the space and norm to be denoted by $L^2(N,g)$ and $\norm{\cdot}_{L^2(N,g)}$ respectively.
\end{defn}

Note that the Sobolev space $H^k(N,g)$ depends on the choice of Riemannian metric $g$. In the case that $N$ is a closed manifold, the space of functions $H^k(N,g)$ is the same for all Riemannian metrics and norms induced by the Sobolev inner products are all equivalent. However, if $N$ is not compact, then the Sobolev spaces can be different depending on the choice of geometry.
% Uncomment for Dissertation
%For example, if $N =\R$ then we can consider the standard euclidean metric and the weighted metric $g = \frac{1}{1+x^2}$. In the case of the Euclidean metric the constant functions are not in Sobolev spaces, but in the latter case constant functions will be in the Sobolev spaces. 
For a discussion of Sobolev spaces on Riemannian manifolds see \cite{hebey_sobolev} or \cite{aubin2012nonlinear}.

\subsection{Closed Manifold Case}
\label{sub:closed_manifold}
We first prove that Cauchy problem (\ref{eq:cauchy_problem}) is wellposed in the case when $(N,g)$ is a closed manifold (i.e. a compact manifold with no boundary) using results from Spectral Theory. For this reason we require that $\mathcal{L}_{g,A,V}$ is a symmetric operator with respect to the smooth density $dV_g$.

% Uncomment for dissertation
\begin{comment}
\begin{lem}\label{lem:eigenfunctions}
    The Magnetic-Schr\"odinger operator $\mathcal{L}_{g,A,V}$ has eigenfunctions $\{\varphi_j\}\subset C^\infty(N)$ which form a complete orthonormal system for $L^2(N,g)$ and has real eigenvalues $\{\lambda_j\}$ such that $\mathcal{L}_{g,A,V}\varphi_j = \lambda_j\varphi_j$ for all $j$ and $\lambda_j\to\infty$.    
\end{lem}
\begin{proof}
    See \cite[Theorem 8.3]{shubin_psdo} or \cite[Section 5.1]{taylor_partial}. 
\end{proof}
Since the second order term of $\mathcal{L}_{g,A,V}$ is the elliptic Laplace-Beltrami operator $-\Delta_g$, \cite[Corollary 9.3]{shubin_psdo} implies the following result.
\begin{cor}
\label{cor:spectral_prop_of_L}
    There is a constant $D > 0$ so that for all $u\in C^\infty(N)$ we have that
    \begin{align*}
        (\mathcal{L}_{g,A,V}(u),u)_g \geq -D(u,u)_g
    \end{align*}
    In particular, $\mathcal{L}_{g,A,V}$ has only finitely many negative eigenvalues, and all the eigenvalues $\{\lambda_j\}$ can be sorted in an ascending order so that
    $$
    - D \leq \lambda_1 \leq\ldots \leq \lambda_j\to \infty
    $$
\end{cor}
\begin{proof}
    See \cite[Corollary 9.3]{shubin_psdo}.
\end{proof}
By making $D > 0$ larger if needed, we may assume that $\mathcal{L}_{g,A,V} + D$ is symmetric positive definite operator on $C^\infty
(N)$. From the smoothness of the eigenfunctions, we can now give a spectral definition of Sobolev spaces for all integer values $k\in\Z$.
\end{comment}

\begin{defn}\label{def:sobolev_spectral_def}
    Let $\magnetic$ be a smooth Magnetic-Schr\"odinger operator on the closed manifold $(N,g)$, let $\{(\lambda_j,\varphi_j)\}$ be the $L^2(N,g)$ orthonormal eigenfunctions and eigenvalues associated with $\mathcal{L}_{g,A,V}$ on $N$, let $D > 0$ be such that $\magnetic + D$ is positive definite. For any $k\in\Z$ we let 
    \begin{align}
        H^k(N) :&= \Big\{\sum_jc_j\varphi_j :\sum_j |c_j|^2|\lambda_j + D|^k < \infty\Big\}.
        \intertext{We define a norm on $H^k(N)$ by}
            \norm{\sum_j c_j\varphi_j}_k^2 :&= \sum_j |c_j|^2|\lambda_j + D|^k  \hspace{4mm}\text{for all}\hspace{4mm} \sum_jc_j\varphi_j\in H^k(N)\nonumber
    \end{align}
\end{defn}

The existence of $L^2(N,g)$ orthonormal eigenfunctions and eigenvalues $\{(\lambda_j,\varphi_j)\}$ associated with $\mathcal{L}_{g,A,V}$ on $N$ and such a constant $D > 0$ rely on the compactness of $N$ and the proofs can be found in \cite[Theorem 8.3]{shubin_psdo}, \cite[Section 5.1]{taylor_partial} and \cite[Corollary 9.3]{shubin_psdo}.

At this point $H^k(N)$ is a set of functions with a topology, but we have not established that $H^k(N)$ is a Hilbert space or what the relationship between $H^k(N)$ and the Sobolev space $H^k(N,g)$ as given in Definition \ref{def:sobolev_def} is. One can use the following Proposition to prove that $H^k(N)$ is a Hilbert space for any $k\in\Z$.

\begin{prop}
\label{prop:equiv_of_Sob_spaces}
     Let $\magnetic$ be a smooth Magnetic-Schr\"odinger operator on the closed manifold $(N,g)$, let $\{(\lambda_j,\varphi_j)\}$ be the $L^2(N,g)$ orthonormal eigenfunctions and eigenvalues associated with $\mathcal{L}_{g,A,V}$ on $N$, let $D > 0$ be such that $\magnetic + D$ is positive definite. Let $\{c_j\} \subset \C$ and $k \in \mathbb{Z}$. The following are equivalent
    \begin{enumerate}
        \item The series $\sum_j |c_j|^2|\lambda_j + D|^{k}$
        converges
        \item The series $\sum_j c_j\varphi_j$ converges in the $H^{k}(N)-$topology
    \end{enumerate}
    
The following are also equivalent
    \begin{enumerate}
        \item The series converges $\sum_j |c_j|^2|\lambda_j+D|^{k}$
        for every $k \in \Z$
        \item The series $\sum_j c_j\varphi_j$ converges in the $C^\infty(N)-$topology
        % \item $\sum_{j} c_j\varphi_j \in C^\infty(N)$
    \end{enumerate}
\end{prop}

\begin{proof}
    See \cite[Proposition 10.2]{shubin_psdo}.
    % Uncomment for dissertation
  %  We summarize the proof in Theorem \ref{thm:sobolev_space_categorization}.
\end{proof}

Since the Sobolev spaces in Definition \ref{def:sobolev_def} are only defined for positive integers $k\in\N$, so we can only compare $H^k(N)$ and $H^k(N,g)$ when $k\in\N$. 

\begin{lem} \label{lem:space_spectral_def}
    Let $\magnetic$ be a smooth Magnetic-Schr\"odinger operator on the closed manifold $(N,g)$, let $\{(\lambda_j,\varphi_j)\}$ be the $L^2(N,g)$ orthonormal eigenfunctions and eigenvalues associated with $\mathcal{L}_{g,A,V}$ on $N$, let $D > 0$ be such that $\magnetic + D$ is positive definite. Then for all $k\in\Z$ we have that
    \begin{align*}
        \norm{w}_{k}^2 &= ((\mathcal{L}_{g,A,V} + D\cdot \text{Id})^k w,w)_{L^2(N,g)} \hspace{4mm}\text{ for all $w\in C^\infty(N)$.}
    \end{align*}
    If $k\in\mathbb{N}$, then $H^k(N) = H^k(N,g)$ as Hilbert spaces and the norms $\norm{\cdot}_k$ and $\norm{\cdot}_{H^k(N,g)}$ are equivalent.
\end{lem}

\begin{proof}
    See the proof of \cite[Proposition 10.2]{shubin_psdo}. 
\end{proof}

\begin{rem} If $(N,g)$ is a closed manifold and $k \in \mathbb{Z}$, then the topological dual to $H^k(N)$ is isomorphic to $H^{-k}(N)$. See \cite[Theorem 7.7]{shubin_psdo} or \cite[Proposition 3.2]{taylor_partial}.
\end{rem}

We will also need to introduce Sobolev Spaces which depend on time.

% Uncomment for dissertation
\begin{comment}
We also give a summary of the results in Subsection \ref{sec:time_dependent_sobolev_spaces} and cover their definitions. In particular see Theorem \ref{thm:time_sobolev_spectral} for a proof that the following space is equivalent to the classical definition for a time dependent Sobolev space given in \cite{evans2010partial}.
\end{comment}

\begin{defn} \label{def:time_sobolev_spectral_space}
    For any $T > 0$, assume that $f(t,\cdot)\in L^2(N)$ for a.e. $t\in [0,T]$ and let $c_j(t) = (f(t,\cdot),\varphi_j)_g$ be the $j^{th}$ Fourier coefficients. For any $k \in \Z$ and $l\in\{0,1,2,\ldots\}$ we say that $ f \in H^l(0,T;H^{k}(N))$ provided that
    $$
        \sum_j \norm{c_j}_{H^l(0,T)}^2(\lambda_j+D)^{k} <\infty
    $$
    Then $H^{l}(0,T;H^{k}(N))$ forms a Hilbert space with norm
    $$
        \norm{f}_{H^l(0,T;H^{k}(N))}^2 := \sum_j\norm{c_j}_{H^l(0,T)}^2(\lambda_j+D)^{k}    $$
\end{defn}
In addition to characterization of Sobolev spaces above, we will also need Sobolev Inequalities which are described by Kondrakov's Theorem. Various versions of the following standard result and their proof can be found in \cite[Theorem 7.6]{shubin_psdo},  \cite[Theorem 2.34]{aubin2012nonlinear}, \cite[Theorem 2.34]{aubin_nonlinear} and \cite[Proposition 3.3]{taylor_partial}.
\begin{lem} \label{sobolev_inequality}
    Let $(N,g)$ be a compact Riemannian manifold with or without smooth boundary of dimension $n$. If $k > \frac{n}{2} + m$, then we have the compact embedding $H^{k}(N)\subset C^m(N)$.    
\end{lem}

We also have the following embedding Theorem for $H^l(0,T;H^k(N))$.

\begin{lem}[Calculus in an abstract space]\label{lem:calculus_abstract_space}
    Let $(N,g)$ be a closed manifold. Let $u\in H^l(0,T;H^{k}(N))$ for $l \in \{1,2,\ldots\}$, then
    $$
        u \in \bigcap_{r=0}^{l-1} C^r([0,T];H^{k-r}(N))
    $$
    and there exists a constant $C > 0$ which depends on $T, N, g$ such that 
    $$
        \max_{0\leq r\leq l-1}\max_{0\leq t \leq T}\norm{ \partial_t^{(r)}u(t,\cdot)}_{H^{k}(N,g)} \leq C \norm{u}_{H^l(0,T;H^{k}(N,g))}
    $$
    In particular, if $k > \frac{n}{2} + m$ and $l > m$ for $m\in \{0,1,2,\ldots\}$ then there exists a constant $C > 0$ such that
    $$
        \norm{u}_{C^m([0,T]\times N)} \leq C \norm{u}_{H^l(0,T;H^k(N,g))}
    $$
\end{lem}

\begin{proof}
    See \cite[Theorem 2 of Section 5.9]{evans2010partial}.
\end{proof}

The results listed above give us sufficient tools to show the existence and uniqueness of solutions to the Cauchy problem $(\ref{eq:cauchy_problem})$.

\begin{thm}[Distributional Solutions on Closed Manifolds]\label{thm:distributional_solution}
    Let $(N,g)$ be a smooth closed Riemannian manifold, let $A$ be a smooth real-valued covector field, let $V$ be a smooth real-valued function and let $T > 0$. Suppose that $k\in\Z$    and $l\in \{0,1,2,\ldots\}$. 
    
    If
    $$ 
    u_0\in H^{k+1}(N),  \hspace{4mm}u_1\in H^{k}(N) \hspace{4mm} \hspace{2mm}
    \text{and} \hspace{4mm}f\in H^l(0,T;H^{k}(N))
    $$
    then there exists a unique 
    $$
         u\in \bigcap_{r=0}^{l+2} H^r(0,T;H^{k+1-r}(N)) \subset C^1([0,T];H^{k-2}(N)) \subset C^1([0,T];\mathcal{D}'(N)).
    $$
    which is a distributional solution to the Cauchy Problem  (\ref{eq:cauchy_problem}). 
    
    Furthermore, for each $r\in \{0,1,\ldots,l+2\}$
    there exists a constant $C_r >0$, that depends on $k$ and $l$, such that
    \begin{align}
    \label{eq:energy_estimate_for_u}
    %     \norm{u}_{L^2(0,T;H^{k + 1}(N))} &\leq C_0\Big(\norm{f}_{L^2(0,T;H^{k}(N))} + \norm{u_0}_{H^{k+1}(N)} + \norm{u_1}_{H^{k}(N)}\Big)
     %    \\
     %    \norm{u}_{H^{1}(0,T;H^{k}(N))} &\leq C_1\Big(\norm{f}_{L^2(0,T;H^{k}(N))} + \norm{u_0}_{H^{k+1}(N)} + \norm{u_1}_{H^{k}(N)}\Big)
      %   \\
        \norm{u}_{H^{r}(0,T;H^{k + 1 -r}(N))} &\leq C_r\Big(\norm{f}_{H^l(0,T;H^{k}(N))} + \norm{u_0}_{H^{k+1}(N)} + \norm{u_1}_{H^{k}(N)}\Big).
    \end{align}
    
    Finally, if $f\in C^\infty_0((0,T)\times N)$ and $u_0,u_1\in C^\infty(N)$, then there exists a unique smooth solution to the Cauchy Problem $(\ref{eq:cauchy_problem})$.
\end{thm}
\begin{proof}
The proof of this theorem is a standard energy estimate argument based on Fourier expansion. The details of the proof are omitted here, since similar proofs can be found in many textbooks including \cite[Section 2.3.]{lassas_inverse} and \cite[Chapter 6]{taylor_partial}. 
\end{proof}
It is important to note that in the previous theorem, we permit the functions to be in $H^k(N)$ where $k$ is a negative integer, that is we permit solutions which are integrable in time and are distributions in $N$.

\subsection{Complete Manifold Case} \label{sub:complete_manifold}
In this section we prove existence and uniqueness of solutions to the Cauchy problem \eqref{eq:cauchy_problem} on complete Riemannian manifolds with infinite time $T=\infty$, under the assumption that the data $u_0,u_1,f$ consist of smooth compactly supported functions.
% (i.e. manifolds with no boundary and complete as a metric space) which do not need to be compact. 
% To this end, we will construct a closed manifold which has a region which is has a Riemannian isometry to a region in $(N,g)$. We will then use finite speed of wave propagation to show that the solutions will be same in a certain region, this will allow us to construct solutions on a complete manifold for finite time.
The first tool we mention is the classical domain of dependency result for hyperbolic problems.  

For any point $p\in N$, we define the following cone
$$
    C(p,T) = \{(t,q)\in [0,T]\times N: d_g(p,q) \leq T-t\}.
$$

\begin{lem}[Domain of Dependence]\label{lem:domain_of_dependence}
        Suppose that $u\in C^\infty([0,T]\times N)$ is a solution to the Cauchy problem $(\ref{eq:cauchy_problem})$ where $u_0|_{B(p,T)} = u_1|_{B(p,T)} = 0$ and $f|_{C(p,T)} = 0$, then $u|_{C(p,T)} = 0$.
\end{lem}

\begin{proof}
    For a proof of the result for local coordinates see \cite[Section 7.2, Theorem 8]{evans2010partial} and for the case of a manifold see \cite[Chapter 2, Theorem 6.1]{taylor_partial}.
\end{proof}

Next we show that waves travel at a finite speed described by the Riemannian metric $g$. Let $\cX \subset N$ be a set and define the following \textit{domain of influence}
\begin{align}
    M(\cX,T) := \{y\in N : \text{dist}_g(y,\cX)\leq T\}. \label{def:domain_of_influence}
\end{align}
Then
$$
    \{y\in N: \text{dist}_g(y,\cX)< T\} \subset \text{int}(M(\cX,T)) \subset M(\cX,T).
$$
where $\text{int}$ indicates the interior of a set.
\begin{cor}[Finite Speed of Wave Propagation] \label{cor:finite_speed}
    Let $\cX \subset N$ be an open set and suppose that $u\in C^\infty([0,T]\times N)$ is a classical solution to the Cauchy problem where $u_0,u_1\in C_0^\infty(\cX)$ and $f\in C_0^\infty((0,T)\times \cX)$. Then $\text{supp}(u(t,\cdot)) \subset \text{int}(M(\cX, t))$ for all $t\in [0,T]$.
\end{cor}

\begin{proof}
    By compactness we may find a compact set $K \subset \cX$ such that $\text{supp}(u_i) \subset K$ and $\text{supp}(f)\subset [0,T]\times K$. From Lemma \ref{lem:domain_of_dependence} it can be shown that $\text{supp}(u(t,\cdot)) \subset M(K,t)$. Since $\cX$ is open and $K\subset \cX$ is compact, then $\text{dist}_g(K,\partial\cX) > 0$, it can then be shown that $M(K,t) \subset \{y\in N: \text{dist}_g(y,\cX) < t\} \subset \text{int}(M(\cX,t))$.
\end{proof}

% With our knowledge of finite speed of propagation for smooth solutions, we can easily show uniqueness of smooth solutions on a complete Riemannian manifold.

% We recall that finite speed of wave propagation implies immediately the uniqueness of smooth solutions to the Cauchy problem \eqref{eq:cauchy_problem}.

\begin{comment}   
\begin{cor}[Uniqueness of Smooth Solutions]\label{cor:unique_solution}
    Suppose that $u,v\in C^\infty([0,T]\times N)$ solve the Cauchy problem $(\ref{eq:cauchy_problem})$ with $u_0,u_1\in C^\infty(N)$ and $f\in C_0^\infty((0,T)\times N)$, then $u = v$.
\end{cor}
\end{comment}

% \begin{proof}
%     Let $b = u-v$ so that $b\in C^\infty([0,T]\times N)$ and $b$ is a solution to
%     \begin{align*}
%     \begin{cases}
%         (\partial_t^2 + \mathcal{L}_{g,A,V})b = 0, \quad \text{ in } (0,T) \times N
%         \\
%         b(0,\cdot) = \partial_tb(0,\cdot) = 0
%     \end{cases}
% \end{align*}
% Then $b = 0$ by Lemma \ref{lem:domain_of_dependence}. Hence $u = v$.
% \end{proof}

We say that $M \subset N$ is a \textit{regular domain} of $N$ if it is a bounded open set with smooth boundary. Let $K \subset N$ be compact.
We move towards establishing the existence and uniqueness of smooth solutions to the initial value problem \eqref{eq:cauchy_problem} in the case when $(N,g)$ is a complete manifold, $T=\infty$, the initial conditions $u_0,u_1$ are supported in $K$, and the forcing function $f$ is supported in $(0,T)\times K$. To do this, we introduce an iterative process with a sequence of nested regular domains $\{M_{j}\}_{j=1}^\infty$ of $N$, such that $K$ is contained in $M_1$, and these sets exhaust $N$. We isometrically embed each domain $M_j$ into a closed manifold $\tilde{N}_j$ and push forward $u_0,u_1,f$ and the Magnetic-Schr\"odinger operator $\magnetic$ onto $\tilde{N}_j$. We use Theorem \ref{thm:distributional_solution} to solve the respective version of the Cauchy problem \eqref{eq:cauchy_problem} in $(0,T_j)\times \tilde{N}_j$ where $\{T_j\}_{j=1}^\infty$ is an increasing unbounded positive sequence. Finally, we pullback this family of  solutions to $(0,\infty) \times N$ and use a partition of unity argument to construct the unique smooth solution to the original Cauchy problem. 
% We will construct infinitely many of these closed manifolds and use the uniqueness of smooth solutions to prove that the solutions agree on their common domain.

% The special closed manifold that we will construct will be the double of a smooth manifold with boundary. At a conceptual level, if one has a smooth manifold $M$ with nonempty smooth boundary, then its double $D(M)$ is a smooth manifold without boundary that is formed by taking two copies of $M$ and gluing them along their boundary. In the case when $M$ is a compact manifold with smooth boundary, then $D(M)$ will be a smooth closed manifold. The connection between $M$ and $D(M)$ is summarized by the following Proposition.

Our construction relies on the following well known facts from Riemannian geometry.

\begin{lem}\label{lem:double_riemannian}
Let $(M,g)$ be a smooth Riemannian manifold with boundary and an open set $M_1 \subset M$ with smooth boundary such that $\p M_1 \cap \p M=\emptyset$. There exists a closed Riemannian manifold $(\tilde{N},\tilde{g})$ and a smooth function $\phi:M \to \tilde{N}$ such that
$$
    \phi:(M_1,g) \to (\phi(M_1),\tilde{g}) \subset (\tilde{N},\tilde{g})
$$
is a Riemannian isometry.
\end{lem}

\begin{proof}
    See \cite[Theorem 9.29]{lee_smooth} and \cite[Example 9.32]{lee_smooth}. 
    % Then construct $\tilde{g}$ on $\tilde{M}_1$. 
\end{proof}

\begin{lem}\label{lem:compact_exhaustion}
    Let $K \subset N$ be compact and let $N$ be a smooth manifold without boundary. There exists a sequence $\{M_j\}$ of compact regular domains of $N$ such that
    \begin{enumerate}
        \item $K\subset M_1 $
        \item $M_{j}\subset M_{j+1}$
        \item $\bigcup_{j\in\N} M_j = N$.
    \end{enumerate}
\end{lem}

\begin{proof}
    See \cite[Proposition 2.28]{lee_smooth}, \cite[Theorem 6.10]{lee_smooth} and \cite[Proposition 5.47]{lee_smooth}.
\begin{comment}    
    All of the Propositions and Theorems mentioned in this proof and their numbers can be found with their proof on \cite{lee_smooth}. Since $N$ is smooth, then there exists a smooth positive function $f: N\to \R$ called an exhaustion function which has the property that $f^{-1}((-\infty,c]) \subset N$ is a compact set for all $c\in \R$. See \cite[Proposition 2.28]{lee_smooth} for a proof of the existence of such a function.\\ \newline 
    Since $K$ is compact, there exists a constant $b$ such that $\max_{x\in K}f (x)\leq b$. Then $K \subset f^{-1}((-\infty,b])$. If we let $C \subset \R$ be the set of critical values of $f$, then by Sard's Theorem (see \cite[Theorem 6.10]{lee_smooth}), we have that $C$ has measure 0. Thus we may find an increasing sequence $\{b_j\} \subset (b,\infty)\cap (\R\backslash C)$ such that $b_j\to \infty$ and each $b_j$ is a regular value of $f$.\\ \newline
    Since $b_j$ is a regular value of $f$, then the set $M_j := f^{-1}((-\infty,b_j])$ is a regular domain (see \cite[Proposition 5.47]{lee_smooth}) and $M_j \subset M_{j+1}$ since $b_j \leq b_{j+1}$. Since $f$ is an exhaustion function, then $M_j$ is also a compact set and $\bigcup_{j\in N}M_j  = f^{-1}(\R) = N$.
\end{comment}
\end{proof}

We now construct a Magnetic-Schr\"odinger operator $\mathcal{L}_{\tilde{g},\tilde{A},\tilde{V}}$ on $\tilde{N}$ using the change of coordinates from Lemma \ref{lem:double_riemannian} and the original operator $\magnetic$ on $N$.
\begin{comment}
For a given co-vector field $A$ and function $V$ in $N$ and a regular domain $M \subset N$ we proceed to construct a Magnetic-Schr\"odinger operator $\mathcal{L}_{\tilde{g},\tilde{A},\tilde{V}}$ on $\tilde{N}$ which commutes with the Magnetic-Schr\"odinger operator $\mathcal{L}_{g,A,V}$, of the original manifold $N$, over the isometry $\phi$, as in Lemma \ref{lem:double_riemannian}, in an appropriate set of smooth functions.
\end{comment}

\begin{prop}\label{prop:magnetic_extension}
    Let $M \subset N$ be a regular domain on a complete Riemannian manifold $(N,g)$. Let $\mathcal{L}_{g,A,V}$ be a Magnetic-Schr\"odinger operator, see \eqref{def:L}, on $N$ with a real valued co-vector field $A$ and function $V$ on $N$. There exists a closed manifold $(\tilde{N},\tilde{g})$ with a regular domain $\tilde{M}$, an isometry $\phi:(M,g)\to (\tilde{M},\tilde{g})$, and a Magnetic-Schr\"odinger operator $\mathcal{L}_{\tilde{g},\tilde{A},\tilde{V}}$ on $\tilde{N}$ such that 
    $$
        \mathcal{L}_{g,A,V}(v\circ \phi) = [\mathcal{L}_{\tilde{g},\tilde{A},\tilde{V}}(v)]\circ\phi,
        \quad \text{for all $v\in C^\infty_0(\tilde{M})$.}
        % \hspace{4mm}\text{on the interior of $\tilde{M}$.}
    $$
\end{prop}

\begin{proof}
The claim follows immediately from Lemma \ref{lem:double_riemannian} and Lemma \ref{lem:magnetic_isometry}.
%
% First we choose a regular domain $M_2$ such that $\bar M \subset M_1$ while $\p M \cap \p M_1 =\emptyset$. Let $\tilde{N} = D(M_1)$ and give $\tilde{g}$ a Riemannian metric in $D(M_1)$ such that $(M,g)$ and $(\tilde{M},\tilde{g})$ are isometric with isometry $\phi: M \to \tilde{M}$ as in Lemma \ref{lem:double_riemannian}. Since $M$ is compact, then $\tilde{N}$ is a closed manifold. Let $\psi:\tilde{N}\to[0,1]$ be a bump function which is equal to 1 on $\tilde{M}$ and equal to $0$ outside of $M_2$. The isometry $\phi$ can be chosen so that $\phi: M_2\to \tilde{M_2}$ is a diffeomorphism. Now define $\tilde{A} = \psi (\phi^{-1})^*A$ and $\tilde{V} = \psi(\phi^{-1})^*V$. Then on $M$ we have that $g = \phi^*\tilde{g}, A = \phi^*\tilde{A}$ and $V = \phi^*\tilde{V}$. By Lemma \ref{lem:magnetic_isometry}, we get the desired relationship. 
\end{proof}

% Proposition \ref{prop:magnetic_extension} will be used to lift the solutions from Theorem \ref{thm:closed_existence} from a closed manifold to a complete manifold with compactly supported data.

With these tools we are ready to prove the main result of this subsection. 

\begin{thm}[Existence and Uniqueness of Smooth Solutions]
\label{thm:uniqueness_of_smooth_solutions}    
Let $(N,g)$ be a complete, connected Riemannian manifold, let $A$ be a smooth real covector field let $V$ be a smooth real-valued function on $N$. If $u_0,u_1\in C_0^\infty(N)$ and $f\in C_0^\infty((0,\infty)\times N)$ then the Cauchy problem \eqref{eq:cauchy_problem}, with $T=\infty$, has a unique smooth solution.
%
%     ((0,\infty)\times N)$. There exists a unique classical solution $u\in C^\infty([0,\infty)\times N)$ to the Cauchy Problem
%     \begin{align*} 
%     \begin{cases}
%         (\partial_t^2 + \mathcal{L}_{g,A,V})u = f, \quad \text{ in } (0,\infty) \times N
%         \\
%         u(0,\cdot
%         ) = u_0, \,\partial_tu(0,\cdot) = u_1
%     \end{cases}
% \end{align*}
\end{thm}

\begin{rem}
We believe that Theorem \ref{thm:uniqueness_of_smooth_solutions} and the presented proof are well known. However, for the convenience of the reader we provide the full proof in here.    
\end{rem}

\begin{proof}[Proof of Theorem \ref{thm:uniqueness_of_smooth_solutions}]
Since the data $u_0,u_1,$ and $f$ are compactly supported there is $T \in (0,\infty)$ and a compact set $K \subset N$ such that $\text{supp}(f) \subset (0,T)\times K$ and $\text{supp}(u_0), \: \text{supp}(u_1) \subset K$. Since $K$ is compact the domain of influence $M(K,T)$, see \eqref{def:domain_of_influence}, is also compact. By Lemma \ref{lem:compact_exhaustion} there exists a compact exhaustion of $N$ by regular domains $\{M_j\}_{j=1}^\infty$ such that $M(K,T)\subset M_1$ and $M_j\subset M_{j+1}$. Then we choose an increasing sequence $T_j > T$, $\lim_{j \to \infty} T_j=\infty$, such that $M_{j-1} \subsetneq  M(K,T_j) \subsetneq M_j$. Moreover, due to the connectedness of $N$ the sequence of domains of influences $\{M(K,T_j)\}_{j=1}^\infty$ exhausts $N$.

Due to Lemma \ref{lem:double_riemannian} and Proposition \ref{prop:magnetic_extension} there exists a sequence of closed Riemannian manifolds $\{(\tilde{N}_j,\tilde{g}_j)\}_{j=1}^\infty$, regular domains $\tilde{M}_j \subset \tilde N_j$, Riemannian isometries $\phi_j:(M_j,g)\to (\tilde{M}_j,\tilde{g}_j) \subset (\tilde{N}_j,\tilde{g}_j)$ and Magnetic-Schr\"odinger operators $\mathcal{L}_j$ on $\tilde{N}_j$ such that for all smooth functions $\tilde{w}$ compactly supported in $\tilde{M}_j$ we have that 
    \begin{equation}
    \label{eq:change_of_coord_for_mag_schro}
        \mathcal{L}_{g,A,V}(\tilde{w}\circ\phi_j) = [\mathcal{L}_j(\tilde{w})]\circ\phi_j.
        % \hspace{4mm}\text{on the interior of $\tilde{M}_j$.}
    \end{equation}

    We now seek to construct solutions $\tilde{u}_j \in C_0^\infty([0,T_j]\times \tilde{N_j})$ that solve an appropriate Cauchy problem and pullback the solutions to $(0,\infty)\times N$. For each $j\in \mathbb{N}$ we have the Riemannian isometry $\phi_j:M_j\to \tilde{M}_j \subset \tilde{N}_j$, let $\tilde{\phi}_j(t,x) = (t,\phi_j(x))$ for all $x\in M_j$. We define the smooth maps
    \begin{align*}
        \tilde{f}_j(t, \tilde{x}) &= \begin{cases}  (f\circ\tilde{\phi}^{-1}_j)(t,\tilde{x}) \hspace{4mm}\tilde{x}\in\phi_j(K)\\
            0 \hspace{27mm}\text{otherwise},
        \end{cases}
%        \tilde{w}_{0,j}(\tilde{x}) &= \begin{cases}
 %           (u_0\circ\phi^{-1}_j)(\tilde{x}) \hspace{5.5mm}\tilde{x}\in\phi_j(K)\\ 
  %          0 \hspace{27.5mm}\text{otherwise},
   %     \end{cases}\\
    %    \tilde{w}_{1,j}(\tilde{x}) &= \begin{cases}
     %       (u_1\circ\phi^{-1}_j)(\tilde{x})\hspace{5.5mm}\tilde{x}\in\phi_j(K)\\
      %      0  \hspace{27.5mm}\text{otherwise}.      
       % \end{cases} 
    \end{align*}
    and define the functions $\tilde{w}_{0,j}$ and $\tilde{w}_{1,j}$ analogously.
    % By the support conditions of $u_0,u_1$ and $f$, it follows that $\tilde{w}_{0,j},\tilde{w}_{1,j}\in C_0^\infty(\tilde{N}_j)$ and $\tilde{f}_j\in C_0^\infty((0,T_j)\times \tilde{N}_j)$. 
    By Theorem \ref{thm:distributional_solution} there exists a unique solution $\tilde{u}_j\in C^\infty([0,T_j]\times\tilde{N}_j)$ to the Cauchy problem
    \begin{align}
    \label{eq:wave_eq_on_N_j}
        \begin{cases}
            (\partial_t^2 + \mathcal{L}_j) \tilde{u}_j = \tilde{f}_j\hspace{6mm}\text{ in }(0,T_j)\times\tilde{N}_j\\
            \tilde{u}(0,\cdot) = \tilde{w}_{0,j}, \,\,\,\partial_t\tilde{u}(0,\cdot) = \tilde{w}_{1,j}.
        \end{cases}
    \end{align}
    Next we define a function $u_j$ in $(0,\infty)\times N$ according to the rule
    \begin{align*}
        u_j(t,x) &= \begin{cases}
            (\tilde{u}_j\circ\tilde{\phi}_j)(t,x) \hspace{10mm} (t,x)\in[0,T_j]\times M_j\\
            0 \hspace{33mm}\text{otherwise}.
        \end{cases}
    \end{align*} 
    Since $\phi_j:(M_j,g)\to (\tilde{M}_j,\tilde{g}_j)$ is a Riemannian isometry then $\phi_j(M(K,T_j)) = \tilde{M}_j(\phi_j(K),T_j)$ where $\tilde{M}_j(\phi_j(K),T_j) = \{y\in \tilde{N}_j: d_{\tilde{g}_j}(y,\phi_j(K))\leq T_j\}$. By Finite Speed of Wave Propagation (i.e. Corollary \ref{cor:finite_speed}) we know that $\text{supp}(\tilde{u}_j) \subset [0,T_j]\times \tilde{M}_j(\phi_j(K),T_j)$.
    Thus $u_j\in C^\infty([0,T_j]\times N)$. 
    Further, 
    % by the construction of the Magnetic-Schr\"odinger operator $\mathcal{L}_j$, we have by Proposition \ref{prop:magnetic_extension} 
    we have by \eqref{eq:change_of_coord_for_mag_schro}, \eqref{eq:wave_eq_on_N_j}, and the support conditions of the data $u_0,u_1,f$ that 
    % $$
    %     (\partial_t^2 + \mathcal{L}_{g,A,V})u_j
    %     % = (\partial_t^2 + \mathcal{L}_{g,A,V})(\tilde{u}_j\circ\tilde{\phi}_j) = [(\partial_t^2 +\mathcal{L}_j)\tilde{u}_j]\circ \tilde{\phi}_j = \tilde{f}_j \circ \tilde{\phi}_j 
    %     = f, \quad \text{in } (0,T_j)\times N.
    % $$
    %
    % % Since $\text{supp}(f)\subset (0,T)\times K \subsetneq [0,T_j]\times M_j$ it follows that $(\partial_t^2+\mathcal{L}_{g,A,V})u_j = f$ on $(0,T_j)\times N$. 
    % We can similarly argue by the supports of the functions and the initial conditions $u_0$ and $u_1$ that $u_j(0,\cdot) = u_0$ and $\partial_tu_j(0,\cdot) = u_1$. Then for all $j\in \mathbb{N}$ we have that 
    $u_j\in C^\infty([0,T_j]\times N)$ is the unique solution to the Cauchy problem
    \begin{align*}
    \begin{cases}
        (\partial_t^2 + \mathcal{L}_{g,A,V})v = f, \quad \text{ in } (0,T_j) \times N
        \\
        v(0,\cdot) = u_0, \,\partial_t v(0,\cdot) = u_1.
    \end{cases}
\end{align*}

Therefore for each $k>j$ we use standard uniqueness arguments for linear equations with Lemma \ref{lem:domain_of_dependence} that $u_k|_{[0,T_j]\times N} = u_j$. 
Hence, we can conclude the proof by noting that the function
\[
u(t,x) := \inf_{j \in \N}\{u_j(t,x): j\in\mathbb{N}\text{ such that }t < T_j\}, \quad \text{ for all } (t,x)\in [0,\infty) \times N,
\]
is a well defined smooth solution to the Cauchy problem \eqref{eq:cauchy_problem}. 
% By uniqueness we have that $\{u_j(t,x):j\in\mathbb{N}\text{ such that }t < T_j\}$ is a set of a single number so $u(t,x)$ is a well-defined number. 
% For any $t_0\in (0,\infty)$ there exists a neighborhood $(t_0-\eps,t_0+\eps)$ for some $\eps > 0$ and an index $j$ such that $u(t,x) = u_j(t,x)$. It follows that $u\in C^\infty([0,\infty)\times N)$ and that $u$ satisfies the Cauchy problem on $(0,\infty)\times N$ as desired.
\end{proof}
\begin{rem}
    By choosing $u_0 = u_1 = 0$ in Theorem \ref{thm:uniqueness_of_smooth_solutions} we see that the Cauchy Problem \eqref{eq:cauchy_problem_infinite_time} has a unique smooth solution for each source $f\in C_0^\infty((0,\infty)\times \cX)$. Hence, the local source-to-solution operator $\operator{}$ is well-defined.
\end{rem}

\subsection{Necessary Tools for the Inverse Problem} \label{sub:approx_controllability}
The goal of this subsection is to highlight some important properties for solutions to the Cauchy Problem \eqref{eq:cauchy_problem_infinite_time}. These properties will be used in the following sections to solve the inverse problem.
% including density in $H_0^k(M(\cX,T),g)$ and to connect $\operator{}$ with the solutions outside of $\cX$ using a Blagovestchenskii identity.
% To verify these results we use the distributional solutions developed in Subsection \ref{sub:closed_manifold}.

The symmetry of the Magnetic-Schr\"odinger operator $\magnetic$ is needed to prove the following Blagovestchenskii identity which was originally derived in \cite{blagovestchenskii1969one, blagoveshchenskii1971inverse}.
\begin{thm}[Blagovestchenskii Identity] \label{thm:blagovestchenskii}
    Let $(N,g)$ be a complete Riemannian manifold. Let $T > 0, \cX\subset N $ be open and bounded. Let $f,h\in C_0^\infty((0,2T)\times\cX)$, then
    $$ (u^f(T,\cdot), u^h(T,\cdot))_g = (f, J\Lambda_{g,A,V}(h))_{L^2((0,T)\times\cX)} - (\Lambda_{g,A,V}(f), J(h))_{L^2((0,T)\times\cX)}$$
    where the operator $J:L^2(0,2T)\to L^2(0,T)$ is defined as
    $$
        J\phi(t) := \int_t^{2T-t}\phi(s)ds
    $$
\end{thm}

\begin{proof}
    See \cite[Theorem 12]{source-to-solution} for the  proof.
\end{proof}

We now focus on proving that solutions to the Cauchy Problem (\ref{eq:cauchy_problem_infinite_time}) are dense in an appropriate sense in $H_0^k(M(\cX,T),g)$ for all $k\in \{0,1,\ldots\}$. As part of this we shall need to regularize distributional solutions. 

\begin{prop}[Mollified Solution] \label{prop:mollified_solution}
    Let $(N,g)$ be a closed  Riemannian manifold. Let $u\in C^1([0,T];\mathcal{D}'(N))$ be the unique distributional solution to the Cauchy Problem
    \begin{align}
    \label{eq:wave_eq_without_source}
        \begin{cases}
            (\partial_t^2+\mathcal{L}_{g,A,V})u = 0 \hspace{4mm}\text{on $(0,T)\times N$}\\
            u(0) =u_0,\hspace{4mm} \partial_tu(0) = u_1,
        \end{cases}
    \end{align}
    where $u_0,u_1 \in\mathcal{D}'(N)$. For any $\eps\in (0,T)$ there exists a function $U_\eps\in C^\infty((\eps,T-\eps)\times N)$ such that 
    $$
        (\partial_t^2 + \mathcal{L}_{g,A,V})U_\eps = 0 \hspace{4mm} \text{ on $(\eps,T-\eps)\times N$}
    $$
    and for any $t\in (0,T)$ we have that
    $
        \lim_{\eps\to 0} U_\eps(t) = u(t)
    $
    in $\mathcal{D}'(N)$.
\end{prop}

\begin{proof}
    We begin the proof by recalling that there exists $k\in\mathbb{Z}$ such that $u_0\in H^{k+1}(N)$ and $u_1\in H^k(N)$ by \cite[Proposition 10.2]{shubin_psdo}. Since $u$ solves $\eqref{eq:wave_eq_without_source}$ we get from Theorem \ref{thm:distributional_solution} that $u\in \bigcap_{r=0}^\infty H^r(0,T;H^{k+1-l}(N))$,
    % since $f = 0$ is an element of $H^l(0,T;H^k(N))$ 
    for all $l\in\mathbb{N}$, 
    and this distribution can be written as the Fourier series
    $$
        u(t) = \sum_j \alpha_j(t)\varphi_j, \quad \text{ for } t \in (0,T),
    $$
    where $\varphi_j$ are the $L^2$-normalized eigenfunctions of the Magnetic-Schr\"odinger operator $\mathcal{L}_{g,A,V}$, and $\alpha_j$ are the smooth solutions of the initial value problem
    \begin{align*}
        \begin{cases}
            \alpha_j''(t) + \lambda_j \alpha_j(t) = 0 \hspace{4mm}t\in (0,T)\\
            \alpha_j(0) = \langle u_0,\varphi_j\rangle, \,\alpha_j'(0) = \langle u_1,\varphi_j\rangle.
        \end{cases} 
    \end{align*}   
    Furthermore, for $D>0$ sufficiently large we get from Proposition \ref{prop:equiv_of_Sob_spaces} and Theorem \ref{thm:distributional_solution} that 
    \begin{equation}
        \label{eq:series_for_alpha_j}
        \sum_j \norm{\alpha_j}_{H^l(0,T)}(\lambda_j + D)^{k+1-l} <\infty\,\hspace{4mm}\text{for all $l\in \mathbb{N}$}.
    \end{equation}

    For each $s \in \R$ let $\eta(s) = C\text{exp}(\frac{1}{s^2-1})$ when $s\in [-1,1]$ and $\eta(s)=0$ for $s$ elsewhere, further choose $C > 0$ so that $\norm{\eta}_{L^1(\R)} = 1$.
    % \begin{align*}
    %     \eta(s) = 
    %     \begin{cases}
    %     C \text{exp}\Big(\frac{1}{s^2-1}\Big) \hspace{12mm}\text{$|s|<1$}\\
    %     0\hspace{32mm} \text{$|s|\geq 1$}
    %     \end{cases}
    % \end{align*}
    % and $C$ is the constant such that $\int_{\R}\eta(s)ds = 1$ so that $\eta$ is a smooth compactly supported function. 
    Then for $\eps > 0$ and $j \in \N$ we define $\eta_\eps(s): = \frac{1}{\eps}\eta(\frac{s}{\eps})$  and
    % in order to mollify functions, for further discussion see \cite[Appendix C.5]{evans2010partial}. In particular, for each $j\in\mathbb{N}$ we can define a function 
    $\alpha_{j,\eps}:(\eps,T-\eps)\to \R$ as the convolution 
    % $$
    %     \alpha_{j,\eps}(t) = \int_\eps^{T-\eps} \eta_\eps(t-s)\alpha_j(s)ds = \int_{-\eps}^\eps \eta_\eps\Big(\frac{s}{\eps}\Big)\alpha_j(t-s)ds
    % $$
    % so that 
    $\alpha_{j,\eps} := \eta_\eps*\alpha_j$. 
    % is a convolution of the functions. Additionally, by properties of the convolution we have that 
    It is straightforward to show that the functions $\alpha_{j,\eps}$ solve the ODE
    $\alpha_{j,\eps}''(t) + \lambda_j \alpha_{j,\eps}(t) = 0$ for all $t\in (\eps,T-\eps)$. 
    
    For each $\eps>0$ we define $U_\eps :(\eps,T-\eps)\to \mathcal{D}'(N)$ by
    \begin{equation}
        \label{eq:series_for_U_eps}
        U_\eps(t) := \sum_{j} \alpha_{j,\eps}(t)\varphi_j,
    \end{equation}
    and seek to show that $U_\eps\in H^l(\eps,T-\eps;H^r(N))$ for all $l\in \mathbb{N}$ and $r\in\mathbb{Z}$. This will imply the smoothness of $U_\eps$. 
    
    First, by Young's inequality for convolutions we know that for any $\beta\in L^2(0,T)$ that $$\norm{\beta*\eta_\varepsilon}_{L^2(\eps,T-\eps)} \leq \norm{\eta_\eps}_{L^1(\R)} \norm{w}_{L^2(0,T)} = \norm{w}_{L^2(0,T)}$$
    since $\eta_\varepsilon$ is normalized to integrate to $1$. Thus $\norm{\alpha_{j,\varepsilon}}_{H^l(\varepsilon,T-\varepsilon)} \leq \norm{ \alpha_j}_{H^l(0,T)}$ and due to 
    % Theorem \ref{thm:distributional_solution} 
    \eqref{eq:series_for_alpha_j} we conclude that $U_\eps \in H^l(\eps,T-\eps;H^{k+1-l}(N))$. Then we take $d\in \mathbb{N}$ and show that $U_\eps \in H^l(\eps,T-\eps;H^{k+1-l+2d}(N))$. 
    % Note that there exists a constant $C_d > 0$ that depends on $d$ and $D> 0$ such that $(a+D)^{2d}\leq C_d \sum_{i=0}^d a^{2i}$ for all $a\in \mathbb{R}$. To show that $U_\eps\in H^l(0,T;H^{k+1-l+2d}(N))$ it suffices to prove that
    % $$
    %     \sum_j \norm{\alpha_{j,\eps}}_{H^l(\eps,T-\eps)}^2 (\lambda_j+D)^{k+1-l+2d} < \infty
    % $$
    We recall that there is a constant $C_d > 0$ that depends on $d$ and $D> 0$  such that $(a+D)^{2d}\leq C_d \sum_{i=0}^d a^{2i}$ for all $a\in \mathbb{R}$. By this inequality and the equation $\alpha_{j,\eps}''= -\lambda_j\alpha_{j,\eps}$ we find that
    \begin{align*}
        \norm{\alpha_{j,\eps}}_{H^l(\eps,T-\eps)}^2(\lambda_j+D)^{k+1-l+2d} &\leq 
        % C_d(\lambda_j+D)^{k+1-l}\sum_{i=0}^d \lambda_j^{2i}\norm{\alpha_{j,\eps}}_{H^l(\eps,T-\eps)}^2\\
        % &= C_d(\lambda_j+D)^{k+1-l}\sum_{i=0}^d\norm{\lambda_j^i\alpha_{j,\eps}}_{H^l(\eps,T-\eps)}^2\\
        % &= 
        C_d(\lambda_j+D)^{k+1-l}\sum_{i=0}^d\norm{\alpha_{j,\eps}^{(2i)}}_{H^l(\eps,T-\eps)}^2.
    \end{align*}
    % Here we also used the equation $\alpha_{j,\eps}''(t) = -\lambda_j\alpha_j(t)$ for all $t\in (\eps,T-\eps)$.  
    Moreover, since
    \begin{align*}
        \norm{\alpha_{j,\eps}^{(2i)}}_{H^l(\eps,T-\eps)}^2 &= \sum_{s=0}^l \int_\eps^{T-\eps}|\alpha_{j}^{(2i+s)}(t)|^2dt,
    \end{align*}
    and
    \begin{align*}
        \alpha_{j,\eps}^{(2i+s)}(t) &= \frac{1}{\eps^{2i+1}}\int_\eps^{T-\eps} \eta^{(2i)}\Big(\frac{t-\xi}{\eps}\Big)\alpha_j^{(s)}(\xi)d\xi,
    \end{align*}
    we get from Young's inequality for convolutions that $\norm{\alpha_{j,\eps}^{(2i+s)}}_{L^2(\R)} \leq C_{i,\eps}\norm{\alpha_j}_{L^2(\R)}$. Thus
    % \begin{align*}
    %     \norm{\alpha_{j,\eps}^{(2i)}}_{H^l(\eps,T-\eps)}^2 &\leq \frac{\norm{\eta^{(2i)}}_{L^1(-1,1)}^2}{\eps^{4i}}\sum_{s=0}^l \norm{\alpha_j^{(s)}}_{L^2(0,T)}^2 = \frac{\norm{\eta^{(2i)}}^2}{\eps^{4i}} \norm{\alpha_{j}}^2_{H^l(0,T)}.
    % \end{align*}
    % Combining with our work above, we have that
    % \begin{align*}
    %     \norm{\alpha_{j,\eps}}_{H^l(\eps,T-\eps)}^2(\lambda_j+D)^{k+1-l+2d} 
    %     &
    %     \leq C_d(\lambda_j+D)^{k+1-l}\sum_{i=0}^d\norm{\alpha_{j,\eps}^{(2i)}}_{H^l(\eps,T-\eps)}^2\\
    %     &\leq C_d\Bigg(\sum_{i=0}^d \frac{\norm{\eta^{(2i)}}_{L^1(-1,1)}}{\eps^{4i}}\Bigg)\norm{\alpha_j}_{H^l(0,T)}^2(\lambda_j+D)^{k+1-l}\\
    %     &\leq \tilde{C}_{d,\eps} \norm{\alpha_j}_{H^{l}(0,T)}^2(\lambda_j+D)^{k+1-l}
    % \end{align*}
        % Where this constant $\tilde{C}_{d,\eps}$ does not depend on $j$ or $\alpha_j$. With this constant $\tilde{C}_{d,\eps}$ it becomes clear that
    \begin{align*}
        \sum_j \norm{\alpha_{j,\eps}}_{H^l(\eps,T-\eps)} (\lambda_j+D)^{k+1-l+2d} &\leq \tilde{C}_{d,\eps}\sum_{j}\norm{\alpha_j}_{H^l(0,T)}^2(\lambda_j+D)^{k+1-l} < \infty,
    \end{align*}
    where the constant $\tilde{C}_{d,\eps}$ does not depend on $j$ or $\alpha_j$.
    % where the last inequality follows since $u\in H^l(0,T;H^{k+1-l}(N))$ for all $l\in\mathbb{N}$. 
    % Hence $U_\eps \in H^l(0,T;H^r(N))$ for all $l\in \mathbb{N}$ and $r\in\mathbb{Z}$. 
    % This is sufficient to prove that $U_\eps \in C^\infty([\eps,T-\eps]\times N)$ by Sobolev embedding Theorems, for more details see the proof of Theorem \ref{thm:closed_existence}. 
    Hence, the Fourier series \eqref{eq:series_for_U_eps} converges in $H^l(0,T;H^r(N))$ for any $l\in \mathbb{N}$ and $r\in\mathbb{Z}$, and we have in $(\eps,T-\eps)\times N$ that
    
    \begin{align*}
        (\partial_t^2+\mathcal{L}_{g,A,V}) U_\eps(t,x) 
        % &= (\partial_t^2+\mathcal{L}_{g,A,V})\sum_j\alpha_{j,\eps}(t)\varphi_j(x) = \sum_j\Big(\alpha_{j,\eps}''(t)\varphi_j(x)+\alpha_{j,\eps}(t)\mathcal{L}_{g,A,V}\varphi_j(x)\Big)\\
        &= \sum_j\Big(\alpha_{j,\eps}''(t)+\lambda_j\alpha_{j,\eps}(t)\Big)\varphi_j(x) = 0.
    \end{align*}
    
    To finish the proof we must show that for any $w\in C^\infty(N)$ and any $t\in (0,T)$ that $\lim_{\eps\to 0}\langle U_\eps(t),w\rangle = \langle u(t),w\rangle$. Let $\eps>0$ and let $w_j$ be the $j^{th}$-Fourier coefficient of $w$. Since $w$ is smooth we have by Lemma \ref{lem:space_spectral_def} that let $w = \sum_j w_j \varphi_j$ and $\sum_{j}|w_j|^2(\lambda_j+D)^r < \infty$ for all $r\in \mathbb{Z}$.
    Let $k \in \Z$ be such that $u\in H^1(0,T;H^k(N))$ and $\sum_j\norm{\alpha_j}_{H^1(0,T)}^2(\lambda_j+D)^k < \infty$. Then
    by the Sobolev embedding $H^1(0,T)\subset C([0,T])$ and Young's inequality for convolutions we have that 
    % that there is a constant $C$ such that $\norm{v}_{L^\infty(0,T)} \leq C \norm{v}_{H^{1}(0,T)}$ for all $v\in H^1((0,T))$. 
    \begin{align*}
        \norm{\eta_\eps*\alpha_j}_{L^\infty(\eps,T-\eps)} &\leq \norm{\alpha_j}_{L^\infty(0,T)} \norm{\eta_\eps}_{L^1(-\eps,\eps)} \leq C\norm{\alpha_j}_{H^1(0,T)}.
    \end{align*}    
    Thus
    \begin{align*}
        |\langle U_\eps(t),w\rangle| 
        % &\leq \sum_j |\alpha_{j,\eps}(t)w_j| \leq \sum_j \norm{\eta}_{L^1(-1,1)} \norm{\alpha_j}_{L^\infty(0,T)}|w_j|\\
        % &\leq C\sum_j \norm{\alpha_j}_{H^1(0,T)} (\lambda_j+D)^{k/2}|w_j|(\lambda_j+D)^{-k/2}\\
        &\leq C \Big(\sum_j \norm{\alpha_j}_{H^1(0,T)}^2(\lambda_j+D)^k\Big)^{1/2}\Big(\sum_j|w_j|^2(\lambda_j+D)^{-k}\Big)^{1/2}<\infty.
    \end{align*}
    % Now $\sum_j\norm{\alpha_j}_{H^1(0,T)}^2(\lambda_j+D)^k < \infty$ since $u\in H^1(0,T;H^k(N))$ and $\sum_j|w_j|^2(\lambda_j + D)^{-k}<\infty$ since $w \in C^\infty(N) = \bigcap_{k\in\mathbb{Z}} H^k(N)$. 
    Since the second term above is independent of $\eps$, 
    we can apply the Dominated Convergence Theorem and obtain
    \begin{align*}
         \lim_{\eps\to 0} \langle U_\eps(t),w\rangle = \lim_{\eps\to 0}\sum_j \alpha_{j,\eps}(t)\overline{w_j} = \sum_j\lim_{\eps\to0} \alpha_{j,\eps}(t)\overline{w_j} =\sum_j\alpha_j(t)\overline{w_j} = \langle u(t),w\rangle.  \end{align*}\qedhere      % The last line is trivial due the construction of $U_\eps$ with the given Fourier series representation.
\end{proof}
Next we show that if the distributional solution $u$ of the previous proposition vanishes on the space-time cylinder $(0,T)\times\cX$, then for each $\eps\in (0,T)$ the respective mollification $U_\eps$ will also vanish in the shorted cylinder $(\eps,T-\eps)\times\cX$.

\begin{lem} \label{lem:vanishing_convolution}
    Let $(N,g)$ be a closed Riemannian manifold. Let $u\in C^1([0,T];\mathcal{D}'(N))$ be the unique distributional solution to the Cauchy Problem \eqref{eq:wave_eq_without_source} with $u|_{(0,T)\times \cX} = 0$. If $\eps>0$ and $U_\eps$ is as in Proposition \ref{prop:mollified_solution} then $U_\eps|_{(\eps,T-\eps)\times \cX} = 0$.
\end{lem}

\begin{proof}
    % Let $u\in H^1(0,T;H^k(N))$ be the distributional solution in $(\ref{eq:wave_eq_without_source})$ where $k\in\mathbb{Z}$, possibly negative. 
    % Then we have that
    % $$
    % u(t) = \sum_j \alpha_j(t)\varphi_j
    % $$
    % where $\alpha_j$ is a smooth function and since $u\in H^1(0,T;H^k(N))$ we have that the Fourier coefficients satisfy
    % $$
    %     \sum_j \norm{\alpha_j}_{H^1(0,T)}^2(\lambda_j + D)^k < \infty
    % $$
    % If we let $\eta(s) = C\exp(\frac{1}{x^2-1})$ on $(-1,1)$, let $\eta(s) = 0$ elsewhere and let $C$ be chosen so that $\int_{\R}\eta(s)ds = 1$ then we can consider the functions $\eta_\eps(s) = \frac{1}{\eps} \eta(\frac{s}{\eps})$. Next define $\alpha_{j,\eps}:(\varepsilon,T-\varepsilon)\to\C$ by $\alpha_{j,\eps} = \eta_\eps* \alpha_j$. 
    Let $\eps>0$. As in the proof of Proposition \ref{prop:mollified_solution} we define 
    % for  $t\in (\eps,T-\eps)$
    % $$
    %     U_\eps(t) := \sum_j\alpha_{j,\eps}(t)\varphi_j .
    % $$
    % Hence, we know that 
    $U_\eps\in C^\infty((\eps,T-\eps)\times\cX)$ by \eqref{eq:series_for_U_eps}. Thus, for all $r\in\mathbb{Z}$ we will have that
    $$
        \sum_j\norm{\alpha_{j,\eps}}_{H^1(\eps,T-\eps)}^2(\lambda_j + D)^r < \infty.
    $$

    Our goal is to show that $U_\eps|_{(\eps,T-\eps)\times \cX} = 0$. To that end let $w\in C_0^\infty((\eps,T-\eps)\times \cX)$ and use the same notation for the respective zero extension. Since $w$ is smooth and compactly supported, then if $w(t,x) = \sum_j\beta_j(t)\varphi_j(x)$ is the respective Fourier series representation, we have that
    $$                  \sum_j\norm{\beta_{j}}_{H^l(\eps,T-\eps)}^2(\lambda_j + D)^r < \infty,
    $$
    for all $r\in\Z$ and all $l\in\mathbb{N}$.  
    
    For each $t \in (\eps, T-\eps)$ we define $w_\eps(t) :=\sum_j (\eta_\eps*\beta_j)(t)\varphi_j$. 
    % where $\eta_\eps(s) = \frac{1}{\eps}\eta(\frac{s}{\eps})$ for a function $\eta$ that is even, smooth, positive, and $L^1$-normalized. 
    Our aim is to verify that this series converges in $C^\infty$-topology to the function $(t,x)\mapsto \int_{-\eps}^\eps\eta_\eps(s)w(t-s,x)ds$. Therefore, $w_\eps$ is the convolution of $w$ only in the time coordinate, and if $w(t,x)=0$ for some $x \in N$ and all $t \in \R$ then $w_\eps(t,x) = 0$ for all $t\in \R$. 
    
    % To do this we recall that $w\in C_0^\infty((\eps,T-\eps)\times \cX)$ so $w\in C^\infty([-\eps,T+\eps]\times N)$ where $N$ is a compact manifold. 
    Our first step in this process will be to prove that for large enough $k \in \N$ there exists a constant $C > 0$ such that
    \begin{align}
         \max_{(s,y)\in [-\eps,T+\eps]\times N}\sum_j|\beta_j(s)\varphi_j(y)| \leq C \norm{w}_{H^1(-\eps,T+\eps;H^k(N))}. \label{eq:absolute_convergence}
    \end{align}
    To verify this estimate we define for each $y\in N$ and $j\in\N$ a constant
    \begin{align*}
        c_j(y) :&= \begin{cases}
            \frac{\overline{\varphi_j(y)}}{|\varphi_j(y)|} \hspace{12mm}\text{when $\varphi_j(y)\not=0$}\\
           \hspace{2mm} 1 \hspace{18mm}\text{when $\varphi_j(y) = 0$}
        \end{cases}
        \intertext{so that $|c_j(y)| = 1$. Also, we define formally the following  distribution on $N$}
        w_y(x) :&= \sum_j c_j(y)\norm{\beta_j}_{H^1(\eps,T-\eps)} \varphi_j(x),
    \end{align*}
    and show that this is actually a smooth function such that 
    \begin{equation}
        \label{eq:w_y_at_y}
    w_y(y) = \sum_j \norm{\beta_j}_{H^1(\eps,T-\eps)}|\varphi_j(y)|.
    \end{equation}
    
    Using definitions \ref{def:sobolev_spectral_def} and \ref{def:time_sobolev_spectral_space} we can compute  for all $k\in \Z$ that
    \begin{equation}
        \label{eq:norm_of_w_y}
        \norm{w_y}_{H^k(N)}^2 
        % = \sum_j |c_j(y)|^2\norm{\beta_j}_{H^1(\eps,T-\eps)}^2 (\lambda_j+D)^k 
        % = \sum_j \norm{\beta_j}_{H^1(\eps,T-\eps)}^2 (\lambda_j+D)^k 
        = \norm{w}_{H^1(\eps,T-\eps;H^k(N))}^2.
    \end{equation}
    Hence $w_y\in H^k(N)$ for all $k\in \mathbb{Z}$ and by Sobolev embedding $w_y$ is smooth, and the equation \eqref{eq:w_y_at_y} follows.  
    % since $w\in C_0^\infty((\eps,T-\eps)\times\cX)$. 
    Moreover, for $k$ large enough the equations \eqref{eq:w_y_at_y} and \eqref{eq:norm_of_w_y} in conjunction with Sobolev embedding imply
    % (i.e. Lemma \ref{sobolev_inequality}) we may find a constants $D_1, D_2 > 0$ such that $\norm{\beta}_{C([\eps,T-\eps])} \leq D_1\norm{\beta}_{H^1(\eps,T-\eps)}$ and $\norm{v}_{C(N)} \leq D_2 \norm{v}_{H^k(N)}$ for sufficiently large $k\in\N$. With this in mind, for any $(s,y)\in [\eps,T-\eps]\times N$ we have the estimate
    \begin{align*}
        \sum_j|\beta_j(s)\varphi_j(y)| 
        % &\leq  \sum_j\norm{\beta_j}_{C([\eps,T-\eps])}|\varphi_j(y)|
        % \\
        % &\leq D_1\sum_j\norm{\beta_j}_{H^1(\eps,T-\eps)}|\varphi_j(y)|\\
        % &= D_1 w_y(y)\\
        % &\leq D_1 \norm{w_y}_{C(N)}\\
        &\leq 
        % D_1 D_2 \norm{w_y}_{H^k(N)}= D_1 D_2 
        C\norm{w}_{H^1(\eps,T-\eps;H^k(N))}, 
        \quad \text{ for all } (s,y)\in [\eps,T-\eps]\times N,
    \end{align*}
    where $C$ is independent of $(s,y)\in [\eps,T-\eps]\times N$. Therefore, we arrive at \eqref{eq:absolute_convergence}. 
    
    Furthermore, since $\|\eta_\eps\|_\infty\leq D_\eps$ for some postive constant $D_\eps$, which depends on $\eps$,  the previous argument yields for all $s\in (-\eps,\eps)$ that
    % it is clear that if we fix $(t,x)\in [0,T]\times N$ and let $s\in (-\eps,\eps)$ be arbitrary, we have the estimate
    $$
        \sum_j|\eta_\eps(s)\beta_j(t-s)\varphi_j(x)| \leq CD_\eps \norm{w}_{H^1(\eps,T-\eps;H^k(N))} < \infty, 
        \quad \text{ whenever } (t,x)\in [0,T]\times N.
    $$
    Hence, $\sum_j \eta_\eps(s)\beta_j(t-s)\varphi_j(x)$ is absolutely integrable with respect to $s$, and by Dominated Convergence Theorem we have that
    \begin{align*}
        \int_{-\eps}^\eps \eta_\eps(s)w(t-s,x)ds
        % &= \int_{-\eps}^\eps \Big(\eta_\eps(s)\sum_{j}\beta_j(t-s)\varphi_j(x)\Big)ds\\
        % &= \sum_j\Big(\int_{-\eps}^\eps \eta_\eps(s)\beta_j(t-s)ds\Big)\varphi_j(x)\\
        &= \sum_j \beta_{j,\eps}(t)\varphi_j(x) = w_\eps(t,x).
    \end{align*}
    
    Next, we fix $j \in \N$. 
    Since, $w\in C_0^\infty((\eps,T-\eps)\times \cX)$ 
    % By standard properties of the convolution and from the assumptions for the support of $w$ 
    we have that $w_\eps \in C_0^\infty((0,T)\times \cX)$ and compute that
    \begin{align*}
        \int_\eps^{T-\eps} \alpha_{j,\eps}(t)\overline{\beta_j(t)}dt 
        % &= \int_\eps^{T-\eps} \Bigg(\int_{-\eps}^\eps \frac{1}{\eps}\eta\Big(\frac{s}{\eps}\Big)\alpha_j(t-s)ds\Bigg)\overline{\beta_j(t)}dt\\
        % &= \int_{-\eps}^\eps\int_\eps^{T-\eps} \frac{1}{\eps}\eta\Big(\frac{s}{\eps}\Big)\alpha_j(t-s)\overline{\beta_j(t)}dtds\\
        % &= \int_{-\eps}^\eps\int_{\eps-s}^{T-\eps -s} \frac{1}{\eps}\eta\Big(\frac{s}{\eps}\Big)\alpha_j(r)\overline{\beta_j(r+s)}dr ds\\
        % &= \int_{-\eps}^\eps\int_{\eps+\xi}^{T-\eps +\xi} \frac{1}{\eps}\eta\Big(\frac{\xi}{\eps}\Big)\alpha_j(r)\overline{\beta_j(r-\xi)}drd\xi\\
        % &= \int_{-\eps}^\eps\int_{0}^{T} \frac{1}{\eps}\eta\Big(\frac{\xi}{\eps}\Big)\alpha_j(r)\overline{\beta_j(r-\xi)}drd\xi\\
        &= \int_0^T \alpha_j(t)\overline{
        \beta_{j,\eps}(t)}dt.  
    \end{align*}
    Here $\beta_{j,\eps}:=\eta_\eps*\beta_j$, and in the derivation of this equality we used the evenness of the mollifier $\eta$ as well as $w\in C_0^\infty((\eps,T-\eps)\times\cX)$ to conclude that $\beta_j$ vanishes outside of $(\eps,T-\eps)$. 
    % Hence, we have
    % \[
    % \int_\eps^{T-\eps} \alpha_{j,\eps}(t)\overline{\beta_j(t)}dt = \int_0^T \alpha_j(t)\overline{\beta_{j,\eps}(t)}dt
    % \]
    
    To finish the proof we will need to verify equalities
    \begin{align}
         \int_{\eps}^{T-\eps} \sum_j \alpha_{j,\eps}(t)\overline{\beta_j(t)}dt &= \sum_j \int_\eps^{T-\eps}\alpha_{j,\eps}(t)\overline{\beta_j(t)}dt \label{eq:a_eps}
         \intertext{and}
          \int_{\eps}^{T-\eps} \sum_j \alpha_{j}(t)\overline{\beta_{j,\eps}(t)}dt &= \sum_j \int_\eps^{T-\eps}\alpha_{j}(t)\overline{\beta_{j,\eps}(t)}dt. \label{eq:b_eps}
    \end{align}
    We will prove the latter one and note the former equality follows by a similar argument. Recall that $u \in H^1(0,T;H^k(N))$ and $w_\eps \in C_0^\infty((0,T)\times \cX) \subset H^1(0,T;H^{-k}(N))$. By the Cauchy-Schwarz inequality and the Sobolev embedding theorem, we have for any $t\in[0,T]$ that
    \begin{align*}
        \sum_j |\alpha_{j}(t)\overline{\beta_{j,\eps}(t)}| &\leq \Big(\sum_j|\alpha_j(t)|^2(\lambda_j+D)^k\Big)^{1/2}\Big(\sum_j |\beta_{j,\eps}(t)|^2(\lambda_j + D)^{-k}\Big)^{1/2}\\
        % &\leq C^2\Big(\sum_j\norm{\alpha_j}_{H^1(0,T)}^2(\lambda_j+D)^k\Big)^{1/2}\Big(\sum_j \norm{\beta_{j,\eps}}^2_{H^1(0,T)}(\lambda_j + D)^{-k}\Big)^{1/2}\\
        &= C\norm{u}_{H^1(0,T;H^k(N))}\norm{w_\eps}_{H^1(0,T;H^{-k}(N))} < \infty.
    \end{align*}
    Hence, $\sum_j |\alpha_{j}(t)\overline{\beta_{j,\eps}(t)}|$ is integrable on $[0,T]$, and the equation $(\ref{eq:b_eps})$ follows from the Dominated Convergence Theorem. 
    % A similar argument shows that $(\ref{eq:a_eps})$ also holds. 
    
    As a consequence, all series and integrals needed to verify the following equation commute. Therefore,
    \begin{align*}
        \int_\eps^{T-\eps}\int_N U_\eps(t,x)\overline{w(t,x)}dV_gdt 
%        &= \int_\eps^{T-\eps}\langle U_\eps(t),w(t)\rangle_gdt\\
%        &= \int_\eps^{T-\eps}\sum_j \alpha_{j,\eps}(t)\overline{\beta_j(t)}dt\\
%        &= \sum_j\int_\eps^{T-\eps}\alpha_{j,\eps}(t)\overline{\beta_j(t)}dt\\
%        &= \sum_j\int_0^{T}\alpha_{j}(t)\overline{\beta_{j,\eps}(t)}dt\\
%        &=\int_0^{T}\sum_j\alpha_{j}(t)\overline{\beta_{j,\eps}(t)}dt\\
        &=\int_0^{T}\langle u(t),w_\eps(t)\rangle dt = 0.
    \end{align*}
    Here the last equality holds because $w_\eps\in C_0^\infty((0,T)\times\cX)$ and $u|_{(0,T)\times \cX}=0$ in the sense of distributions. Recall that $w\in C_0^\infty((\eps,T-\eps)\times \cX)$ was abitrary, hence $U_\eps|_{(\eps,T-\eps)\times \cX} = 0$ in the sense of distribution. But since $U_\eps\in C^\infty((\eps,T-\eps)\times N)$, then $U_\eps|_{(\eps,T-\eps)\times \cX} = 0$ pointwise.
\end{proof}

For an open set $\cX \subset N$ we define the following open double light cone 
\begin{equation}
    \label{eq:double_cone}
    C(T,\cX) = \{(t,x)\in (0,2T)\times N: \text{dist}_g(x,\cX) < \min\{t,2T-t\}\}.
\end{equation}

\begin{thm}[Tataru's Unique Continuation principle]\label{thm:unique_continuation}
    Let $(N,g)$ be a complete Riemannian manifold, and $\cX\subset N$ be open and bounded. Let $A$ and $V$ be a smooth real co-vector field and a smooth real-valued function on $N$. Let $T>0$ and suppose that $u \in C^\infty([0,2T]\times N)$ solves $(\partial_t^2 + \mathcal{L}_{g,A,V})u = 0 $ in $(0,2T)\times M(\cX,T)$ and $u|_{(0,2T)\times\cX} = 0$. Then $u|_{C(T,\cX)} = 0$.
\end{thm}

\begin{proof}
    This result was originally shown in the Euclidean case in Theorem 1 and Theorem 2 of \cite{tataru1995unique}. For the Riemannian version we refer to Theorem 2.66 and Theorem 3.16 of \cite{lassas_inverse}.
\end{proof}

By $H_0^k(M(\cX,T),g)$ we shall mean the completion of $C_0^\infty(\text{int}(M(\cX,T)))$ with respect to $\norm{\cdot}_{H^k(N,g)}$ as defined in Definition \ref{def:sobolev_def}. We develop an Approximate Controllability result for higher order Sobolev spaces in order to prove the existence of nonvanishing waves at the end of this subsection.

\begin{thm}[Higher Order Approximate Controllability]
\label{thm:approximate_controllability}
    Let $(N,g)$ be a complete Riemannian manifold, let $A$ and $V$ be a smooth real co-vector field and a smooth real-valued function on $N$ respectively. Let $\cX\subset N$ be open and bounded. For any $T > 0$, the set
    \begin{align*}
        \mathcal{W}_T := 
        & \{w^f(T) \in C^\infty_0(N): \: 
        w^f \text{ is a solution to \eqref{eq:cauchy_problem_infinite_time} with source }
       f\in C_0^\infty((0,T)\times\cX)
        \}
    \end{align*}
    is dense in $H_0^k(M(\cX,T),g)$ with respect to the $H^k(N,g)$ topology for all $k \in \{0,1,2,\ldots\}$. 
\end{thm}

\begin{rem}
The $L^2$-version of the approximate controllability was established in \cite{source-to-solution} for the solutions of a standard wave equation. 
\end{rem}

\begin{proof}
    % If $f\in C_0^\infty((0,2)\times\cX)$ so that the compact set $\text{supp}(f)$ is properly contained in the open set $(0,T)\times \cX$. 
    Let $k\in \{0,1,2,\ldots\}$. Then by Finite speed of propagation (i.e. Corollary \ref{cor:finite_speed}) we have that 
    % $$\text{supp}(u^f(T))\subset \{y\in N: \text{dist}_g(y,\cX) < T\} \subset \text{int}(M(\cX,T)).
    % $$ It follows that    
    \begin{equation}
       \label{eq:inclusions_of_spaces}
    \mathcal{W}_T \subset C_0^\infty(\text{int}(M(\cX,T))) \subset H_0^k(M(\cX,T),g) \subset H^k(N,g).
    \end{equation} 
    % for all $k\in \{0,1,2,\ldots\}$. 
    % % Since $\cX$ is bounded, then $M(\cX,T)$ is compact.
    % Fix $k\in \{0,1,2,\ldots\}$ so that $\mathcal{W}_T \subset H_0^k(M(\cX,T),g) \subset H^k(N,g)$ and give these spaces the standard Sobolev norm induced by $g$ so that $H^k(N,g)$ is a well-defined Banach space and the spaces $\mathcal{W}_T$ and $H_0^k(M(\cX,T),g)$ are vector subspaces of $H^k(N,g)$. 
    We let $(H^k(N,g))^*\subset \mathcal{D}'(N)$ denote the space of continuous linear functionals on $H^k(N,g)$ with respect to the standard $H^k(N,g)$ topology.\\ \newline
    For a vector subspace $V \subset H^k(N,g)$ we let 
    $$
        V^\perp := \{\ell \in (H^k(N,g))^*: \ell(v) = 0\,\,\text{for all $v\in V$}\}.
    $$
    % be the orthogonal complement of $V$. 
    % We note that the set $\mathcal{W}_T \subset H_0^k(M(\cX,T),g)$ is dense with respect to $H^k(N,g)$-topology if $\overline{\mathcal{W}_T} = H_0^k(M(\cX,T),g)$ 
    Since $H_0^k(M(\cX,T),g)$ is a closed vector subspace of $H^k(N,g)$ then $\mathcal{W}_T \subset H_0^k(M(\cX,T),g)$ is dense with respect to $H^k(N,g)$-topology if $\mathcal{W}_T^\perp = (H_0^k(M(\cX,T),g))^\perp$ (c.f.\cite[Proposition 1.9]{brezis2011functional}). 
    Due to \eqref{eq:inclusions_of_spaces} we always have that $(H_0^k(M(\cX,T),g))^\perp \subset \mathcal{W}_T^\perp$. 
    Thus our goal is to verify the reverse inclusion. 
    % that if $\rho \in \mathcal{W}_T^\perp$, then $\rho\in (H_0^k(M(\cX),T),g)^\perp$. 
    Due to continuity we have that
    % since
    \begin{align*}
        (H_0^k(M(\cX,T),g))^\perp 
        % &= \{ \rho \in (H^k(N,g))^*: \rho(w) = 0 \,\text{for all }w \in H_0^k(M(\cX,T))\}\\
        % &= \{ \rho \in (H^{k}(N,g))^*: \rho(w) = 0 \,\text{ for all }w \in C_0^\infty(\text{int}(M(\cX,T)))\}\\
        &= \{ \rho \in (H^{k}(N,g))^*: \rho|_{\text{int}(M(\cX,T))} = 0\}.
    \end{align*}
    % \teemu{What do we mean by $\rho|_{\text{int}(M(\cX,T))} = 0$? This needs some more details.}
    Thus, in order to finish the proof it suffices to show that if $\rho \in \mathcal{W}_T^\perp$ then $\rho|_{\text{int($M(\cX,T)$)}} = 0$. 
    
    Let $\rho \in \mathcal{W}_T^\perp$.
    % and $\chi\in C_0^\infty(N)$ be such that $\chi|_{M(\cX,T)} = 1$. Then $\chi\rho \in (H_0^k(M(\cX,T),g))^*$ 
    % \teemu{Should we have here  $(H_0^k(M(\cX,T),g))^\perp$}
    % if and only if $\rho\in (H_0^k(M(\cX,T)))^*$. Thus 
    % Since $\cX$ is bounded we may assume without loss of generality that $\rho$ has a compact support. 
    % We plan to use the gluing argument as was done in the proof of Theorem \ref{thm:distributional_solution} in order to study the compact region $M(\cX,T)$ on closed manifold rather than a noncompact complete manifold.
    Since $\cX$ is bounded, we may assume without loss of generality that $\rho$ is supported in a compact set $K \subset N$ that also contains the domain of influence $M(\cX,T)$. By Lemma \ref{lem:compact_exhaustion} there exists a regular domain $M$ such that $M(K,T) \subset M \subset N$. Then by Proposition \ref{prop:magnetic_extension} there exists a closed Riemannian manifold $(\tilde{N},\tilde{g})$ and a Riemannian isometry $\phi:(M,g)\to (\tilde{M},\tilde{g})$ for a regular region $\tilde{M} \subset \tilde{N}$.
    Furthermore, there exists a Magnetic-Schr\"odinger operator $\tilde{\mathcal{L}}$ on $\tilde{N}$ such that for all smooth functions $\tilde{w}$ which are compactly supported on the interior of $\tilde{M}$ we have that  
    $$
        \mathcal{L}_{g,A,V}(\tilde{w}\circ\phi) = [\tilde{\mathcal{L}}(\tilde{w})]\circ\phi \hspace{4mm}\text{on the interior of $\tilde{M}$.}
    $$
    Since $\phi$ is a Riemannian isometry then $\phi(M(\cX,T)) = M(\phi(\cX),T)$. Let $\chi\in C^\infty(\tilde{N})$ satisfy $\chi = 1$ on $M(\phi(\cX),T)$ and $\text{supp}(\chi) \subset\tilde{M}$. Now define a linear map $\tilde{\rho}:C^\infty(\tilde{N})\to \C$ by
    $$
        \tilde{\rho}(\tilde{w}) := \rho\Big((\chi\tilde{w})\circ\phi\Big)
    $$
    where $(\chi\tilde{w})\circ\phi$ has been extended to be equal to $0$ outside $M$ so that $(\chi\tilde{w})\circ\phi\in C_0^\infty(N)$. If $\{\tilde{w}_n\} \subset C^\infty(\tilde{N})$ converges then $(\chi\tilde{w}_n)\circ\phi$ will also converge in $C^\infty_0(N)$, then since $\rho\in \mathcal{D}'(N)$ it follows that $\tilde{\rho}\in\mathcal{D}'(\tilde{N})$. 
    
    Since $\tilde N$ is compact we have that $\mathcal{D}'(\tilde{N}) = \bigcup_{r\in\mathbb{Z}} H^r(\tilde{N})$ by \cite[Proposition 10.2]{shubin_psdo} 
    and there exists $r\in\mathbb{Z}$ such that $\tilde{\rho}\in H^{r+1}(\tilde{N})$. Due to Theorem \ref{thm:distributional_solution} for any $l\in \mathbb{N}$ there exists a unique distributional solution $\tilde{u}\in H^l(0,T;H^{r+1-l}(\tilde{N}))$  to the backward in time Cauchy Problem
    \begin{align}
        \begin{cases} \label{eq:approximate_control_pde_1}
        (\partial_t^2 + \tilde{\mathcal{L}})\tilde{u} = 0, \quad\quad \text{ in } (0,T) \times \tilde{N}
        \\
        \tilde{u}(T,\cdot) = 0,\,\, \partial_t\tilde{u}(T,\cdot) = \tilde{\rho}.
        \end{cases}
    \end{align}
    Here we mean that $\tilde{u}$ is a distributional solution to $(\ref{eq:approximate_control_pde_1})$ if
    \begin{align}
        \int_0^T \langle \tilde{u}(t),(\partial_t^2+\tilde{\mathcal{L}})\tilde{w}(t,\cdot)\rangle dt &= \langle \tilde{\rho}, \tilde{w}(T)\rangle, 
        \quad \text{ for all } \tilde{w}\in C_0^\infty((0,T]\times\tilde{N}).
        \label{eq:approximate_control_pde_2}
    \end{align}
    % this is the analog to the definition of the distributional solution in $(\ref{eq:distributional_solution})$ when initial conditions are given instead of terminal conditions. 
    
    Let $\tilde{f}\in C_0^\infty((0,T)\times\phi(\cX))$ and let $\tilde{w} \in C^\infty([0,T]\times\tilde{N})$ be the unique classical solution of 
    % $(\ref{eq:cauchy_problem})$ with vanishing initial conditions and source term $\tilde{f}$. 
    \begin{align*}
        \begin{cases}
        (\partial_t^2+\tilde{\mathcal{L}})\tilde{w} = \tilde f, \quad\quad \text{ in } (0,T) \times \tilde N
        \\
        \tilde w(0) = \partial_t \tilde w(0) = 0.
        \end{cases}
    \end{align*}
     If we define $w(t,x) = \tilde{w}(t,\phi(x))$ and $f(t,x) = \tilde{f}(t,\phi(x))$
    for all $x\in M$ and $0$ outside of $M$, then by the finite speed of wave propagation the function $w$ is the unique solution to the problem \eqref{eq:cauchy_problem_infinite_time} with the interior source $f\in C_0^\infty((0,T)\times\cX)$.
    % and $w\in C^\infty([0,T]\times N)$ by finite speed of propagation since $\text{supp}(\tilde{w}(t)) \subset M(\phi(\cX),T) \subset \tilde{M}$ for all $t\in [0,T]$. Further, $w$ is the unique solution to the problem \eqref{eq:cauchy_problem_infinite_time}.
    %  \begin{align*}
    %     \begin{cases}
    %     (\partial_t^2 + \mathcal{L}_{g,A,V})w = f, \quad\quad \text{ in } (0,T) \times N
    %     \\
    %     w(0) = \partial_t w(0) = 0.
    %     \end{cases}
    % \end{align*}
    In particular, we have that $w(T)\in \mathcal{W}_T$. 
    % Thus 
    % $$
    % \langle \tilde{\rho},\tilde{w}(T)\rangle = \tilde{\rho}(\tilde{w}(T)) = \tilde{\rho}(w(T,\cdot)\circ\phi) = \rho(w(T)) = 0
    % $$
    Since $\rho\in \mathcal{W}_T^\perp$ we deduce from \eqref{eq:approximate_control_pde_2} that
    $$
        \int_0^T \langle \tilde{u}(t),\tilde{f}(t)\rangle dt = 0.
    $$
    As $f\in C_0^\infty((0,T)\times\cX)$ was arbitrarily chosen, the distribution $\tilde{u}(t)$ vanishes in the set $C_0^\infty(\phi(\cX))$ for all $t\in (0,T]$. 
    % By the continuity terminal conditions in $(\ref{eq:approximate_control_pde_1})$ we have that $\tilde{u}(T) = 0$ as well. 
    
    Now let $\tilde{U}$ be the unique distributional solution to
    \begin{align*}
        \begin{cases}
        (\partial_t^2 + \tilde{\mathcal{L}})\tilde{U} = 0, \quad\quad \text{ in } (0,2T) \times \tilde{N}
        \\
        \tilde{U}(0,\cdot) = \tilde{u}(0),\,\, \partial_t\tilde{U}(0,\cdot) = \partial_t\tilde{u}(0).
        \end{cases}
    \end{align*}
    By uniqueness of distributional solutions on closed manifolds we have that $\tilde{U}|_{[0,T]\times \tilde{N}} = \tilde{u}$ as distributions. 
    Then for $t\in [T,2T]$ we consider the distribution $\Tilde{v}(t,x) =-  \tilde{u} (2T-t,x)$ which is the unique distributional solution to 
    \begin{align*}
        \begin{cases}
        (\partial_t^2 + \tilde{\mathcal{L}})\tilde{v} = 0, \quad\quad \text{ in } (T,2T) \times \tilde{N}
        \\
        \tilde{v}(T,\cdot) = 0,\,\, \partial_t\tilde{v}(T,\cdot) = \tilde{\rho}.
        \end{cases}
    \end{align*}
    Since $\tilde{U}(T,\cdot) = \tilde{u}(T,\cdot) = 0$ and $\partial_t\tilde{U}(T,\cdot) = \partial_t \tilde{u}(T,\cdot) = \tilde{\rho}$, then uniqueness implies that $U|_{[T,2T)\times \tilde{N}} = \tilde{v}$ and in particular $\tilde{U}|_{(0,2T)\times \phi(\cX)} = 0$. 

    To finish the proof we verify the second equality in 
    \begin{equation}
        \label{eq:rho_vanishes}
    \tilde{\rho}|_{\text{int}(M(\phi(\cX),T))} =\p_t \tilde{U}(T,\cdot)|_{\text{int}(M(\phi(\cX),T))} = 0.
    \end{equation}
    By Proposition \ref{prop:mollified_solution} we know that for every $\eps\in (0,2T)$ there exists a smooth function $\tilde{U}_\eps \in C^\infty((\eps,2T-\eps)\times\tilde{N})$ which satisfies $(\partial_t^2+\tilde{\mathcal{L}})\tilde{U}_\eps = 0$ on $(\eps,2T-\eps)\times \tilde{N}$ in the classical sense and $\tilde{U}_\eps(t)\to \tilde{U}(t)$ for all $t\in (0,2T)$ in the sense of distributions as $\eps\to 0$. 
    We also know that $\tilde{U}_\eps|_{(\eps,2T-\eps)\times \phi(\cX)} = 0$ since $\tilde{U}|_{(0,2T)\times\phi(\cX)} = 0$ by Lemma \ref{lem:vanishing_convolution}.
    
    By Tataru's Unique Continuation (Theorem \ref{thm:unique_continuation}) we have that $\tilde{U}_\eps|_{C_\eps(T,\phi(\cX))} = 0$ where $C_\eps(T,\phi(\cX))$ is the double light cone defined by
    $$
        C_\eps(T,\phi(\cX)) = \{(t,x) \in (\eps,2T-\eps)\times \tilde{N}: \text{dist}_{\tilde{g}}(x,\phi(\cX)) < \min\{ t-\eps,2T-t-\eps\}\}.
    $$    
    In particular, if $Z \subset \tilde N$ is an open set such that $\overline{Z} \subset \text{int}(M(\phi(\cX),T))$ then there is a $\delta>0$ such that 
    $
    (-\delta,\delta) \times Z \subset C_\eps(T,\phi(\cX))
    $
    for every $\eps>0$ small enough. Hence $\tilde U_\eps$ vanishes in the set $(-\delta,\delta) \times Z$, whenever $\eps>0$ is small enough.
    % Now for any $\hat{\eps} > 0$ and any 
    % $\eps \in (0,\hat{\eps})$ we have that 
    % \[
    % C_{\hat{\eps}}(T,\phi(\cX)) \subset C_\eps(T,\phi(\cX)) \subset C(T,\phi(\cX)),
    % \]
    % and 
    % $$
    % \{T\}\times \text{int}(M(\phi(\cX),T-\hat{\eps})) \subset C_{\hat{\eps}}(T,\phi(\cX)).
    % $$ 
    % Therefore, $\tilde{U}_\eps(T,\cdot)|_{\text{int}(M(\phi(\cX),T-\hat{\eps})} = 0$ for all 
    % $\eps \in (0,\hat{\eps})$. 
    % Recall that 
    % \[
    % \lim_{\eps\to 0^+} \tilde{U}_\eps(T) = \tilde{U}(T) = \tilde{\rho}
    % \]
    % where the limit is taken in the sense of distributions. 
    Therefore, we have that
    $$
     \tilde{U}(t)|_{Z} = \lim_{\eps \to 0^+} \tilde{U}_\eps(t)|_{Z} = 0, \quad \text{ for all } t\in (-\delta,\delta),
    $$
    and the equation \eqref{eq:rho_vanishes} follows.
    % Since $\tilde{\rho}|_{\text{int}(M(\phi(\cX),T-\hat{\eps}))} = 0$ for all $\hat{\eps} > 0$, then $\tilde{\rho}|_{\text{int}(M(\phi(\cX),T)} = 0$. Because $\text{int}(M(\phi(\cX),T)) \subset \tilde{M}$ and $\phi:(M,g)\to (\tilde{M},\tilde{g})$ is a Riemannian isometry then $\rho|_{\text{int}(M(\cX,T))} = 0$ and we can conclude that $\rho \in (H_0^k(M(\cX,T),g))^\perp$ which proves that $\mathcal{W}_T^\perp = (H_0^k(M(\cX,T)))^\perp$ which in turn proves that $\overline{\mathcal{W}_T} = H_0^k(M(\cX,T),g)$.
\end{proof}

% We will also need a nonvanishing condition when we study lower order terms in Section \ref{sec:lower_order_terms}. The following corollary holds even if $\cX$ is unbounded.
\begin{cor}[Nonvanishing Wave]
\label{cor:nonvanishing_condition}
    Let $(N,g)$ be a complete Riemannian manifold, $T>0$, $\cX \subset N$ open, and let $x_0 \in N$ be such that $d_g(x_0,\cX) < T$. There exists $f\in C_0^\infty((0,T)\times \cX)$ such the respective smooth solution $u^f$ of the Cauchy Problem \eqref{eq:cauchy_problem_infinite_time} does not vanish at $(T,x_0)$.
\end{cor} 

\begin{proof}
    Since $\text{dist}_g(x_0,\cX) < T$, 
    % then if necessary we may find an open bounded set $\tilde{\cX} \subset \cX$ such that $d_g(x_0,\tilde{\cX}) < T$ so 
    we may assume without loss of generality that $\cX$ is bounded. Thus, $M(\cX,T)$ is compact and $x_0 \in \text{int}(M(\cX,T))$. We choose
    $\psi \in C_0^\infty(\text{int}(M(\cX,T)))$ such that $\psi(x_0)\not=0$, and fix $k > \frac{n}{2}$. By Theorem \ref{thm:approximate_controllability} there is a sequence of functions $\{f_m\} \subset C_0^\infty((0,T)\times\cX)$ 
    such that $\norm{u^{f_m}(T)-\psi}_{H^k(N,g)} \to 0$ and $u^{f_m}$ solves \eqref{eq:cauchy_problem_infinite_time} with the interior source $f_m$. We choose a regular domain $M \subset N$ such that $M(\cX,T) \subset M$, then $\norm{u^{f_m}(T)-\psi}_{H^k(M,g)} \to 0$, and by Sobolev embedding it follows that $u^{f_m}(T)\to \psi$ in $C(M)$. In particular this means that $u^{f_m}(T,x_0) \to \psi(x_0)$. Since $\psi(x_0)\not=0$, we have $u^{f_m}(T,x_0)\not=0$ for all $m$ large enough.
\end{proof}

 \subsection{Multiplicative Gauge} \label{sub:multiplicative_gauge}
This section is focused on showing the existence of a multiplicative gauge on the lower order terms as described in Section \ref{sec:introduction}. That is, we prove in Proposition \ref{prop:existence_of_gauge} that if $\kappa:N\to\C$ is a smooth map such that $\kappa|_{\cX} = 1$ and $|\kappa| = 1$, then 
\[
\Lambda_{g,A,V} = \Lambda_{g,A+i\kappa^{-1}d\kappa,V}.
\]

\begin{comment}
To that end, suppose $\magnetic$ is a Magnetic-Schr\"odinger operator on $(N,g)$ and $\kappa\in C^\infty(N)$ which is non-vanishing. We start with two questions. First, when is the operator $\mathcal{L}_\kappa$ defined by
\begin{align*}
    \mathcal{L}_\kappa(w) := \kappa \magnetic(\kappa^{-1}w) \hspace{4mm}\text{for all $w\in C^\infty(N)$}
\end{align*}
a Magnetic-Schr\"odinger operator of the form $\mathcal{L}_{\tilde{g},\tilde{A},\tilde{V}}$. Second, when $\mathcal{L}_\kappa$ is a Magnetic-Schr\"odinger operator, what is the connection between $(g,A,V)$ and $(\tilde{g},\tilde{A},\tilde{V})$? Both of these questions are answered in the following Proposition.
\end{comment}
\begin{prop} \label{prop:gauge_operator}
    Let $(N,g)$, $A$ and $V$ be as in Theorem \ref{thm:main_thm}, and let $\magnetic$ be the respective Magnetic-Schr\"odinger operator. Let $\kappa\in C^\infty(N)$ be non-vanishing. The linear partial differential operator
    \[
    \mathcal{L}_{\kappa}(w) := \kappa \magnetic(\kappa^{-1}w), \quad w\in C^\infty(N),
    \]
    is a Magnetic-Schr\"odinger operator if and only if $\kappa$ has constant magnitude, and in this case we have that
\begin{equation}
    \label{eq:conjugated_L}
    \mathcal{L}_\kappa = \mathcal{L}_{g,A + i\kappa^{-1}d\kappa, V}.
    \end{equation}
\end{prop}

\begin{proof}
With the help of the equation (\ref{def:L}) we compute that
\begin{align}
\label{eq:eq:conjugated_L_expanded}
    \mathcal{L}_{\kappa}(w) 
    &= -\Delta_g(w) - 2i\langle A + i\kappa^{-1}d\kappa, dw\rangle_g 
    + (id^*A + |A|_g^2 + V - \kappa\Delta_g(\kappa^{-1}) + 2\langle A,i\kappa^{-1}d\kappa\rangle_g) w.
\end{align}
Then we use the linearity of the exterior derivative and the identity $\kappa^{-1} = \frac{\overline{\kappa}}{|\kappa|^2}$ to find that
    \begin{equation}
        \label{eq:Real_part_of_kappa^-1dkappa}
    \text{Re}(\kappa^{-1}d\kappa) =\frac{d(|\kappa|^2)}{|\kappa|^2}.
    \end{equation}

Suppose first that $\mathcal{L}_{\kappa}$ is a Magnetic-Schr\"odinger operator. Since $A$ is a real co-vector field, we see by comparing the first order term in \eqref{eq:eq:conjugated_L_expanded} to the first order term in \eqref{def:L} that $A + i\kappa^{-1}d\kappa$ must be a real covector field. Hence, the real part of $\kappa^{-1}d\kappa$ vanishes and the equation \eqref{eq:Real_part_of_kappa^-1dkappa} implies that $|\kappa|$ is a constant.

If $|\kappa|$ is a constant, then the equality \eqref{eq:Real_part_of_kappa^-1dkappa} implies that $A + i\kappa^{-1}d\kappa$ is a real covector. To finish the proof  it suffices to verify the equality \eqref{eq:conjugated_L}. This can be achieved by simplifying the zeroth order term in \eqref{eq:eq:conjugated_L_expanded}. Since,  $|\kappa| = C$, for some $C>0$ we have that $\kappa^{-1} = \frac{\overline{\kappa}}{C^2}$ and by the chain rule
\begin{align}
    -\langle d\kappa,d(\kappa^{-1})\rangle_g -\kappa^{-1}\Delta_g(\kappa) = \kappa\Delta_g(\kappa^{-1}) + |\kappa^{-1}d\kappa|_g^2&\label{commute}.
\end{align}
Then we use $(\ref{commute})$ to observe that 
\begin{align}
    id^*(A + i\kappa^{-1}d\kappa) 
    &= id^*A -\kappa\Delta_g(\kappa^{-1}) - |\kappa^{-1}d\kappa|_g^2\label{divergence}.
\end{align}
We are now in a position to rewrite the zeroth order term in \eqref{eq:eq:conjugated_L_expanded} by using $(\ref{divergence})$ and properties of the inner product. Hence,
\begin{align*}
    id^*A + |A|_g^2 + V - \kappa\Delta_g(\kappa^{-1}) + 2\langle A,i\kappa^{-1}d\kappa\rangle_g 
    &= id^*(A + i\kappa^{-1}d\kappa) + |A + i\kappa^{-1}d\kappa|_g^2 + V.
\end{align*}
Finally, we substitute this expression in \eqref{eq:eq:conjugated_L_expanded} and arrive in \eqref{eq:conjugated_L}.
\end{proof}

We can now describe the effect of the multiplicative gauge.  
\begin{lem} \label{lem:gauge_equivalent}
Let $(N,g)$, be a complete Riemannian manifold and $\cX \subset N$ to be open. Let $A_1$ and $A_2$ be some smooth real-valued co-vector fields on $N$ and let $V_1$ and $V_2$ be smooth real-valued functions in $N$. Let $\mathcal{L}_{g,A_i,V_i}$ be the respective Magnetic-Schr\"odinger operators. 
For each $f\in C_0^\infty((0,\infty)\times \cX)$ we denote by $u_i^f$ the solution of the Cauchy problem \eqref{eq:cauchy_problem_infinite_time} when $\mathcal{L}_{g,A,V}$ is replaced by $\mathcal{L}_{g,A_i,V_i}$.

If $\kappa:N\to \C$ is a smooth unitary function for which $\kappa|_{\cX} = 1$ then the following are equivalent:
    \begin{enumerate}
        \item $u_1^f = \kappa u_2^f$ for all $f\in C_0^\infty((0,\infty)\times \cX)$
        \item $\mathcal{L}_{g,A_1,V_1}(v) = \kappa \mathcal{L}_{g,A_2,V_2}(\kappa^{-1}v) = \mathcal{L}_{g,A_2+i\kappa^{-1}d\kappa,V_2}(v)$ for all $v\in C_0^\infty(N)$
        \item $A_1 = A_2 + i\kappa^{-1}d\kappa$ and $V_1 = V_2$.
    \end{enumerate}
\end{lem}
\begin{proof}
First we note that by Proposition \ref{prop:gauge_operator} 
% that if $\kappa$ is a smooth unitary function, then $\kappa \mathcal{L}_{g,A_2,V_2}(\kappa^{-1}v) = \mathcal{L}_{g,A_2+i\kappa^{-1}d\kappa,V_2}$ for all $v\in C_0^\infty(N)$. 
that equation \eqref{eq:conjugated_L} holds.

Condition $(3)$ clearly implies $(2)$. If $(2)$ is true then by working in local coordinates and using $(\ref{def:L})$ one can observe that $(3)$ must also hold pointwise by choosing appropriate bump functions $v\in C_0^\infty(N)$, just as was done in the proof of Lemma \ref{lem:laplace_isometry}. 
    
Suppose that $(1)$ is true and let $f\in C_0^\infty((0,\infty)\times\cX)$ be arbitrary. Then we have that
    \begin{align*}
        &(\partial_t^2 + \mathcal{L}_{g,A_2,V_2})(\kappa^{-1} u_1^f) =  (\partial_t^2 + \mathcal{L}_{g,A_2,V_2})(u_2^f) = f.
    \end{align*}
    Since $\kappa|_\cX = 1$ and $\text{supp}(f(t,\cdot)) \subset \cX$ then $\kappa f = f$. In particular it follows that $u_1^f$ is a solution to the Cauchy problem
    \begin{align*}
    &\begin{cases}
        (\partial_t^2 + \kappa\mathcal{L}_{g,A_2,V_2}(\kappa^{-1}\,\cdot))u_1^f = f, \quad \text{ in } (0,\infty) \times N
        \\
        u_1^f(0,\cdot
        ) =\partial_tu_1^f(0,\cdot) = 0
    \end{cases}  
    \end{align*}
    % where $\kappa\mathcal{L}_{g,A_2,V_2}(\kappa^{-1}\,\cdot)(v) := \kappa\mathcal{L}_{g,A_2,V_2}(\kappa^{-1}v)$.
    At the same time $(\partial_t^2+\mathcal{L}_{g,A_1,V_1})u_1^f = f$, and we conclude that
    \begin{align}
        \mathcal{L}_{g,A_1,V_1}(u_1^f) &= \kappa \mathcal{L}_{g,A_2,V_2}(\kappa^{-1}u_1^f)   \label{eq:magnetic_operators_gauge_proof}
    \end{align}
    on $(0,\infty)\times N$.
    % for all $f\in C_0^\infty((0,\infty)\times N)$. 
    
    Let $k \in \N$. If $v\in C_0^\infty(N)$ then there exists $T > 0$ such that $\text{supp}(v)\subset \text{int} M(\cX,T)$,
    and by Theorem \ref{thm:approximate_controllability}, for any $k\in\mathbb{N}$ we can find a sequence of source functions $\{f_j\}\subset C_0^\infty((0,T)\subset\cX)$ such that $u_1^{f_j}(T,\cdot) \to v$ in $H^k(N,g)$. Thus, Sobolev embedding theorem implies that for sufficiently large $k$ we can find source functions such that $\{u_1^{f_j}(T)\}$ converges to $v$ in $C^2(M(\cX,T))$. Since $\mathcal{L}_{g,A_j,V_j}$ are second order operators, and $v$ is compactly supported it then follows from the equation $(\ref{eq:magnetic_operators_gauge_proof})$ that $\mathcal{L}_{g,A_1,V_1}(v) = \kappa\mathcal{L}_{g,A_2,V_2}(\kappa^{-1}v)$. Hence we have proved that $(1)$ implies $(2)$.

    Finally, suppose that $(2)$ holds, and let $f\in C_0^\infty((0,\infty)\times \cX)$. In particular, due to Finite Speed of Wave Propagation for each $j \in \{1,2\}$ the functions $u_j^f(t,\cdot)\in C_0^\infty(N)$ are compactly supported for all $t>0$. Thus $(2)$ yields
    $$
        \mathcal{L}_{g,A_1,V_1}(u_1^f(t,\cdot)) = \kappa \mathcal{L}_{g,A_2,V_2}(\kappa^{-1}u_1^f(t,\cdot)), \quad \text{ for all } t>0.
    $$
    Since $\kappa$ is time independent, the previous equation implies that 
    % then $\partial_t^2u_1^f = \kappa\partial_t^2(\kappa^{-1}u_1^f)$. In particular we have that 
    $\kappa^{-1}u_1^f$ is a solution to the Cauchy problem
    \begin{align}
         &\begin{cases}
        \kappa(\partial_t^2 + \mathcal{L}_{g,A_2,V_2})(\kappa^{-1}u_1^f) = f, \quad \text{ in } (0,\infty) \times N
        \\
        \kappa^{-1}u_1^f(0,\cdot
        ) =\partial_t(\kappa^{-1}u_1^f(0,\cdot)) = 0.
        \end{cases} \nonumber
        \intertext{Since $\kappa|_{\cX} = 1$ and $\text{supp}(f(t,\cdot)) \subset \cX$, then $\kappa^{-1}f = f$. In particular it follows that $\kappa^{-1}u_1^f$ is a solution to the following Cauchy problem}
         &\begin{cases}
        (\partial_t^2 + \mathcal{L}_{g,A_2,V_2})(\kappa^{-1}u_1^f) = f, \quad \text{ in } (0,\infty) \times N
        \\
        \kappa^{-1}u_1^f(0,\cdot
        ) =\partial_t(\kappa^{-1}u_1^f(0,\cdot)) = 0.
        \end{cases} \label{eq:cauchy_equation_2}
        \end{align}
        But $u_2^f$ is also solves $(\ref{eq:cauchy_equation_2})$. So, by uniqueness of smooth solutions to the Cauchy problem (i.e. Lemma \ref{lem:domain_of_dependence}) we must have that $u_1^f(t,x) = \kappa(x)u_2^f(t,x)$.
\end{proof}

We can now show the existence of a multiplicative gauge for our local source-to-solution operator.

\begin{prop} \label{prop:existence_of_gauge}
    Let $\cX \subset N$ be an open subset of a complete Riemannian manifold $(N,g)$. Let $A$ be a smooth real covector field and let $V$ be a real-valued smooth function on $N$. If $\kappa \in C^\infty(N)$ is a unitary function with $\kappa|_{\cX} = 1$, then 
    $$
        \Lambda_{g,A,V} = \Lambda_{g,A+i\kappa^{-1}d\kappa,V}.
    $$
\end{prop}

\begin{proof}
The claim follows from Lemma \ref{lem:gauge_equivalent} by choosing $A_2=A$, $A_1 = A_2 + i\kappa^{-1}d\kappa$ and $V_1=V_2=V$.
    % Let $\mathcal{L}_{g,A_2,V_2}$ be the Magnetic-Schr\"odinger operator and let $\kappa\in C^\infty(N)$ be a unitary function with $\kappa|_{\cX} = 1$. If $\mathcal{L}_{1} = \mathcal{L}_{g,A_2+i\kappa^{-1}d\kappa,V_2}$, then by Lemma \ref{lem:gauge_equivalent} we have that $$u_1^f(t,x) = \kappa(x)u_2^f(t,x)$$ for all $(t,x)\in (0,\infty)\times N$. From the definition of the local source-to-solution operators and $\kappa|_{\cX} =1$, it follows that for all $f\in C_0^\infty((0,\infty)\times\cX)$ that
    % $$
    %     \Lambda_{g,A_2+i\kappa^{-1}d\kappa, V_2}(f) = u_1^f|_{(0,\infty)\times\cX} = \kappa u_2^f|_{(0,\infty)\times\cX} = u_2^f|_{(0,\infty)\times\cX} = \Lambda_{g,A_2,V_2}(f).
    % $$
    % Hence $\Lambda_{g,A_2,V_2} = \Lambda_{g,A_2+i\kappa^{-1}d\kappa,V_2}$.
\end{proof}

\section{Recovery of the Geometry} \label{sec:geometry}
In this section we show that if the local source-to-solution operators $\opnum{1}$ and $\opnum{2}$ on $(N_1,g_1)$ and $(N_2,g_2)$ are equivalent, in other words Hypothesis \ref{hyp:same_data} is true, then there exists a unique Riemannian isometry $\Phi:(N_1,g_1)\to (N_2,g_2)$ such that $\Phi|_{\cX_1} = \phi$. 

This result will be proven over four subsections. In Subsection \ref{sub:travel_time_data} we recover the travel time data from our PDE data. In Subsection \ref{sub:topology} we will use the travel time data to construct a homeomorphism $\Phi:N_1\to N_2$. In Subsection \ref{sub:smooth} we show that $\Phi$ is a diffeomorphism. Finally, in Subsection \ref{sub:geometry} we show that $\Phi$ is the only Riemannian isometry having the property $\Phi|_{\cX_1} = \phi$. Whenever we work on one manifold only we will omit the subindices.

Many results in this section were originally developed in \cite[Section 5]{source-to-solution}, and we will omit the analogous proofs, but will restate the respective claims. 

\subsection{Recovery of Travel Time Data} \label{sub:travel_time_data}
% The Boundary Control Method was originally developed in \cite{belishev_russian} for the Euclidean setting. These ideas were extended to an equivalent Spectal Problem in \cite{belishev_english} and further extended in \cite{lassas_inverse} and  \cite{kurylev_boundary_distance_map} for an inverse problem on compact Riemannian manifolds with boundary. The authors of \cite{source-to-solution} extended the ideas of the Boundary Control Method to complete manifolds without boundary whose data is given on an open set. 
The primary result of this subsection is the following Theorem.
 
\begin{thm}[Travel Time Data] 
\label{thm:travel_time_data}
    For $i \in \{1,2\}$ let $(N_i,g_i)$  be two complete Riemannian manifolds, $\cX_i \subset N_i$ be open, $A_i$ a smooth real valued co-vector field and $V_i$ a smooth real valued function on $N_i$. If Hypothesis $(\ref{hyp:same_data})$ holds, then
%    . Suppose there exists a diffeomorphism $\phi:\cX_1\to\cX_2$ such
%    \begin{align*}
%        \tilde{\phi}^*\opnum{2}(f) &= \opnum{1}(\tilde{\phi}^*f), \quad \text{ for all } f\in C_0^\infty((0,\infty)\times \cX _2)
%    \end{align*}
%    where $\tilde{\phi}(t,x) = (t,\phi(x))$. 
    \[
    \{d_{g_1}(p,\cdot)|_{\cX_1} :\: p \in N_1\}=\{d_{g_2}(\tilde{p},\phi(\cdot))|_{\cX_1}: \: \tilde{p} \in N_2\}.
    \]
\end{thm}

We begin by showing the local source-to-solution data $(\cX,\operator{})$ will determine $d_g|_{\cX\times\cX}$. It is well known property in Riemannian geometry that on any open set $\cX$ that $d_{g}:\cX\times\cX\to [0,\infty)$ will determine $g|_\cX$. Hence the local source-to-solution data $(\cX,\operator{})$ will determine the \textit{augmented data} 
\begin{align}
    (\cX,g|_\cX,d_g|_{\cX\times\cX},\operator{}). \label{def:augmented_data}
    \end{align}
\begin{lem} \label{lem:distance_function_data}
    % The data $(\cX,\operator{})$ determines the distance function $d_g$ on $\cX\times\cX$. 
    If Hypothesis \ref{hyp:same_data} is satisfied, then for all $x,y\in \cX_1$ we have that 
    \begin{equation}
    \label{eq:metric_isometry_on_X}
    d_{g_1}(x,y) = d_{g_2}(\phi(x),\phi(y))\hspace{6mm}\text{and} \hspace{6mm} g_1|_{\cX_1} = \phi^*(g_2|_{\cX_2}).
    \end{equation}
    In particular, the map $\phi\colon (\cX_1,g_1) \to (\cX_2,g_2)$ is a Riemannian isometry.
\end{lem}

\begin{proof}
%See \cite[Lemma 17]{source-to-solution} for the proof that the local source-to-solution map $\opnum{}$ determines the distance $d_g(x,y)$ for all $x,y \in \cX$.
%\teemu{The rest of the proof needs to be included and some parts of the previous proofs may need to be incorporated in here.}
% First we consider the case with just one manifold. Let $x,y\in\cX$ be arbitrary. As $\cX$ is a smooth manifold, it is metrizable so there exists a metric $d_0$ which might be different from $d_g$ but induces the same topology on $\cX$. Let $\eps > 0$ be arbitrary and define $\cB_\eps(x) := (0,\infty)\times B_{d_0}(x,\eps) \subset (0,\infty)\times \cX$. We now define
%     \begin{align}
%         t_\eps := \inf\{t > 0 : \exists\,\text{$f\in C_0^\infty(\cB_\eps(x))$ such that supp$(\operator{}f(t,\cdot))\cap B_{d_0}(y,\eps)\not=\emptyset$}\}.\label{eq:t_eps}
%     \end{align}
%     It is shown in \cite[Lemma 17]{source-to-solution} that $\lim_{\eps\to 0} t_\eps = d_g(x,y)$.

Due to the Hypothesis \ref{hyp:same_data} there is a diffeomorphism $\phi\colon \cX_1 \to \cX_2$ such that 
\\
$\tilde{\phi}^*(\opnum{2}f) = \opnum{1}(\tilde{\phi}^*f)$ holds for all $f\in C_0^\infty((0,\infty)\times\cX_2)$. Here $\tilde{\phi}(t,x)=(t,\phi(x))$, Since $\phi:\cX_1\to\cX_2$ is a diffeomorphism then $\tilde{\phi}^*:C_0^\infty((0,\infty)\times\cX_2)\to C_0^\infty((0,\infty)\times\cX_1)$ is an isomorphism where $\tilde{\phi}^*(f) = f\circ\tilde{\phi}$. 

We proceed by choosing  a metric $d_2$ on $\cX_2$ that induces the same topology as $d_{g_2}$. Then we define $d_1:\cX_1\times\cX_1\to [0,\infty)$ by $d_1(x,y) = d_2(\phi(x),\phi(y))$. Since $\phi:\cX_1\to \cX_2$ is a homeomorphism, then $d_1$ defines a metric on $\cX_1$ which induces the same topology as $d_{g_1}$ and $\phi:(\cX_1,d_1)\to (\cX_2,d_2)$ is a metric isometry. 

Now fix $\eps>0$, $x_1,y_1\in \cX_1$ and denote $x_2 = \phi(x_1)$ and $y_2 = \phi(y_1)$. We also set $\mathcal{B}_1(\eps) = (0,\infty)\times B_{d_1}(x_1,\eps)$ and $\mathcal{B}_2(\eps) = (0,\infty)\times B_{d_2}(x_2,\eps)$, and note that $\tilde{\phi}^*:C_0^\infty(\mathcal{B}_2(\eps))\to C_0^\infty(\mathcal{B}_1(\eps))$ is a vector space isomorphism. 
% We define the numbers $t_\eps^1, t_\eps^2$ as $t_\eps$ in \eqref{eq:t_eps} but setting $\operator{}=\operator{i}$ and $d_0=d_i$. 
%Recall that one has the limit $\lim_{\eps\to\infty} t_\eps = d_{g_1}(x,y)$. Let $s_\eps$ be the similarly defined number such that $\lim_{\eps\to0}s_\eps = d_{g_2}(\phi(x),\phi(y))$. 
% From the (\ref{eq:t_eps}), we have that
Also, for both $i \in \{1,2\}$ we proceed to define 
\begin{align*}
    t_{i,\varepsilon}: = \inf\{s > 0: \exists\, f\in C_0^\infty(\cB_i(\eps)) \text{ such that }\text{supp}(\opnum{i}f (s,\cdot))\cap B_{d_i}(y_i,\eps)\not=\emptyset\}.
\end{align*}

Then we use the fact that $\tilde{\phi}^*\opnum{2}(f) = \opnum{1}(\tilde{\phi}^*f)$ for all $f\in C_0^\infty((0,\infty)\times \cX _2)$, and note due to  the definition of $t_{i,\eps}$ it follows that $t_{1,\eps} = t_{2,\eps}$. Finally, from the proof of \cite[Lemma 17]{source-to-solution} we have that 
    $$
    d_{g_1}(x,y) = \lim_{\eps\to 0} t_{1,\eps} = \lim_{\eps\to 0} t_{2,\eps} = d_{g_2}(\phi(x),\phi(y)).
    $$
     Hence, the equation \eqref{eq:metric_isometry_on_X} follows.
\end{proof}

Through out this section we will use the following facts from Riemannian geometry that follow from the Hopf-Rinow Theorem.

\begin{lem} \label{lem:boundary_properties}
    Let $(N,g)$ be a connected and complete Riemannian manifold. Let $p\in N$ and $r > 0$, then the following hold.
    \begin{enumerate}
        \item $\overline{B(p,r)} = \{x\in N:d_g(x,p)\leq r\}$    
        \item $\partial B(p,r) = \{x\in N: d_g(x,p)=r\}$
        \item $\text{dist}_g(p,A) = \text{dist}_g(p,\overline{A})$ for any $A\subset N$
        \item $M(A,r) = M(\overline{A},r)$
        \item $M(B(p,s),r) = \overline{B(p,s+r)}$.
    \end{enumerate}
\end{lem}

For each $0 < \eps < r < T$, and $x \in N$ we define the following space-time cylinder
$$
    S_\eps(x,r) := (T-(r-\eps),T)\times B(x,\eps).
$$
Note that if $f \in C_0^\infty(S_\eps(x,r))$, then by Finite Speed of Wave Propagation and Lemma \ref{lem:boundary_properties} we have that $\text{supp}(u^f(T)) \subset B(x,r)$. We recall that in here $u^f$ stand for the unique solution of the Cauchy problem \eqref{eq:cauchy_problem_infinite_time} for the source $f$. By adjusting the space-time cylinder $S_\eps(x,r)$ we will have control of the support of the solutions. This can be used to extract metric information.

\begin{lem}\label{lem:main_connection}
    Let $p,y,z \in N$ and let $0 < \eps < l_p,l_y,l_z < T$. Then the following conditions are equivalent:
    \begin{enumerate}
        \item $B(p,l_p) \subset \overline{B(y,l_y) \cup B(z,l_z)}$
        \item  For all $f \in C_0^\infty(S_\eps(p,l_p))$ there exists a sequence of functions $\{f_j\} \subset C_0^\infty(S_\eps(y,l_y)\cup S_\eps(z,l_z))$ such that
        $$ 
        \lim_{j\to \infty} \norm{u^f(T) - u^{f_j}(T)}_{L^2(N,g)} = 0. 
        $$
    \end{enumerate}
    If in addition $B(x,\eps),B(y,\eps),B(z,\eps) \subset \cX$ then the following is also equivalent to the former two statements
    \begin{enumerate}[resume]
        \item For all $f \in C_0^\infty(S_\eps(p,l_p))$
        $$
            \inf\Big\{ \langle f-h,  J\operator{}(f-h)\rangle - \langle \operator{}(f-h),  J(f-h)\rangle     \Big|\,h\in C_0^\infty(S_\eps(y,l_y)\cup S_\eps(z,l_z))\Big\} = 0
        $$
        where $\langle u,v\rangle = \int_0^T\int_{\cX} u(t,x)\overline{v(t,x)}dV_gdt$.
    \end{enumerate}
\end{lem}
\begin{proof} See \cite[Lemma 10]{source-to-solution} for the original proof.
\end{proof}

\begin{defn}
\label{def:cut_time}
    Let $(y,\xi)\in SN$. Let $\gamma_{y,\xi}$ be the unique geodesic on $(N,g)$ with initial conditions $\gamma(0) = y$ and $\dot{\gamma}(0) = \xi$. Then the respective cut time is defined as 
    % $\tau(y,\xi)$ is defined to be the largest time $T$ such that $\gamma|_{[0,T]}$ is a distance minimizing curve. In particular that
    $$
        \tau(y,\xi) = \sup\{ t> 0: d_g(y,\gamma(t)) = t\}.
    $$
    %It is possible that $\tau(y,\xi)=+\infty$.
\end{defn}
In the next lemma we show that $\tau$ can be described via inclusions of certain metric balls.
% in the next Lemma and then use the equivalent conditions in Lemma \ref{lem:main_connection} to connect $\tau$ to the augmented data. The next  lemma holds on any complete Riemannian manifold $(N,g)$.
\begin{lem}\label{lem:cut_time_lemma}
    Let $(N,g)$ be a complete Riemannian manifold. Let $x,y \in N$ and let $\gamma_{y,\xi}$ be a distance minimizing geodesic from $y$ to $x$ with initial velocity $\xi \in S_yN$. Let $s:= d_g(x,y)$ and let $r > 0$. If $\tau(y,\xi) < s + r$, then 
    \begin{align}
        \text{there exists $\eps > 0$ such that}\,\,B(x,r+\eps) \subset \overline{B(y,s+r)} \label{lem:single_ball_containment}
    \end{align}
    Further, if $(\ref{lem:single_ball_containment})$ is satisfied for some $r$ and $s$, then $\tau(y,\xi)\leq s+r$. Moreover, if $x,y\in N$, $\xi \in S_yN, s = d_g(x,y) > 0$ and $x = \gamma_{y,\xi}(s)$, then
    \begin{align}
        \tau(y,\xi) = \inf\{s + r\,:\, r > 0\, \text{and } B(x,r+\eps)\subset\overline{B(y,s+r)}\text{ for some $\eps > 0$}\}. \label{eq:tau_characterization}
    \end{align}
\end{lem}
\begin{proof}
    See \cite[Lemma 11]{source-to-solution} for the original proof.
\end{proof}
Our focus now shifts to showing that if Hypothesis \ref{hyp:same_data} is valid, then the respective travel time data are equivalent in the sense of Theorem \ref{thm:travel_time_data}. We do this by finding explicit relationships between certain metric balls on $(N_1,g_1)$ and $(N_2,g_2)$, then we show how the cut times $\tau_1$ and $\tau_2$ are related via $\phi$.
\begin{lem}\label{lem:same_sets}
    Suppose that Hypothesis $(\ref{hyp:same_data}
    )$ holds.
%    $\phi:\cX_1\to\cX_2$ is a diffeomorphism such that 
%    \begin{align*}
%        \Lambda_{\cX_1, A_1,V_1}(f\circ\Tilde{\phi}) &= [\Lambda_{\cX_2, A_2,V_2}(f)]\circ\Tilde{\phi}, \quad \text{ for all } f\in C_0^\infty((0,\infty)\times \cX_2).
%    \end{align*}    
    If $x,y,z\in \cX_1$ and $l_x,l_y,l_z > 0$, then
    \begin{align*}
        B_1(x,l_x) &\subset \overline{B_1(y,l_y) \cup B_1(z,l_z)}
        \intertext{if and only if}
        B_2(\phi(x),l_x) &\subset \overline{B_2(\phi(y),l_y) \cup B_2(\phi(z),l_z)} 
    \end{align*}
    where $B_i(p,r) = \{ q\in N_i : d_{g_i}(p,q) < r\}$.
\end{lem}
% \begin{rem}
%     The metric balls are computed with respect to the induced metric on $(N_i,g_i)$ rather than the distance metric induced on $(\cX_i,g_i|_{\cX_i})$ which is not a complete Riemannian manifold.
% \end{rem}
% \teemu{I do not understand this remark.s}
\begin{proof}
    Since $x,y,z\in\cX_1$, then $\phi(x),\phi(y),\phi(z)\in \cX_2$. Since $\cX_1$ is an open set, then we can find an $0 < \eps < l_x,l_y,l_z$ such that $B_1(x,\eps),B_1(y,\eps),B_1(z,\eps) \subset\cX_1$. By Lemma \ref{lem:distance_function_data} we have that
    $$
        \phi(B_1(x,\eps)) = B_2(\phi(x),\eps) \subset \cX_2. 
    $$
    Now let $T > l_x,l_y,l_z$, then under a slight abuse of notation we have that
    \begin{align*}
        \tilde{\phi}(S_\eps(x,l_x)) = (T-(l_x-\eps),T)\times \phi(B_1(x,\eps)) = S_\eps(\phi(x),l_x).
    \end{align*}
    Identical statements hold for the other two space-time cylinders. Our plan to complete proof is based on the equivalent statements in Lemma \ref{lem:main_connection}.\\ \newline    
    Since $\tilde{\phi}:S_\eps(x,l_x)\to S_\eps(\phi(x),l_x)$ is a diffeomorphism, then the pullback  $\tilde{\phi}^*:C_0^\infty(S_\eps(\phi(x),l_x)) \to C_0^\infty(S_\eps(x,l_x))$ defined by $\tilde{\phi}^*h := h\circ\tilde{\phi}$ is a bijection. If we can show that 
    \begin{align}
    \label{eq:transformation_of_Blago_id}
        \langle \tilde{\phi}^*f-\tilde{\phi}^*h,  J\Lambda_1(\tilde{\phi}^*f-\tilde{\phi}^*h)\rangle_1 &- \langle \Lambda_1(\tilde{\phi}^*f-\tilde{\phi}^*h),  J(\tilde{\phi}^*f-\tilde{\phi}^*h)\rangle_1
        \nonumber
        \\
        &= 
        \\
        \langle f-h,  J\Lambda_2(f-h)\rangle_2 &- \langle \Lambda_2(f-h),  J(f-h)\rangle_2     
        \nonumber
    \end{align}
    for all $f \in C_0^\infty(S_\eps(\phi(x),l_x))$ and for all $h \in C_0^\infty(S_\eps(\phi(y),l_y)\cup S_\eps(\phi(z),l_z))$ where 
    $$
        \langle u,v \rangle_i = \int_0^T \int_{\cX_i} u\overline{v} dV_{g_i}dt.
    $$
    By Lemma \ref{lem:main_connection} we will be able to conclude that
     \begin{align*}
        B_1(x,l_x) &\subset \overline{B_1(y,l_y) \cup B_1(z,l_z)}
        \intertext{if and only if}
        B_2(\phi(x),l_x) &\subset \overline{B_2(\phi(y),l_y) \cup B_2(\phi(z),l_z)} .
    \end{align*}
    By Lemma \ref{lem:distance_function_data} we know that $g_1|_{\cX_1} = \phi^*(g_2|_{\cX_2})$, so 
    $$
        \int_{\cX_1} (\rho\circ\phi) dV_{g_1} = \int_{\cX_2} \rho dV_{g_2},
    $$ 
    where 
    % $dV_{g_i}$ is the smooth density induced by $g_i$ and
    $\rho\in C_0^\infty(\cX_2)$. At the same time, by Hypothesis \ref{hyp:same_data} we have that $\tilde{\phi}^*\Lambda_2(f) = \Lambda_1(\tilde{\phi}^*f)$ for all $f \in C_0^\infty((0,\infty)\times\cX_2)$. Furthermore $\tilde{\phi}^* J (f) = J(\tilde{\phi}^* f)$ where $$J(f)(t,x) = \frac{1}{2}\int_t^{2T-t} f(s,x)ds.$$
    Combining these observations yields
    \begin{align*}
        \langle f-g,J\Lambda_2(f-g)\rangle_2 &=\langle \tilde{\phi}^*f-\tilde{\phi}^*g, J\Lambda_1(\tilde{\phi}^*f- \tilde{\phi}^*g)\rangle_1
        \intertext{and}
        \langle \Lambda_2(f-g),J(f-g)\rangle_2 &= \langle \Lambda_1(\tilde{\phi}^*f-\tilde{\phi}^*g), J(\tilde{\phi}^*f-\tilde{\phi}^*g)\rangle_1.    
    \end{align*}
These identities prove \eqref{eq:transformation_of_Blago_id}.
\end{proof}
Now we will show that the cut time functions will be equivalent under the diffeomorphsim $\phi$ if the local source-to-solution operators are equivalent in sense of the Hypothesis \ref{hyp:same_data}.
\begin{lem}\label{lem:same_cut}
    Suppose that Hypothesis $(\ref{hyp:same_data})$ holds.
%    $\phi:\cX_1\to\cX_2$ is a diffeomorphism such that 
%    \begin{align*}
%        \opnum{1}(f\circ\Tilde{\phi}) &= [\opnum{2}]\circ\Tilde{\phi}, \quad \text{ for all } f\in C_0^\infty((0,\infty)\times \cX_2).
%    \end{align*}
    Let $\tau_i$ be the cut time function on $(N_i,g_i)$. If $y\in\cX_1$ and $\xi\in S_yN_1$, then 
    $$
        \tau_1(y,\xi) = \tau_2(\phi(y),d\phi_y\xi).
    $$
\end{lem}

\begin{proof}
    Let $\gamma$ be the unique geodesic on $(N_1,g_1)$ with initial conditions $\gamma(0) = y$ and $\dot{\gamma}(0) = \xi$. Since $\cX_1$ is an open set, then we can find $s > 0$ such that $B_1(y,s) \subset \cX_1$ and $s < \tau_1(y,\xi)$. We denote $x = \gamma(s)$. By Lemma \ref{lem:cut_time_lemma} we have that
    $$
    \tau_1(y,\xi) = \inf\{s + r: r > 0\,\text{such that there exists $\eps > 0$ s.t. } B_1(x,r + \eps) \subset \overline{B_1(y,r+s)}\}.
    $$
    Since $\phi:\cX_1\to\cX_2$ is a metric and Riemannian isometry by Lemma \ref{lem:distance_function_data}, then $\tilde{\gamma}|_{[0,s]} = \phi\circ\gamma|_{[0,s]}$ is a distance minimizing geodesic with respect to $(\cX_2,g_2)$ with initial conditions $\tilde{\gamma}(0) = \phi(y)$ and $\dot{\tilde{\gamma}}(0) = d\phi_y\xi$. Hence $B_2(\phi(y),s) \subset \cX_2$ and $\tilde{\gamma}|_{[0,s]}$ is also a distance minimizing curve connecting $\phi(y)$ to $\phi(x)$. From Lemma \ref{lem:cut_time_lemma} it follows that
    \begin{align*}
        \tau_2(\phi(y),d\phi_y\xi) = \inf\{&s + \tilde{r} : \tilde{r} > 0 \,\text{and there exists $\eps > 0$ such that }
        \\
        &\hspace{12mm}B_2(\phi(x), \tilde{r} + \eps) \subset \overline{B_2(\phi(y),\tilde{r} + s)}\}.
    \end{align*}
    Moreover, by Lemma \ref{lem:same_sets} we have that
    $$
        B_1(x,r + \eps) \subset \overline{B_1(y,r+s)}\,\,\,\text{if and only if}\,\,\,B_2(\phi(x),r + \eps) \subset
         \overline{B_2(\phi(y),r+s)}. $$
    Hence, $\tau_1(y,\xi) = \tau_2(\phi(y),d\phi_y\xi)$.
\end{proof}

We are now in a position to prove Theorem \ref{thm:travel_time_data}.  
\begin{proof}[Proof of Theorem \ref{thm:travel_time_data}]
    For any $y\in\cX_1$ and $\xi\in S_yN_1$ by Lemma \ref{lem:same_cut}, we have that $\tau_1(y,\xi) = \tau_2(\phi(y),d\phi_y\xi)$. Let $\gamma = \gamma_{y,\xi}$ be a geodesic on $N_1$ with initial condition $y$ and initial velocity $\xi$. We recall that it was proven in \cite[Proposition 2]{source-to-solution} that if $r < \tau_1(y,\xi)$ and if $0 < s < r$ is small enough that $\gamma|_{[0,s]} \subset \cX_1$  and $z = \gamma(s)$ then for any $x\in \cX_1$ we have that
    \begin{align*}
        d_{g_1}(\gamma(r),x) &= \inf\{R > 0 : \text{there is $\eps > 0$ such that } B_1(z, r-s+\eps) \subset \overline{B_1(y,r)\cup B_1(x,R)}\}. 
    \end{align*}

    Now since $\tau_1(y,\xi) = \tau_2(\phi(y),d\phi_y\xi)$ and if we let $\tilde{\gamma}|_{[0,s]} = \phi\circ\gamma|_{[0,s]}$, then
    \begin{align*}
            d_{g_2}(\tilde{\gamma}(r),\phi(x)) = \inf\{&\hat{R} > 0 : \text{there is $\eps > 0$ such that }\\
        &B_2(\phi(z), r-s+\eps) \subset \overline{B_2(\phi(y),r)\cup B_2(\phi(x),\hat{R})}\}. 
    \end{align*}
    Thus, by Lemma \ref{lem:same_sets} we conclude that
    $$
        d_{g_1}(\gamma(r),x) = d_{g_2}(\tilde{\gamma}(r),\phi(x)).
    $$
By \cite[Lemma 12]{source-to-solution}, 
% Uncomment for dissertation
%Lemma \ref{lem:geodesic_reach}
for any $p \in N_1$ we can find $y\in \cX_1$ and $\xi \in S_yN_1$ such that $\gamma_{y,\xi}(r) = p$ and $r< \tau_1(y,\xi)$. Thus for $\tilde{p} = \tilde{\gamma}(r)$ and for all $x\in \cX_1$ we have that
    $$
        d_{g_1}(p,x) = d_{g_2}(\tilde{p},\phi(x)).
    $$
    Therefore,
    \[
    \{d_1(p,\cdot)|_{\cX_1} :\: p \in N_1\} \subset\{d_2(\tilde{p},\phi(\cdot))|_{\cX_1}: \: \tilde{p} \in N_2\}.
    \]
    The reverse inclusion is also true since $\phi$ is a diffeomorphism and the manifolds $N_1$ and $N_2$ are in a symmetric position.
\end{proof}

\subsection{Recovery of Topology} \label{sub:topology}
In this subsection we construct a homeomorphism from $N_1$ to $N_2$ using the augmented data $(\ref{def:augmented_data})$. Let $U_1\subset \cX_1$ be open and bounded such that $\overline{U}_1 \subset\cX_1$. We let $U_2 = \phi(U_1) \subset \cX_2$. Then $\overline{U}_i$ is compact and $C(\overline{U}_i)$ with the supremum topology is a Banach space. For each $U_i$ we define the respective \textit{travel time maps} $R_i:(N_i,g_i)\to (C(\overline{U}_i),\|\cdot\|_\infty)$ by
\begin{align}
    R_i(p) = d_{g_i}(p,\cdot)|_{\overline{U}_i} \label{def:R_i}.
\end{align}
The image of the travel time map is the \textit{travel time data} 
$$
R_i(N_i) = \{d_{g_i}(p,\cdot)|_{\overline{U}_i}:p\in N_i\} \subset C(\overline{U}_i).
$$

We give the travel time data the subspace topology that it inherits from the supremum topology on $C(\overline{U}_i)$. 
% By Proposition \ref{prop:2} we know that the map can be determined by the augmented data $(\cX,g|_\cX,d_g|_{\cX\times \cX},\operator{})$. 
The first part of this subsection is dedicated to proving that the travel time map is a homeomorphism between the manifold and its travel time data. Once we show the travel time map is a homeomorphism we construct a homeomorphism from $(N_1,g_1)$ to $(N_2,g_2)$ using the travel time maps and the diffeomorphism $\phi$ as in Hypothesis \ref{hyp:same_data}. Next result was originally given as \cite[Lemma 13]{source-to-solution} with an additional assumption of $U$ having a smooth boundary. Here we drop this assumption and provide a new proof.

\begin{lem}\label{lem:travel_continuous}
    If $(N,g)$ is a complete Riemannian manifold and $U \subset N$ is open and bounded then the respective travel time map $R:(N,g)\to (C(\overline{U}),\norm{\cdot}_\infty)$ is a continuous injection.
\end{lem}

\begin{proof}
    Let $p,q\in N$ be given. By the reverse triangle inequality we have the estimate
    \begin{align*}
        \| R(p) - R(q)\|_\infty = \sup_{z\in \overline{U}}|d_g(p,z) - d_g(q,z)| \leq d_g(p,q).
    \end{align*}
    Thus the travel time map $R:N\to C(\overline{U})$ is $1-$Lipschitz.
    
    Suppose that for some $p,q \in N$ we have that $R(p)=R(q)$. Let $x\in U$ and $s := d_g(p,x)$, then there exists a distance minimizing geodesic $\gamma:[0,s]\to N$ such that $\gamma(0) =x$ and $\gamma(s) = p$. Since $U$ is open there exists $t\in (0,s)$ such that $\gamma([0,t]) \subset U$. We denote $\hat{x} = \gamma(t)$, and note that since $\gamma|_{[0,s]}$ is distance minimizing we have that
    \begin{equation}
        \label{eq:length_of_the_con_curve}
        d_g(q,x) = d_g(p,x) = d_g(p,\hat x) + d_g(\hat x,x) = d_g(q,\hat{x}) + d_g(\hat{x},x).
    \end{equation}
    Now let $\alpha:[t,s]\to N$ be a distance minimizing geodesic from $\hat{x}$ to $q$ and consider the concatenated curve $\hat{\gamma} = \alpha\circ \gamma|_{[0,t]}$ which, due to the equation \eqref{eq:length_of_the_con_curve}, is a distance minimizing curve from $q$ to $x$. Since distance minimizing curves are geodesics then $\gamma$ and $\hat{\gamma}$ are geodesics which coincide on $[0,t]$. However, geodesics on a complete manifold that are equal on an interval must agree everywhere. In particular $p = \gamma(s) = \hat{\gamma}(s) = q$. Hence $R:(N,g)\to (C(\overline{U}),\|\cdot\|_\infty)$ is continuous and injective. 
\end{proof}

\begin{comment}
To show that the inverse map is continuous we will need the following notions from topology.

\begin{defn}
    Let $X$ be a topological space. A sequence $\{x_n\}$ in $X$ is said to escape to infinity if for any compact set $K$ we have that $x_n\in K$ for only a finite number of $\{x_n\}$.
\end{defn}

\begin{defn}
    A map $f:(X,d_X)\to (Y,d_Y)$ between metric spaces is said to be proper if $f^{-1}(K) \subset X$ is compact for any compact set $K \subset Y$.
\end{defn}
For proofs of the following lemmas see \cite{willard2012general}.

\begin{lem} \label{lem:escape_infinity}
    Let $(X,d_X)$ and $(Y,d_Y)$ be metric spaces and let $f:X\to Y$ be continuous. Then $f$ is proper if and only if for every sequence $\{x_n\} \subset X$ that escapes to infinity the image sequence $\{f(x_n)\} \subset Y$ escapes to infinity.
\end{lem}

\begin{lem} \label{lem:closed_map}
    Let $(X,d_X)$ and $(Y,d_Y)$ be metric spaces. If $f:X\to Y$ is injective, continuous and proper, then $f$ is a closed map.
\end{lem}
\end{comment}
\begin{prop} \label{prop:travel_homeomorphism}
    The travel time map \eqref{def:R_i} is a topological embedding.
\end{prop}

\begin{proof}
    See \cite[Proposition 3]{source-to-solution} for the original proof. 
\end{proof}

% Uncomment for dissertation
\begin{comment}
    By Lemma \ref{lem:travel_continuous} we know that $R:(N,g) \to (C(\overline{U}),\|\cdot\|_\infty)$ is a continuous injection. To finish the proof we show that $R$ is proper and thus a closed map. Since every continuous map that is injective and closed is a topological embedding, we will be done.\\ \newline
    If $N$ is bounded, then as a consequence of the Hopf-Rinow Theorem $N$ is compact. On a compact set there are no sequences that escape to infinity, hence Lemma \ref{lem:escape_infinity} and Lemma \ref{lem:closed_map} are trivially satisfied. Next assume that $N$ is unbounded and let $\{x_n\} \subset N$ be a sequence that escapes to infinity. 
%    Let $x_0\in \overline{U}$, define $X_j = \overline{B(x_0,j)}$ for every $j\in\N$ and define $Y_j = R(X_j)$. Since $N$ is connected, then $\bigcup_{j=1}^\infty X_j = N$ and $\lim_{j\to\infty} d_g(x_0,x_j) = \infty$ since $\{x_n\}$ escapes to infinity.\\ \newline
    Let $x_0\in U$, then $d_g(x_0,x_j) \to +\infty$ as $j\to +\infty$. Denote $r_0 = R(x_0)$ and $r_j = R(x_j)$, then since $x_0\in \overline{U}$ we have that
    \begin{align*}
        d_g(x_0,x_j) &= | d_g(x_0,x_0) - d_g(x_0,x_j)|\leq \sup_{z\in \overline{U}} |d_g(x_0,z) - d_g(x_j,z)| =d_\infty(r_0,r_j).
    \end{align*}
    It follows that $d_\infty(r_0,r_j) \to \infty$. Every compact set on a metric space is bounded, thus $\{r_j\}$ escapes to infinity. Therefore Lemma \ref{lem:escape_infinity} and Lemma \ref{lem:closed_map} are satisfied. Thus $R:(N,g)\to (R(N),\|\cdot\|_\infty)$ is a closed map.
\end{comment}

We are now in a position to show that $N_1$ and $N_2$ are homeomorphic and construct the homeomorphism provided that Hypothesis $(\ref{hyp:same_data})$ is satisfied.
\begin{prop} \label{prop:phi_homeomorphism}
    Suppose that Hypothesis $\ref{hyp:same_data}$ is satisfied with respect to a diffeomorphism $\phi\colon \cX_1 \to \cX_2$ and let $U_1 \subset\cX_1$ be an open and bounded set such that $\overline{U}_1\subset \cX_1$. Let $\Psi:C(\overline{U}_1) \to C(\phi(\overline{U}_1))$ be defined by $\Psi(w) = w\circ\phi^{-1}$. The map
    \begin{align}
         \Phi_{U_1} := R_2^{-1}\circ \Psi \circ R_1 :N_1 \to N_2   \label{eq:Phi_U}
    \end{align}
    is a homeomorphism. Furthermore, for all $(p,x)\in N_1\times\overline{U}_1$ we have that  
    \begin{align}
        d_{g_1}(p,x) &= d_{g_2}(\Phi_{U_1}(p),\phi(x)) \hspace{4mm}\text{and}\hspace{6mm} \Phi_{U_1}|_{U_1} = \phi|_{U_1}. \label{eq:quasi_isometry}
    \end{align}
\end{prop}

\begin{proof}
    By Proposition \ref{prop:travel_homeomorphism}, $R_i:(N_i,g_i)\to (R_i(N_i),\|\cdot\|_\infty)$ is a homeomorphism for $i \in\{1,2\}$. From Theorem \ref{thm:travel_time_data} we have that $\Psi:R_1(N_1)\to R_2(N_2)$ is a homeomorphism since $\phi$ is a bijection and $\overline{U}_1$ is compact. Hence, $\Phi_{U_1}$ is a homeomorphism.

    Next we prove that \eqref{eq:quasi_isometry} holds. Fix $p\in N_1$, then by the definition of $R_i$ and $\Phi_{U_1}$ we have that
    \begin{align*}
        d_{g_2}(\Phi_{U_1}(p),\cdot)|_{\overline{U}_2} = R_2(\Phi_{U_1}(p)) = (\Psi\circ R_1)(p) = d_{g_1}(p,\phi^{-1}(\cdot))|_{\overline{U}_2}
    \end{align*}
    Hence, the first part of \eqref{eq:quasi_isometry} follows. The second part of \ref{eq:quasi_isometry} follows from the first since
    $$
        d_{g_2}(\Phi_{U_1}(x),\phi(x)) = d_{g_1}(x,x) = 0.
    $$
    for all $x\in\overline{U}$.
\end{proof}

% \begin{rem}
%     The set $U_1$ does not need to be connected.
% \end{rem}

  \subsection{Recovery of Smooth Structure} \label{sub:smooth}  
%In the augmented data (\ref{def:augmented_data}), we have knowledge of the topological and smooth structures of $\cX_i$ and thus know the smooth and topological information of $\overline{U}_i$. By knowledge of the smooth structure on $\overline{U}_i$ we mean that given a function $f\in C(\overline{U}_i)$ and a point $z\in \overline{U}_i$ we can determine if $f$ is smooth in an open neighborhood of $z$. Because of how the distance function $d_g$ is constructed, we can recover smooth atlases by studying when $d_g$ is and is not smooth. 
In this subsection we prove that $\Phi_{U_1}$ as defined in $(\ref{eq:Phi_U})$ is a diffeomorphism. We will need the following definitions from Riemannian geometry. First, the \textit{injectivity domain} at any point $p\in N$ is the open subset of $T_pN$ defined by
\begin{align*}
    ID(p) :&= \{ v\in T_pN:\,d_g(p,\text{exp}_p(tv)) = t\|v\|_g\text{ for all $t\in [0,1)$}\}\\
    &= \{v\in T_pN:\, \|v\|_g < \tau(p,v/\|v\|_g)\},
\end{align*}
c.f. \cite[Chapter 9 Proposition 20]{petersen2006riemannian}. Further, we define the \textit{cut locus} of the point $p$ to be the set
\begin{equation}
    \label{eq:Cut_locus}
    \text{Cut}(p) := \text{exp}_p(\p ID(p)) 
    =\exp_{p}(v \in T_pN: \: \|v\|_g=\tau(p,v/\|v\|_g)) \subset N.
\end{equation}
Therefore, the restriction of the exponential map at $p$ to the respective  injectivity domain, i.e. the map $\text{exp}_p:ID(p)\to N\backslash\text{Cut}(p)$ is a diffeomorphism, see \cite[Theorem 10.34]{lee_riemannian}. As the next lemma demonstrates, we can characterize the injectivity domain via smoothness of distance functions.

\begin{lem} \label{lem:dist_gradient}
    Let $p,z\in N$ and $\eta\in T_zN$. The following are equivalent
    \begin{enumerate}
        \item $\eta\in ID(z)\backslash\{0\}$ and $\text{exp}_z(\eta) = p$
        \item $d_g(p,\cdot)$ is smooth in an open neighborhood of $z$ and $\eta = - d_g(p,z)\nabla_gd_g(p,\cdot)|_z$.
    \end{enumerate}
\end{lem}

\begin{proof}
    This was proven in \cite[Lemma 16]{source-to-solution}.
\end{proof}

Let $z\in U$ be given and define
\begin{align}
    W_z = \{r \in R(N): \text{ $r$ is smooth in an open neighborhood of $z$ or $r(z)=0$}\} \subset C(\overline{U}) \label{def:W_z}
\end{align}
So that $W_z = R(N\backslash \text{Cut}(z))$ and $W_z$ is open since $R$ is a homeomorphism and $\text{Cut}(z)$ is closed. Then we define a function
$$
\rho_z:W_z \to ID(z), \quad    \rho_z(r) :=  -r(z)\nabla_gr(z).
$$
Here we interpret $\rho_z(R(z))=0$.
Due to the augmented data \eqref{def:augmented_data} we know $g|_U$ and can determine the gradient of a smooth function on the open set $U$. It follows that $W_z$ and $\rho_z$ are determined by our data. Using the map $\rho_z$, Lemma \ref{lem:dist_gradient} and $\text{exp}_z:ID(z)\to N\backslash \text{Cut}(z)$ being a diffeomorphism, we find that
\begin{align}
    \text{exp}_z\circ \rho_z \circ R:N\backslash(\text{Cut}(z))\to N\backslash(\text{Cut}(z)) \label{eq:identity}
\end{align}
is the identity map, and as a result the map is a bijection.
\begin{equation}
    \label{eq:rho_z_via_exp_and_E}
    \rho_z = \text{exp}_z^{-1}\circ R^{-1} : W_z \to \text{ID}(z).
\end{equation}

We are ready to recover the smooth structure.

\begin{prop} \label{prop:phi_diffeomorphism}
    Let $\cX_1,U_1,\Psi$ and $\Phi|_{U_1}$ be as in Proposition \ref{prop:phi_homeomorphism}. If Hypothesis $\ref{hyp:same_data}$ is satisfied, then $\Phi_{U_1}$ is a diffeomorphism. In particular, if $z\in U_1$ then $\Phi_{U_1}:N_1\backslash\text{Cut}(z) \to N_2\backslash\text{Cut}(\phi(z))$ is given by the composition
    \begin{align}
        \Phi_{U_1} = \exp_{\phi(z)}\circ d\phi_z\circ \exp_z^{-1}. \label{eq:phi_u_linear}
    \end{align}
\end{prop}
\begin{proof}
    If $z\in U_1$ then $\phi(z)\in U_2:=\phi(U_1)$. By the definition of $W_z$ in equation (\ref{def:W_z}) we have that $r\in W_z$ if and only if $r\circ \phi^{-1}\in W_{\phi(z)}$ and so $\Psi:W_z\to W_{\phi(z)}$ is a well-defined map. Then from \eqref{eq:rho_z_via_exp_and_E} we have that 
    \begin{align*}
        \rho_{\phi(z)} \circ \Psi\circ \rho_z^{-1}: ID(z) \to ID(\phi(z))
    \end{align*}
    is well-defined and
    \begin{align*}
        \rho_{\phi(z)}\circ \Psi \circ \rho_z^{-1} = \text{exp}_{\phi(z)}^{-1} \circ \Phi_{U_1}\circ \text{exp}_z.
    \end{align*}
    Thus, the previous local representation of $\Phi_{U_1}$ is smooth and the equation (\ref{eq:phi_u_linear}) follows if we can show that in $ID(z)$ it holds that
    \begin{equation}
    \label{eq:local_rep_of_Psi}
    \rho_{\phi(z)}\circ \Psi \circ\rho_z^{-1} = d\phi_z.
    \end{equation} 
    For any $\eta\in ID(z)\backslash\{0\}$ we have by equation \eqref{eq:rho_z_via_exp_and_E} that 
    $\rho_z^{-1}(\eta) = d_{g_1}(\text{exp}_z(\eta),\cdot)|_{\overline{U}_1}.$ Thus
    \begin{align*}
        (\rho_{\phi(z)}\circ \Psi \circ \rho_z^{-1})(\eta) &= (\rho_{\phi(z)}\circ\Psi)(d_{g_1}(\text{exp}_z(\eta),\cdot)|_{\overline{U}_1})\\
        &= \rho_{\phi(z)}(d_{g_1}(\text{exp}_z(\eta),\cdot)\circ \phi^{-1}|_{\overline{U}_2}) \\
        &= -d_{g_1}(\text{exp}_z(\eta),z) \nabla_{g_2}(d_{g_1}(\text{exp}_z(\eta),\cdot)\circ \phi^{-1})|_{\phi(z)}.
        \end{align*}
    Since $\phi:(\cX_1,g_1)\to (\cX_2,g_2)$ is a Riemannian isometry by Lemma \ref{lem:distance_function_data}, so for any function $f\colon \cX_1 \to \R$ which is smooth in an open neighborhood of $z$ we have that $d\phi_z(\nabla_{g_1} f|_z) = \nabla_{g_2}(f\circ\phi^{-1})|_{\phi(z)}$. Hence the former equation implies that
    \begin{align*}
        (\rho_{\phi(z)}\circ \Psi \circ \rho_z^{-1})(\eta) &= - d_{g_1}(\text{exp}_z(\eta),z) d\phi_z(\nabla_{g_1} d_{g_1}(\text{exp}_z(\eta),\cdot)|_z).
    \end{align*}
    As $\eta\in ID(z)\backslash\{0\}$, we have that $d_{g_1}(\text{exp}_z(\eta),z) = \norm{\eta}_{g_1}$. Furthermore, by Lemma \ref{lem:dist_gradient} it follows that $\nabla_{g_1}d_g(\text{exp}_z(\eta),\cdot)|_z = - \norm{\eta}_{g_1}^{-1}\eta$. Hence,
    \begin{align*}
        (\rho_{\phi(z)}\circ \Psi \circ \rho_z^{-1})(\eta) &= -\norm{\eta}_{g_1}d\phi_z(-\norm{\eta}_{g_1}^{-1}\eta) = d\phi_z\eta. 
    \end{align*}
    On the other hand, for $\eta=0\in ID(z)$ we have by equations \eqref{eq:quasi_isometry} and \eqref{eq:rho_z_via_exp_and_E} that 
    \begin{align*}
     (\rho_{\phi(z)}\circ \Psi \circ \rho_z^{-1})(0)
     =&
     % (\exp_{\phi(z)}^{-1}\circ R_2^{-1} \circ \Psi \circ R_1\circ \exp_z)(0)
     % =
     (\exp_{\phi(z)}^{-1}\circ \Phi_{U_1}\circ \exp_z)(0)
     =\exp_{\phi(z)}^{-1}(\Phi_{U_1}(z))=\exp_{\phi(z)}^{-1}(\phi(z))
    % \\
    % =&
    =0=d\phi_z0.
    \end{align*}
    Hence, equation \eqref{eq:local_rep_of_Psi} is verified in $ID(z)$.
    
    % This holds for all $\eta \in ID(z)\backslash\{0\}$, hence $\rho_{\phi(z)}\circ\Psi\circ\rho_z^{-1} = d\phi_z$ and we have that $\rho_{\phi(z)}\circ\Psi\circ\rho_z^{-1}$ is a diffeomorphism on its domain. 
    
    Finally, we note that since $U_1$ is open, then $\cup_{z\in U_1}(N_1\backslash\text{Cut}(z)) = N_1$ and it follows that $\Phi_{U_1}:N_1\to N_2$ is a diffeomorphism by \cite[Corollary 2.8]{lee_smooth} and Proposition \ref{prop:phi_homeomorphism}. 
\end{proof}

\subsection{Recovery of Riemannian Structure} \label{sub:geometry}
In this subsection we show that if two manifolds satisfy Hypothesis \ref{hyp:same_data} then they are Riemannian isometric. 

\begin{prop} \label{prop:phi_local_isometry}
       Let $\cX_1,U_1,\Psi$ and $\Phi|_{U_1}$ be as in Proposition \ref{prop:phi_homeomorphism}. If Hypothesis $\ref{hyp:same_data}$ is satisfied, then $\Phi_{U_1}$ is a Riemannian isometry.
\end{prop}

\begin{proof}
    For the rest of this proof we fix $x_0\in N_1$, and define 
    $$
    V_{x_0} := \{y\in U_1: d_{g_1}(x_0,\cdot)\,\text{is smooth in an open neighborhood of $y$}\} = U_1\backslash (\text{Cut}(x_0)\cup\{x_0\}).
    $$ 
    The set $V_{\Phi_{U_1}(x_0)} \subset U_2:=\phi(U_1)$ is defined analogously. 
    Since $\Phi_{U_1}:N_1\to N_2$ is a diffeomorphism 
    % and by $(\ref{eq:quasi_isometry})$ we can show that $\phi: V_{x_0}\to V_{\Phi_{U_1}(x_0)}$ is a bijection. In particular, 
    we get by \eqref{eq:quasi_isometry} and the chain rule that
    \begin{align}
        d(d_{g_1}(\cdot,y))|_{x_0} = d(d_{g_2}(\cdot,\phi(y))|_{\Phi_{U_1}(x_0)}) d\Phi_{U_1}|_{x_0}, \quad
        \text{ for all } y\in V_{x_0}.
        \label{eq:isom_differential}
    \end{align}
    We define 
    \begin{align*}
        Z(x_0) :&= \{d(d_{g_1}(\cdot,y))|_{x_0}: y\in V_{x_0}\} \subset T^\ast_{x_0}N_1
        \intertext{and}
        Z(\Phi_{U_1}(x_0)) :&= \{d(d_{g_2}(\cdot,\hat{y}))|_{\Phi_{U_1}(x_0)}:\hat{y}\in V_{\Phi_{U_1}(x_0)}\} \subset T^\ast_{\Phi_{U_1}(x_0)}N_2.
    \end{align*}
    Since $\Phi_{U_1}$ is a diffeomorphism, then $d\Phi_{U_1}|_{x_0}$ is invertible, which in combination with $(\ref{eq:isom_differential})$ implies that 
    $$
        d(\Phi_{U_1})_{x_0}^\ast: Z(\Phi(x_0)) \to Z(x_0)
    $$
    is a well-defined bijection, where $d(\Phi_{U_1})_{x_0}^*$ is the pullback of a covector field. 
    
    Next we show that $Z(x_0) \subset S_{x_0}^*N_1$ is an open set in the subspace topology on the unit cosphere with respect to $g_1$. We recall that the musical isomorphism $\sharp : T_{x_0}^*N_1 \to T_{x_0}N_1$, which acts on the differential as $(df)^\sharp = \nabla_{g_1}f$, is a vector space isomorphism between $T_{x_0}^*N_1$ and $T_{x_0}N_1$. Hence, the operator $\sharp$ maps $Z(x_0)$ bijectively into the set
    $$
        Z(x_0)^\sharp = \{\nabla_{g_1}d_{g_1}(\cdot,y)|_{x_0}: y\in V_{x_0}\} \subset T_{x_0}N_1.
    $$
    Due to the Lemma \ref{lem:dist_gradient} we know that 
    $$
    \nabla_{g_1} d_{g_1}(\cdot,y)|_{x_0} = -\frac{\text{exp}_{x_0}^{-1}(y)}{\|\text{exp}_{x_0}^{-1}(y)\|_{g_1}} \in S_{x_0}N_1.
    $$ 
    This shows that $Z(x_0)^\sharp \subset S_{x_0}N_1$. Since $\text{exp}_{x_0}$ is continuous then $\text{exp}_{x_0}^{-1}(V_{x_0}) \subset T_{x_0}N$ is an open set. At the same time $v \to \frac{v}{\|v\|_g}$ is an open map provided $v\not=0$. Hence $Z^\sharp(x_0)$ is open in the subspace topology and so $Z(x_0) \subset S_{x_0}^*N_1$ is open in the subspace topology as well. Similarly $Z(\Phi_{U_1}(x_0)) \subset S_{\Phi_{U_1}(x_0)}^* N_2$ is open in the subspace topology where the unit cosphere is taken with respect to $g_2$. In particular, we have that
    \begin{align}
        \norm{d(\Phi_{U_1})_{x_0}^*\omega}_{g_1}^2 &= 1 = \norm{\omega}_{g_2}^2 \hspace{4mm}\text{for all $\omega\in Z(\Phi_{U_1}(x_0))$}.\nonumber
        \intertext{Furthermore, if we consider the open cone $\R_+ Z(\Phi_{U_1}(x_0))$, and $\tilde{\omega}\in \R_+Z(\Phi_{U_1}(x_0))$ then there exists a unique $\alpha > 0$ and $\omega\in Z(\Phi_{U_1}(x_0))$ such that $\tilde{\omega} =\alpha\omega$. Then by linearity of $d(\Phi_{U_1})_{x_0}^*$ we have that}
        \norm{d(\Phi_{U_1})_{x_0}^* \tilde{\omega}}_{g_1}^2 &= \alpha^2 = \norm{\tilde{\omega}}_{g_2}^2\hspace{4mm}\text{for all $\tilde{\omega}$ in $\R_+Z(\Phi_{U_1}(x_0))$}.    \label{eq:covector_1}
    \end{align}
     If we were to let $\{E_j\}_{j=1}^n$ 
     and $\{\tilde{E}_j\}_{j=1}^n$ be local coordinates for the open cones $\R_+Z(x_0)$ and $\R_+Z(\Phi_{U_1}(x_0))$, then $(\ref{eq:covector_1})$ is equivalent to 
    \begin{align}
        g_{2}^{ij}(\Phi_{U_1}(x_0)) \tilde{\omega}_i\tilde{\omega}_j = g^{kl}_1(x_0) \frac{\partial \Phi^i_{U_1}(x_0)}{\partial E_k} \frac{\partial \Phi^j_{U_1}(r_0)}{\partial E_l} \tilde{\omega}_i\tilde{\omega}_j\hspace{4mm}\text{for all $\tilde{\omega}\in \R_+Z(\Phi(x_0))$}\label{eq:covector_2}
    \end{align}
    where $\{g^{ij}_1\}$ and $\{g_2^{ij}\}$ are the inverse matrices of $g_1$ and $g_2$. Since $\R_+Z(\Phi_{U_1}(x_0)) \subset T_{\Phi_{U_1}(x_0)}^*N_2$ is open, then by differentiating with respect to $\tilde{\omega}_i$ in \eqref{eq:covector_2} we get the equality
    $$
        g_{2}^{ij}(\Phi(x_0)) = g^{kl}_1(x_0) \frac{\partial \Phi^i_{U_1}(x_0)}{\partial E_k} \frac{\partial \Phi^j_{U_1}(x_0)}{\partial E_l}.
    $$    
     Since the equality above holds for all $x_0 \in N_1$ and is equivalent to $g_1 = \Phi^*_{U_1} g_2$, then $\Phi_{U_1}:(N_1,g_1)\to (N_2,g_2)$ is a Riemannian isometry. 
\end{proof}

We are ready to provide and prove the main result of this section.

\begin{thm}
\label{thm:isometry_between_manifolds}
   % Suppose that Hypothesis \ref{hyp:same_data} is satisfied with $\phi:\cX_1 \to \cX_2$, 
   If the conditions of the Theorem \ref{thm:main_thm} are valid
   then there exists a unique Riemannian isometry $\Phi:(N_1,g_1)\to (N_2,g_2)$ satisfying $\Phi|_{\cX_1} = \phi$.
\end{thm}

\begin{proof}
    Let $U_1 \subset \cX_1$ be open and bounded such that $\bar U_1 \subset \cX_1$, then by Proposition \ref{prop:phi_local_isometry} we know that $\Phi_{U_1}:(N_1,g_1)\to (N_2,g_2)$ is a Riemannian isometry and by \eqref{eq:quasi_isometry} we know that $\Phi_{U_1}|_{U_1} = \phi|_{U_1}$.

    We pause our main proof to recall the following well known fact from Riemannian geometry. If $X$ and $Y$ are two connected and complete Riemannian manifolds and $F_i\colon X \to Y$ for $i \in \{1,2\}$ are two Riemannian isometries that agree on some open set $\Omega \subset X$ then $F_1 = F_2$ (see \cite[Proposition 5.22]{lee_riemannian}).

    Since $\cX_1$ is an open subset, then there exists a sequence of open and bounded sets $\{U_j\}$ such that $U_j\subset U_{j+1}$ and $\bigcup_{j\in\mathbb{N}} U_j = \cX_1$. We have so far proved that for each $U_j$ there exists a Riemannian isometry $\Phi_{U_j}:(N_1,g_1)\to (N_2,g_2)$ such that $\Phi_{U_j}|_{U_j} = \phi|_{U_j}$. Since $U_j\subset U_{j+1}$, then $\Phi_{U_j}|_{U_j} = \phi|_{U_j}= \Phi_{U_{j+1}}|_{U_j}$. Since $U_j$ is open, then by the above rigidity result for isometries we conclude that $\Phi_{U_j} = \Phi_{U_{j+1}}$. We define $\Phi := \Phi_{U_1}$, then by induction it follows that $\Phi = \Phi_{U_j}$ for all $j$. As a consequence, for any $x\in \cX_1 = \bigcup_{j\in\N} U_j$ there exists $j$ such that $x\in U_j$. Hence $\Phi(x) = \Phi_{U_j}(x) = \phi(x)$, and so there exists a unique Riemannian isometry $\Phi:(N_1,g_1)\to (N_2,g_2)$ such that $\Phi|_{\cX_1} = \phi$.
\end{proof}

\section{Recovery of the Lower Order Terms} \label{sec:lower_order_terms}
In Section \ref{sec:geometry} we proved that if Hypothesis \ref{hyp:same_data} holds, then the complete manifolds $(N_1,g_1)$ and $(N_2,g_2)$ are Riemannian isometric. In this section we prove the existence and uniqueness of a gauge in the lower order terms as promised in Theorem \ref{thm:main_thm}. First in Subsection \ref{sub:reduction_manifold} we use the Riemannian isometry constructed in Section \ref{sec:geometry} to show that we can assume without loss of generality that our source-to-solution operators are defined on the same complete manifold and moving forward assume that $\opnum{1} = \opnum{2}$ with $g_1=g_2$ and $\cX_1=\cX_2$. The remainder of the section is dedicated to proving the converse of Proposition \ref{prop:existence_of_gauge}, namely if $\Lambda_{g,A_1,V_1} = \Lambda_{g,A_2,V_2}$ then $A_1 = A_2 + i\kappa^{-1}d\kappa$ and $V_1= V_2$ where $\kappa$ is a smooth unitary function such that $\kappa|_{\cX} = 1$. 
In Subsection \ref{sub:blago_extension} we extend the Blagovestchenskii Identity $(\ref{thm:blagovestchenskii})$ to show that if the source-to-solution operators coincide, then the inner product of solutions to the Cauchy problem will coincide on special controllable sets. Next in Subsection \ref{sub:controllable_sets} we show that the controllable sets can be constructed in such a way so that we can use Lebesgue Differentiation Theorem to conclude that if the source-to-solution operators coincide then also products of two waves agree pointwise. That is if $f,h$ are some source terms then the respective solutions of the Cauchy problem \eqref{eq:cauchy_problem_infinite_time} satisfy the property
$u_1^f\overline{u^h_1} = u_2^f\overline{u_2^h}$.
Finally, in Subsection \ref{sub:determine_gauge} we use the product property to show the existence of a smooth function $\kappa$ as in condition (2) of Lemma \ref{lem:gauge_equivalent}. Hence, Theorem \ref{thm:main_thm} follows from this lemma.
%Finally in Subsection \ref{sub:determine_gauge} we use the product property to construct the smooth function $\kappa$ that is needed in Subsection \ref{sub:gauge} in order to prove that we have the desired gauge. 

\subsection{Reduction to one Manifold} \label{sub:reduction_manifold}
Let $\magneticnum{1}$ and $\magneticnum{2}$ be Magnetic-Schr\"odinger operators on complete manifolds $(N_1,g_1)$ and $(N_2,g_2)$ such that the local source-to-solution operators $\opnum{1}$ and $\opnum{2}$ satisfy Hypothesis $(\ref{hyp:same_data})$ on open sets $\cX_i \subset N_i$ with diffeomorphism $\phi:\cX_1\to\cX_2$. From Section \ref{sec:geometry} we know that there exists a unique Riemannian isometry $\Phi:(N_1,g_1)\to (N_2,g_2)$ such that $\Phi|_{\cX_1} = \phi$. 

We define a metric tensor $\tilde{g}: = \Phi^*g_2$, a co-vector field $\tilde{A}: = \Phi^*A_2$, and a function $\tilde{V} = \Phi^*V_2$ on $N_1$ and let $\mathcal{L}_{\tilde{g},\tilde{A},\tilde{V}}$ be the associated Magnetic-Schr\"odinger operator, while $\tilde{\Lambda}$ is the associated local source-to-solution operator of the Cauchy Problem \ref{eq:cauchy_problem_infinite_time} on $C_0^\infty
((0,\infty)\times\cX_1)$. By Lemma \ref{lem:magnetic_isometry} we have for any $v \in C^\infty_0(N_2)$ that
$$
    \mathcal{L}_{\tilde{g},\tilde{A},\tilde{V}} (v\circ\Phi) = [\magneticnum{2}(v)]\circ\Phi.
$$

Let $f\in C_0^\infty((0,\infty)\times\cX_2)$ be arbitrary and let $u^f\in C^\infty([0,\infty)\times N_2)$ be the unique solution to the Cauchy Problem $(\ref{eq:cauchy_problem})$ with the source $f$ and $\magneticnum{}=\magneticnum{2}$. By a direct computation we see that $\tilde{u}^{\tilde{f}} = u^f\circ\tilde{\Phi}$ is the unique smooth solution to the Cauchy problem \eqref{eq:cauchy_problem}   with source $\tilde{f} = f\circ\tilde{\Phi} \in C^\infty([0,\infty)\times N_1)$ and $\magneticnum{}=\mathcal{L}_{\tilde{g},\tilde{A},\tilde{V}}$. In particular, since $\tilde{\Phi}|_{(0,\infty)\times\cX_1} = \tilde{\phi}$ we get that
$$
    \tilde{\Lambda}(f\circ\tilde{\Phi}) = \tilde{u}^{\tilde{f}}|_{(0,\infty)\times\cX_1} = u^f\circ\tilde{\Phi}|_{(0,\infty)\times\cX_1} = \Lambda_2(f)\circ\tilde{\Phi}|_{(0,\infty)\times \cX_1} = \Lambda_2(f)\circ\tilde{\phi} = \Lambda_1(f\circ\tilde{\phi}).
$$
Hence, $\tilde{\Lambda} = \Lambda_1$, and above we adopted the short hand notation $\Lambda_i := \Lambda_{g_i,A_i,V_i}$. 

Since $\Phi:(N_1,g_1)\to (N_2,g_2)$ is a Riemannian isometry, we have that $g_1 = \tilde{g} = \Phi^*g_2$, and due to the previous discussion we may with out loss of generality assume that $(N,g) = (N_1,g_1) = (N_2,g_2)$, $\cX= \cX_1 = \cX_2$, and $\Lambda_1=\Lambda_2$ as operators on $C_0^\infty((0,\infty
)\times\cX)$.

\subsection{Improved Blagovestchenskii Identity} \label{sub:blago_extension}
The Blagovestchenskii Identity $(\ref{thm:blagovestchenskii})$ can be used to compute the inner products of the functions $\{u^f(T): f\in C_0^\infty((0,2T)\times\cX)\}$. We will now extend the Blagovestchenskii identity to compute the inner products of functions of the form $\{1_Zu^f(T): f\in C_0^\infty((0,T)\times\cX)\}$ where $1_Z$ is the characteristic function of a set $Z\subset N$, chosen from a collection satisfying certain specific properties. 
% Later on we will use this freedom in our choice of sets $A$ in order to relate the solutions $u_1^f$ and $u_2^f$ pointwise by using the Lebesgue Differentiation Theorem. These results are largely based off of the results developed in the Boundary Control Method for lower order terms. In order to extend the Blagovestchenskii Identity we begin with the following Lemma. 
Many results presented in this subsection are adaptations of analogous results in \cite{nursultanov2023} that concerns the recovery of lower order terms from the Dirichlet-to-Neumann map of the initial boundary value problem for the wave equation with lower order terms.

For a given open set $\mathcal{Y}\subset N$ and $s>0$ we recall that the respective domain of influence $M(\mathcal{Y},s)$ was defined in the equation \eqref{def:domain_of_influence}.  
\begin{lem} \label{double-convergence}
    Let $\Lambda_1 = \Lambda_2$, $f\in C_0^\infty((0,T)\times\cX)$, $s\in (0,T]$, and  $\mathcal{Y}\subset \cX$ be open. 
    If $\{f_j\}\subset C_0^\infty((T-s,T)\times\mathcal{Y})$ are functions such that
    \begin{align*}
        u_1^{f_j}(T,\cdot) &\stackrel{j \to \infty}{\longrightarrow} 1_{M(\mathcal{Y},s)}u_1^f(T,\cdot) \hspace{4mm} \text{in $L^2(N)$},
        \intertext{then}
        u_2^{f_j}(T,\cdot) & \stackrel{j \to \infty}{\longrightarrow}  1_{M(\mathcal{Y},s)}u_2^f(T,\cdot) \hspace{4mm} \text{in $L^2(N)$}.    
    \end{align*}
\end{lem}

\begin{proof}
The proof is nearly identical to the proof of \cite[Lemma 8]{nursultanov2023} that concerns the case when $M$ is a compact manifold with boundary. We skip the proof here.
\end{proof}

We can now extend the Blagovestchenskii Identity to get the following result. 

\begin{lem}\label{lem:Blagovestchenskii-Extension} 
Let $\Lambda_1 = \Lambda_2$, $T> 0$, $f,h\in C_0^\infty((0,T)\times\cX)$, $\mathcal{Y}, \hat{\mathcal{Y}}\subset \cX$ be open, and $s,\hat{s} \in (0,T]$. Then
    \begin{align}
        \langle 1_{M(\mathcal{Y},s)} u^f_1(T),1_{M(\hat{\mathcal{Y}},\hat{s})}u_1^h(T)\rangle_{L^2(N)} &= \langle 1_{M(\mathcal{Y},s)} u^f_2(T),1_{M(\hat{\mathcal{Y}},\hat{s})}u_2^h(T)\rangle_{L^2(N)}.  \label{eq:blago_extension}
    \end{align}
\end{lem}
\begin{proof}
    % See \cite[Corollary 3]{nursultanov2023} for a proof for the case when the source term is on the boundary of a compact manifold with boundary. 
    Since $1_{M(\mathcal{Y},s)}u_1^f(T)\in L^2(M(\mathcal{Y},s))$ and $1_{M(\hat{\mathcal{Y}},\hat{s})}u_1^h(T)\in L^2(M(\hat{\mathcal{Y}},\hat{s}))$, we get from the  Approximate Controllability as described in Theorem \ref{thm:approximate_controllability}, that there are $\{f_j\}\subset C_0^\infty((T-s,T)\times\mathcal{Y})$ and $\{h_j\}\subset C_0^\infty((T-\hat{s},T)\times\hat{\mathcal{Y}})$ such that $u_1^{f_j}(T)\to 1_{M(\mathcal{Y},s)}u_1^f(T)$ and $u_1^{h_j}(T)\to 1_{M(\hat{\mathcal{Y}},\hat{s})}u_1^h(T)$ in $L^2$. 
    
    Then by Lemma $\ref{double-convergence}$ we also have that $u_2^{f_j}(T)\to 1_{M(\mathcal{Y},s)}u_2^f(T)$ and $u_2^{h_j}(T)\to 1_{M(\hat{\mathcal{Y}},\hat{s})} u_2^h(T)$ in $L^2$. Hence, the equation \eqref{eq:blago_extension} follows from the Blagovestchenskii Identity in Theorem \ref{thm:blagovestchenskii} and $\Lambda_1 = \Lambda_2$.
    \end{proof}
    We also have the following corollary. 
\begin{cor}
    Let $\Lambda_1 = \Lambda_2$, $T > 0$, $f,h\in C_0^\infty((0,T)\times\cX)$, $\mathcal{Y}, \hat{\mathcal{Y}}\subset \cX$ be open, and $s,\hat{s}\in (0,T]$. Then
    \begin{align}
        \langle 1_{M(\mathcal{Y},s)}u_1^f(T), u_1^h(T)\rangle_{L^2(N)} &= \langle 1_{M(\mathcal{Y},s)}u_2^f(T), u_2^h(T)\rangle_{L^2(N)},\nonumber
        \intertext{and}   
        \langle 1_{M(\mathcal{Y},s)\backslash M(\hat{\mathcal{Y}},\hat{s})} u^f_1(T), u_1^h(T)\rangle_{L^2(N)} &= \langle 1_{M(\mathcal{Y},s)\backslash M(\hat{\mathcal{Y}},\hat{s})} u^f_2(T), u_2^h(T)\rangle_{L^2(N)}.
        \label{eq:difference}  
    \end{align}
\end{cor}

\begin{proof}
    For the first equality we take $\hat{\mathcal{Y}} = \cX$ and $\hat{s} = T$ and apply Lemma \ref{lem:Blagovestchenskii-Extension}. Additionally, we note that $\text{supp}(u_k^h(T))\subset M(\cX,T)$ by Corollary \ref{cor:finite_speed}, and therefore $1_{M(\cX,T)}u_k^f(T) = u_k^f(T)$. This verifies the first equation. 
    
    For the second equation we note first that $1_{M(\mathcal{Y},s)\backslash M(\hat{\mathcal{Y}},\hat{s})}= 1_{M(\mathcal{Y},s)} - 1_{M(\mathcal{Y},s)}1_{M(\hat{\mathcal{Y}}, \hat{s})}$. Therefore, we obtain the equality $\eqref{eq:difference}$, by combining the results from Lemma \ref{lem:Blagovestchenskii-Extension} and the first equality of the current lemma.
\end{proof}

\subsection{Construction of Controllable Sets} \label{sub:controllable_sets}
The purpose of this section is to construct for each $x_0 \in N$ a sequence of sets $Z_j:=\overline{M(\mathcal{Y}_j,s_j)\backslash M(\hat{\mathcal{Y}}_j,\hat{s}_j)}$ where $\mathcal{Y}_j,\hat{\mathcal{Y}}_j\subset \cX$ so that $Z_j$ converges to $x_0$. In particular, we will take $\mathcal{Y}_j$ and $\hat{\mathcal{Y}}_j$ to be certain metric balls so that we have an explicit control on the convergence of $Z_j$. This will allow us to use the Lebesgue Differentiation Theorem for the equation \eqref{eq:difference}. 
%
% \begin{defn}
%     On our complete manifold $(N,g)$, for any measurable set $A\subset N$. We let
%     \begin{align*}
%         |A| &= \int_A dV_g
%     \end{align*}
% \end{defn}
% \begin{thm}\label{FTC}
%     Let $f\in C(N)$ and let $\{Z_k\}$ be a compact, contracting sequence of measurable sets such that $\bigcap_{k\in\N} Z_k = \{x_0\}$ and $|Z_k|\not=0$ for all $k\in \N$. Then
%     $$
%         \lim_{k\to\infty}\frac{1}{|Z_k|}\int_{Z_k} fdV_g = f(x_0)
%     $$
% \end{thm}
% \begin{proof}
%     Let $\eps > 0$. Since $f$ is continuous, there exists $\delta > 0$ such that $$f(B(x_0,\delta))\subset B(f(x_0),\eps)$$
%     Since $\{Z_k\}$ is a contracting sequence of sets and $\bigcap_{k\in \N}Z_k = \{x_0\}$, then there exists $K\in\N$ such that $k\geq K$ implies that $Z_k\subset B(x_0,\delta)$. Then by continuity of $f$ we have for $k\geq K$ that
%     \begin{align*}
%         \Big|\frac{1}{|Z_k|}\int_{Z_k} f(y)dV_g(y) - f(x_0) \Big| &= \Big|\frac{1}{|Z_k|}\int_{Z_k}f(y)dV_g(y) - \frac{1}{|Z_k|}\int_{Z_k}f(x_0)dV_g(y)\Big|\\
%         &= \Big|\frac{1}{|Z_k|}\int_{Z_k} f(y) - f(x_0)dV_g(y)\Big|\\
%         &\leq \frac{1}{|Z_k|}\int_{Z_k}|f(y)-f(x_0)|dV_g\\
%         &\leq \frac{1}{|Z_k|}\int_{Z_k} \eps dV_g = \eps
%     \end{align*}
% \end{proof}
%

We recall that according to the equation \eqref{eq:Cut_locus} for each $x_0 \in N$ its cut locus is the set
\[
\text{Cut}(x_0):=\{\gamma_{x_0,v}(\tau(x_0,\xi)) \in N: \: \xi \in S_{x_0}N\},
\]
where $\tau$ is the cut time function of Definition \ref{def:cut_time}. Suppose that $x_0\in N$. Since $\text{Cut}(x_0)$ is a closed set of measure $0$ and since $\cX$ is an open set, we may find $y\in \cX\backslash \text{Cut}(x_0)$. Hence we have also that $x_0\not\in \text{Cut}(y)$. We denote $s := d_g(x_0,y) > 0$. Thus, there is $\varepsilon>0$ and a unique distance minimizing geodesic $\gamma\colon[-\eps,s+\eps]\to N$ such that $\gamma(0)=y$, $\gamma(s)=x_0$ and for all $\delta\in (0,\eps]$ we have that
    \begin{enumerate}
        \item $\gamma(-\delta)\not\in \text{Cut}(x_0)$
        \item $\gamma(s + \delta)\not\in \text{Cut}(y)$
        \item $\overline{B(\gamma(-\delta),\delta)} \subset \cX$.
    \end{enumerate}

We are finally in a position to define our desired sets. 

\begin{defn}\label{def:set_definition}
    Let $x_0\in N$. Choose
$y\in \cX\backslash \text{Cut}(x_0)$, $s:=d_g(x_0,y)$, $\gamma$, $\eps>0$ as above. 
% Set $s := d_g(x_0,y) > 0$. As $x_0\not\in \text{Cut}(y)$, then we have a unique geodesic $\gamma:\R\to N$ which satisfies $\gamma(0) = y$ and $\gamma(s) = x_0$. Now since $\cX$ is open and the sets $\text{Cut}(x_0)$ and $\text{Cut}(y)$ are closed, we may find $\eps> 0$ such that for all $\delta\in (0,\eps]$ we have that
%     \begin{enumerate}
%         \item $\gamma(-\delta)\not\in \text{Cut}(x_0)$
%         \item $\gamma(s + \delta)\not\in \text{Cut}(y)$
%         \item $\overline{B(\gamma(-\delta),\delta)} \subset \cX$
%     \end{enumerate}
    We define for $\delta\in(0,\eps]$ the open sets $\mathcal{Y}_\delta = B(\gamma(-\delta),\delta)$. Finally, we define 
    \begin{align}\label{def:Z}
        Z_\delta &:= \overline{M(\mathcal{Y}_\delta,s+\delta)\backslash M(\mathcal{Y}_\eps, s)}
    \end{align}
    The visualization of the set $Z_\delta$ is given in Figure \ref{fig:Z_delta}.
\end{defn}

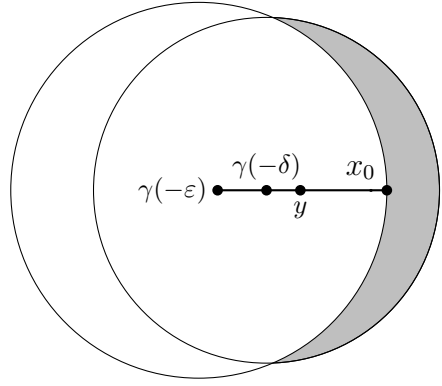
\begin{figure}
\begin{tikzpicture}[scale=0.5]
\filldraw[fill=lightgray](-0.2,0) circle [radius=3 + 1.6];
\filldraw[fill=white](-2,0) circle [radius=3 + 2];
\draw[thin](-0.2,0) circle [radius=3+1.6];
\draw[thick, draw=black] (-1.5,0)--(3,0) node [left] {.};
\node[anchor=south east] at (3,0) {$x_0$};
\fill (3,0) circle[radius=4pt];
\node[anchor=south] at (-0.2,0) {\footnotesize $\gamma(-\delta)$};
\fill (-0.2,0) circle[radius=4pt];
\node[anchor=east] at (-1.5,0) {\footnotesize $\gamma(-\eps)$};
\fill (-1.5,0) circle[radius=4pt];
\node[anchor=north] at (0.7,0) {\footnotesize $y$};
\fill (0.7,0) circle[radius=4pt];
\end{tikzpicture}
\caption{The shaded in region and its boundary is the set $Z_\delta$, the left circle represents $M(\mathcal{Y}_\eps,s)$ and the right circle is the set $M(\mathcal{Y}_\delta,s+\delta)$}
\label{fig:Z_delta}
\end{figure}

% It is important to highlight that despite $\gamma$ being a distance minimizing curve, then metric balls in Definition \ref{set_definition} do not need need to be geodesic balls and in fact do not need to have a smooth boundary.

The remainder of the subsection is focused on proving the following result, that implies that the sets $\{Z_\delta\}$ satisfy the conditions for the Lebesgue Differentiation Theorem. We prove the main proposition via a series of auxiliary lemmas. In particular, if $Z \subset N$ is measurable then we let $|Z| := \int_Z dV_g$ stand for the measure of $Z$.

\begin{prop} \label{prop:summary}
    For $0 < \delta\leq \eps$ the sets $\{Z_\delta\}_{\delta\in (0,\eps]}$ satisfy the following conditions
    \begin{enumerate}
        \item $Z_\delta\subset Z_{\delta'}$ if $\delta \leq \delta'$
        \item $\bigcap_{0 < \delta\leq \eps} Z_\delta = \{x_0\}$
        \item $|Z_\delta|\not= 0$
        \item There exists $\tilde{\delta}\in (0,\eps)$ so that $|Z_\delta| = |M(\mathcal{Y}_\delta,s+\delta)\backslash M(\mathcal{Y}_\eps, s)| $ for any $\delta\in (0,\tilde{\delta}]$    
    \end{enumerate}
\end{prop}

For the remainder of the subsection, $x_0,y\in N$, $s > 0$, $\eps>0$, $\gamma:\R\to N$, shall all remain fixed. On our path towards showing that the sets $\{Z_\delta\}_{\delta\in (0,\eps)}$ satisfy the claim of Proposition \ref{prop:summary}, begins from the following inclusion property for $Z_\delta$. 

\begin{lem} \label{lem:Z_part2} 
If $\delta \in(0,\eps]$, then
\begin{equation}
\begin{aligned}\label{eq:Z_delta_equal} Z_\delta &\subset
        \Big(M(\mathcal{Y}_\delta, s+\delta)\backslash M(\mathcal{Y}_\eps,s)\Big)
        % \\&
        % \hspace{10mm} 
        \cup \Big(\overline{B(\gamma(-\delta), s + 2\delta)}\cap \partial B(\gamma(-\eps),s+\eps)\Big)
        % &= \overline{B(\gamma(-\delta), s + 2\delta)}\cap\Big(\Big(N\backslash\overline{B(\gamma(-\eps),s+\eps)}\Big) \nonumber\\
        % &\hspace{10mm}\cup \partial B(\gamma(-\eps),s+\eps)\Big) \nonumber\\
        % &= \Big(\overline{B(\gamma(-\delta), s + 2\delta)}\backslash \overline{B(\gamma(-\eps),s+\eps)}\Big) \nonumber\\
        % &\hspace{10mm} \cup \Big(\overline{B(\gamma(-\delta), s + 2\delta)}\cap \partial B(\gamma(-\eps),s+\eps)\Big)\nonumber\\
        \\
        &= \overline{B(\gamma(-\delta), s + 2\delta)}\cap \Big(N\backslash B(\gamma(-\eps),s+\eps)\Big)
    \end{aligned}
\end{equation}
    % \begin{align}
    % \label{eq:Z_delta_inclusion}
    %     Z_\delta &\subset \Big(M(\mathcal{Y}_\delta, s+\delta)\backslash M(\mathcal{Y}_\eps,s)\Big) \cup\Big(\overline{B(\gamma(-\delta),s + 2\delta)}\cap \partial B(\gamma(-\eps),s+\eps) \Big).
    % \end{align}
    % Furthermore, 
    % \[
    % \overline{B(\gamma(-\delta),s + 2\delta)}\cap \partial B(\gamma(-\eps),s+\eps)
    % =
    % \overline{B(\gamma(-\delta), s + 2\delta)}\cap \Big(N\backslash B(\gamma(-\eps),s+\eps)\Big).
    % \]
    % \teemu{I don't see where this follows. Based on \eqref{eq:Z_delta_equal} the RHS seems to be a bigger set.}
\end{lem}

\begin{proof}
    The second equality in \eqref{eq:Z_delta_equal} follows from Lemma \ref{lem:boundary_properties}.
    % \[
    % M(\mathcal{Y}_\delta,s+\delta) = \overline{B(\gamma(-\delta), s+2\delta)}
    % \text{ and } M(\mathcal{Y}_\eps,s)=\overline{B(\gamma(-\eps),s+\eps)}.
    % \]
    % These yield 
    % This shows the equivalence of the last two equalities and it follows that $$\Big(M(\mathcal{Y}_\delta, s+\delta))\backslash M(\mathcal{Y}_\eps,s)\Big) \cup\Big(\overline{B(\gamma(-\delta),s + 2\delta)}\cap \partial B(\gamma(-\eps),s+\eps) \Big)$$
    Hence, the first inclusion in \eqref{eq:Z_delta_equal} holds as the last expression 
    describes a closed set. 
    % Thus \eqref{eq:Z_delta_inclusion} holds due to the definition of $Z_\delta$.
    % So by minimality of the closure we have that 
    % \begin{align}
    %     Z_\delta &\subset \Big(M(\mathcal{Y}_\delta, s+\delta)\backslash M(\mathcal{Y}_\eps,s)\Big) \cup\Big(\overline{B(\gamma(-\delta),s + 2\delta)}\cap \partial B(\gamma(-\eps),s+\eps) \Big)
    % \end{align}
\end{proof}

\begin{comment}
        \intertext{To achieve the equality we only need to show that set containment}
        &\overline{B(\gamma(-\delta),s + 2\delta)}\cap \partial B(\gamma(-\eps),s+\eps) \subset Z_\delta. \label{eq:Z_delta_goal}

    First consider the case that $N\backslash\overline{B(\gamma(-\eps),s+\eps)}=\emptyset$, by $(\ref{eq:Z_delta_equal})$ it follows that $\overline{B(\gamma(-\delta),s + 2\delta)}\cap \partial B(\gamma(-\eps),s+\eps) = \emptyset$ and by $(\ref{eq:Z_delta_equal})$ and $(\ref{eq:Z_delta_sub})$ it follows that $Z_\delta = \emptyset$ and so $\ref{eq:Z_delta_goal}$ holds when $N\backslash \overline{B(\gamma(-\eps),s+\eps)} = \emptyset$.

    So suppose that $N\backslash\overline{B(\gamma(-\eps),s+\eps)}\not=\emptyset$, the set is also open. If $z\in \overline{B(\gamma(-\delta),s + 2\delta)}\cap \partial B(\gamma(-\eps),s+\eps)$, then we can find $\{z_n\}\subset \overline{B(\gamma(-\delta),s + 2\delta)}\cap \overline{B(\gamma(-\eps),s+\eps)}^c = M(\mathcal{Y}_\delta, s+\delta))\backslash M(\mathcal{Y}_\eps,s)$ such that $z_n\to z$. But by the definition of the closure, this means that $z\in Z_\delta$ as desired.
\end{comment}

We are now in a position to show that $\{Z_\delta\}$ have the desired properties to employ the Lebesgue differentiation theorem. Our first step is to show the nested property, a visualization is provided in Figure \ref{fig:nested_sets}.

\begin{lem} \label{lem:subset}
    If $0 < \delta < \delta' \leq \eps$, then $Z_\delta \subset Z_{\delta'}$.
\end{lem}

\begin{proof}
    By Definition \ref{def:set_definition} of the set $Z_\delta$ and Lemma \ref{lem:boundary_properties} 
% $M(\mathcal{Y}_\delta,s+\delta) = \overline{B(\gamma(-\delta),s+2\delta)}$ by Proposition \ref{prop:domain=ball}, 
it suffices to prove that 
    $$
    \overline{B(\gamma(-\delta),s + 2\delta)}\subset \overline{B(\gamma(-\delta'),s+2\delta')}.
    $$   
    % since the closure of a set preserves set containment. 
    Since $\gamma:[-\eps,s+\eps]\to N$ is a minimizing geodesic, then 
    $$
    d_g(\gamma(-\delta),\gamma(-\delta')) = \delta'-\delta.
    $$
    Then for any $x\in \overline{B(\gamma(-\delta),s + 2\delta)}$ we have that
    $$
        d_g(x,\gamma(-\delta')) \leq d_g(x,\gamma(-\delta)) + d_g(\gamma(-\delta),\gamma(-\delta'))\leq s + 2\delta + \delta'-\delta < s + 2\delta'.
    $$
    Hence, the desired inclusion holds. 
\end{proof}

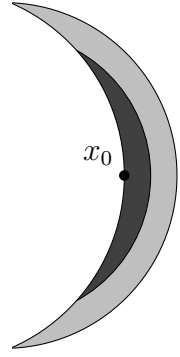
\begin{figure}
    \centering
    \begin{tikzpicture}[scale=0.5]
\filldraw[fill=lightgray](-0.2,0) circle [radius=3 + 1.6];
\filldraw[fill=darkgray](-0.1,0) circle [radius=3+0.8];
\filldraw[fill=white](-2,0) circle [radius=3 + 2];
\fill[fill=white] (-8,-6) rectangle ++(8,12);
\node[anchor=south east] at (3,0) {$x_0$};
\fill (3,0) circle [radius=4pt];
\end{tikzpicture}
    \caption{The dark gray set is $Z_\delta$ and the light gray set is $Z_{\delta'}$ where $0 < \delta < \delta'$.}
    \label{fig:nested_sets}
\end{figure}

\begin{lem} \label{lem:intersection}
    $\bigcap_{0<\delta \leq \eps} Z_\delta = \{x_0\}$
\end{lem}

\begin{proof}
    Recall that $y = \gamma(0)$. First we verify the auxiliary equation: 
    \begin{equation}
\label{eq:intersection_of_delta_balls}
    \bigcap_{0<\delta\leq \eps} \overline{B(\gamma(-\delta),s+2\delta)}= \overline{B(y,s)}.
    \end{equation}

    Suppose that $x\in \bigcap_{0< \delta \leq \eps}\overline{B(\gamma(-\delta),s+2\delta)}$. Then $d_g(x,\gamma(-\delta))\leq s+ 2\delta$  for all $\delta > 0$. Letting $\delta\to 0$ and using the continuity of $\gamma$, we find that 
    $$
    d_g(x,y) = d_g(x,\gamma(0)) = \lim_{\delta\to 0}d_g(x,\gamma(-\delta))\leq \lim_{\delta\to 0} s+2\delta = s.
    $$ 
    
    For the reverse direction we let $x\in \overline{B(y,s)}$ so that $d_g(x,y)\leq s$. For any $0<\delta \leq \eps$ we have that $\gamma:[-\delta,0]\to N$ is a minimizing curve so that $d_g(y,\gamma(-\delta)) = \delta$. Which gives us $$d_g(x,\gamma(-\delta))\leq d_g(x,y) + d_g(y,\gamma(-\delta))\leq s + \delta.$$
    So $x\in \overline{B(\gamma(-\delta),s+2\delta)}$ for all $0 < \delta\leq \eps$, hence $x\in \bigcap_{0<\delta\leq \eps}\overline{B(\gamma(-\delta),s+2\delta)}$.
    
    Now equation \eqref{eq:intersection_of_delta_balls} in combination with Lemma \ref{lem:Z_part2} and equation \eqref{eq:Z_delta_equal} yields
    \begin{align}
        \bigcap_{0<\delta\leq \eps} Z_\delta \subset \overline{B(y,s)}\cap \Big(N\backslash B(\gamma(-\eps),s+\eps)\Big). 
        \label{eq:intersections}    
    \end{align}
    Since $d_g(y,\gamma(-\eps)) = \eps$ we get from the triangle inequality that $B(y,s) \subset B(\gamma(-\eps),s + \eps)$. Hence,
    $$
        \overline{B(y,s)}\cap \Big(N\backslash B(\gamma(-\eps),s+\eps)\Big) = \partial B(y,s)\cap \Big(N\backslash B(\gamma(-\eps),s+\eps)\Big).
    $$
    We also have that $\overline{B(y,s)}\subset \overline{B(\gamma(-\eps),s+\eps)}$ so
    $$
    \overline{B(y,s)}\cap \Big(N\backslash B(\gamma(-\eps),s + \eps)\Big) \subset \overline{B(\gamma(-\eps),s+\eps)}\cap \Big(N\backslash B(\gamma(-\eps),s+\eps)\Big) = \partial B(\gamma(-\eps),s+\eps)
    $$
    It then follows from $(\ref{eq:intersections})$ that
    \begin{align} \label{subsets}
        \bigcap_{0<\delta\leq \eps} Z_\delta \subset \partial B(y,s)\cap \partial B(\gamma(-\eps),s+\eps).
    \end{align}    
    Now let $x\in \partial B(y,s)\cap \partial B(\gamma(-\eps),s+\eps)$, then $d_g(x,y) = s$. Since $N$ is a complete manifold, then there exists a minimizing geodesic $\hat{\gamma}:[0,s]\to N$ such that $\hat{\gamma}(0) = y$ and $\hat{\gamma}(s)=x$. Since $N$ is a complete manifold, then $\hat{\gamma}$ is defined for all time. Define a new curve
    \begin{align*}
        \Tilde{\gamma}(t) &= \begin{cases}
            \gamma(t)\hspace{6mm} t\leq 0\\
            \hat{\gamma}(t)\hspace{6mm} t\geq 0
        \end{cases}
    \end{align*}
    Since $\gamma(0) = \hat{\gamma}(0) = y$ then $\tilde{\gamma}$ is a rectifiable curve that connects $\gamma(-\eps)$ to $x$. If we let $L_g$ be the length functional on $(N,g)$, then we have that 
    \begin{align*}
        L_g(\tilde{\gamma}|_{[-\eps,s]}) &= L_g(\gamma|_{[-\eps,0]}) + L_g(\hat{\gamma}|_{[0,s]}) = \eps +s = d_g(\gamma(-\eps),x)
    \end{align*}
    since $x\in \partial B(\gamma(-\eps),s+\eps)$. This shows that $\tilde{\gamma}:[-\eps,s]\to N$ is a distance minimizing curve and thus is a geodesic segment. Now $\tilde{\gamma}|_{(-\eps,0]} = \gamma|_{(-\eps,0]}$ and since geodesics which agree on a part of their domain will agree on their entire domain, then we have that $\tilde{\gamma} = \gamma$. As a consequence we have that
    $x = \tilde{\gamma}(s) = \gamma(s) = x_0$.\\ \newline
    We conclude that $\partial B(y,s)\cap\partial B(\gamma(-\eps),s +\eps)\subset \{x_0\}$ and in combination with (\ref{subsets}) find that $\bigcap_{0<\delta\leq \varepsilon} Z_\delta \subset \{x_0\}$.
    
    For the reverse direction let $\delta\in (0,\varepsilon]$ and consider the points $x_\eta = \gamma(s+\eta)$ for $\eta \in (0,\delta]$. Since $\gamma:[-\varepsilon,s+\varepsilon]$ is a distance minimizing curve then
    \begin{align*}
        d_g(x_\eta,\gamma(-\delta)) &= s+\eta - (-\delta) = s+\eta + \delta < s+2\delta,
    \end{align*}
    and
    \begin{align*}
        d_g(x_\eta,\gamma(-\varepsilon)) &= s+\eta - (-\varepsilon) = s+\eta + \varepsilon > s+\varepsilon.
    \end{align*}
    Hence 
    \[
    \{x_\eta\}_{\eta \in (0,\delta)} \subset \overline{B(\gamma(-\delta),s+2\delta)}\backslash\overline{B(\gamma(-\varepsilon),s+\varepsilon)} = M(\mathcal{Y}_\delta,s+\delta)\backslash M(\mathcal{Y}_\varepsilon,s).
    \]
    Since $x_\eta \to x_0$ as $\eta\to 0$, we have that $x_0\in Z_\delta$. Thus, $x_0 \in \bigcap_{0<\delta\varepsilon}Z_\delta$ which gives us our desired conclusion. 
\end{proof}

\begin{lem} \label{lem:not_null}
    For each $0 < \delta \leq \eps$ we have that $|Z_\delta|\not=0$.
\end{lem}

\begin{proof}
    Let $0 < \delta \leq \eps$.
    Since open sets have positive measure, it suffices to show that the set $Z_\delta$ has nonempty interior. We have already shown that $x_0\in Z_\delta$. 
    % $ Z_\delta  =\overline{M(\mathcal{Y}_\delta,s+\delta)\backslash M(\mathcal{Y}_\eps,s)}
    % \not=\emptyset$ 
    This implies that $M(\mathcal{Y}_\delta,s+\delta)\backslash M(\mathcal{Y}_\eps,s)\not=\emptyset$. 
    Hence, Lemma \ref{lem:boundary_properties} yields
    % Now as 
    % $$M(\mathcal{Y}_\delta,s+\delta) = \overline{B(\gamma(-\delta),s+2\delta)} = B(\gamma(-\delta),s+2\delta)\cup\partial B(\gamma(-\delta),s+2\delta)$$
    % Taking the set difference yields
    $$ 
        \Big(B(\gamma(-\delta),s+2\delta)\backslash M(\mathcal{Y}_\eps,s)\Big)\cup\Big(\partial B(\gamma(-\delta),s+2\delta)\backslash M(\mathcal{Y}_\eps,s)\Big) \not=\emptyset.
    $$
    
    To finish the proof it suffices to show that the open set $B(\gamma(-\delta),s+2\delta)\backslash M(\mathcal{Y}_\eps,s)$ is not empty. For the sake of contradiction we assume that this is not the case. Then $B(\gamma(-\delta),s+2\delta) \subset M(\mathcal{Y}_\eps,s)$. But since $M(\mathcal{Y}_\eps,s)$ is a closed set, we would conclude that
    $$
        M(\mathcal{Y}_\delta,s+\delta)  = \overline{B(\gamma(-\delta),s+\delta)} \subset M(\mathcal{Y}_\eps,s).
    $$
    This contradicts the fact that $M(\mathcal{Y}_\delta,s+\delta)\backslash M(\mathcal{Y}_\eps,s)\not=\emptyset$. 
\begin{comment}
    Hence $B(\gamma(-\delta),s+2\delta)\backslash M(\mathcal{Y}_\eps,s) \not= \emptyset$. Now $B(\gamma(-\delta),s+2\delta)\backslash M(\mathcal{Y}_\eps,s) \not= \emptyset$ is an open set which satisfies 
    $$
        B(\gamma(-\delta),s+2\delta)\backslash M(\mathcal{Y}_\eps,s) \subset M(\mathcal{Y}_\delta,s+\delta)\backslash M(\mathcal{Y}_\eps,s) \subset Z_\delta    
    $$
    This shows that $Z_\delta$ has a nonempty interior.     
\end{comment}    
\end{proof}

\begin{lem} \label{lem:equal_up_to_null}
    There exists $\Tilde{\delta}\in (0,\eps]$ such that if $0 < \delta \leq\tilde{\delta}$, then 
    $$
        |Z_\delta| = |M(\mathcal{Y}_\delta,s+\delta)\backslash M(\mathcal{Y}_\eps, s)| 
    $$
    In other words $Z_\delta$ and $M(\mathcal{Y}_\delta,s+\delta)\backslash M(\mathcal{Y}_\eps, s)$ differ by a null set for all sufficiently small $\delta$. 
\end{lem}

\begin{proof} 
    From the definition of the set $Z_\delta$ and in light of Lemma \ref{lem:Z_part2} we know that 
    $$
    Z_\delta\backslash\Big(M(\mathcal{Y}_\delta,s + \delta) \backslash M(\mathcal{Y}_\eps,s)\Big)\subset \Big(\overline{B(\gamma(-\delta),s + 2\delta)}\cap \partial B(\gamma(-\eps),s+\eps)\Big)=:W_\delta.
    $$
    So it suffices to show that the set $W_\delta$
    % $$
    %     \overline{B(\gamma(-\delta),s + 2\delta)}\cap \partial B(\gamma(-\eps),s+\eps)
    % $$
    is null for sufficiently small $\delta$. 
    
    % We shall let $W_\delta := \overline{B(\gamma(-\delta),s + 2\delta)}\cap \partial B(\gamma(-\eps),s+\eps)$. 
    By our construction of $\eps > 0$, we know that $\gamma:[-\eps, s+\eps]\to N$ is the unique distance minimizing geodesic between $\gamma(-\eps)$ and $\gamma(s + \eps)$. So we can find geodesic normal coordinates $(U,\varphi)$ centered at $\gamma(-\eps)$ such that $\gamma([-\eps,s+\eps])\subset U$. It follows that $x_0 = \gamma(s)\in U$ and by the proof of Lemma \ref{lem:intersection}  that $\bigcap_{0 < \delta\leq \eps} W_\delta  = \{x_0\}\subset U$. 
    
    For now we consider the compact sets $\{W_\delta\backslash U\}$, and aim to show that for small enough $\delta$ these sets are empty. Suppose that this is not the case. Then by Lemma \ref{lem:subset} the inequality $\delta < \tilde{\delta}$ implies that $W_\delta\subset W_{\tilde{\delta}}$. Hence, 
 $W_\delta\backslash U \subset W_{\tilde{\delta}}\backslash U$. At the same time $\bigcap_{0  <\delta \leq \eps} W_\delta\backslash U = \{x_0\}\backslash U = \emptyset$. Thus, we arrive in a contradiction since a nested sequence of nonempty compact sets has a nontrivial intersection. In particular, there is $\tilde{\delta} \in (0,\eps)$ such that $W_\delta \subset U$ for all  $\delta \in (0,\tilde{\delta}]$. 
 
 % Then we have that 
 %    $$
 %    \overline{B(\gamma(-\tilde{\delta}),s + 2\tilde{\delta})}\cap \partial B(\gamma(-\eps),s+\eps)\subset U.
 %    $$
    Let $\delta \in (0,\tilde{\delta}]$. Since $(U,\varphi)$ are geodesic normal coordinates centered at $\gamma(-\eps)$ and in these coordinates all geodesics emanating from $\gamma(-\eps)$ are straight lines starting at the origin, then it follows that
    $$
        \varphi(W_\delta
        % \overline{B(\gamma(-\tilde{\delta}),s + 2\tilde{\delta})}\cap \partial B(\gamma(-\eps),s+\eps)
        ) \subset \{v\in \R^n: \norm{v} = s +\eps\} \cap \varphi(U).
    $$
    Clearly $\{v\in \R^n:\norm{v} = s+\varepsilon\}\subset \R^n$ is a null set, and since $\varphi:U \to \varphi(U)\subset \R^n$ is a diffeomorphism it follows that 
    % $\overline{B(\gamma(-\tilde{\delta}),s + 2\tilde{\delta})}\cap \partial B(\gamma(-\eps),s+\eps)$ 
    $W_\delta$
    is a null set 
    % for sufficiently small $\delta$ 
    by \cite[Proposition 6.5]{lee_smooth}. 
\end{proof}

We are in position to prove Proposition \ref{prop:summary}.
\begin{proof}[Proof of Proposition \ref{prop:summary}]
The claims of Proposition \ref{prop:summary} follow from lemmas \ref{lem:subset}
-- 
% \ref{lem:intersection},  \ref{lem:not_null} and 
\ref{lem:equal_up_to_null}.    
\end{proof}

\subsection{Determination of the Multiplicative Gauge} \label{sub:determine_gauge}

\begin{lem} \label{lem:product_equal}
    If $\Lambda_1 = \Lambda_2$, $T > 0$ and $f,h\in C_0^\infty((0,T)\times\cX)$, then
    \begin{align} \label{product}
        u_1^f(T,x_0)\overline{u_1^h(T,x_0)} = u_2^f(T,x_0)\overline{u_2^h(T,x_0)} \hspace{4mm}\text{for all $x_0\in N$.}
    \end{align}
\end{lem}

\begin{proof}
    First note that for all $t > 0$ and all $x_0\in \cX$ we have that
    \begin{align*}
        u_1^f(t,x_0) = \Lambda_1(f)(t,x_0) &= \Lambda_2(f)(t,x_0) = u_2^f(t,x_0).
    \end{align*}
    So $(\ref{product})$ clearly holds for $x\in\cX$ when $t = T$. Since $u_1^f(T,\cdot)$ and $u_2^f(T,\cdot)$ are both continuous, then the pointwise equality extends to $\overline{\cX}$ as well. 
    
    Now let $x_0 \in \text{int}(M(\cX,T))$ so that $d_g(\cX,x_0) < T$, then we can find $y\in \cX$ such that there is a unique geodesic $\gamma$ connecting $y$ to $x$, which is parameterized so that $\gamma(0) = y, \gamma(s) = x_0$ and $d_g(y,x_0)= s < T$. Note that neither $x$ nor $y$ is in the cut locus of the other. We can find $\eps > 0$ to construct sets $\mathcal{Y}_\delta$ for $\delta\in (0,\eps]$ as given before Definition \ref{def:set_definition}. Furthermore, we can choose $\eps > 0$ sufficiently small so that $s + \delta\in (0,T]$ for $\delta\in (0,\eps]$. Finally, we take $Z_\delta$ as in Definition $\ref{def:set_definition}$. 
    
    % To summarize, we have $\eps > 0$ small enough so that $\mathcal{Y}_\delta \subset \cX$ and $s, s + \delta \in (0,T]$ for all $\delta\in (0,\eps]$. 
    Hence, due to the equation \eqref{eq:difference} we have for all $\delta\in (0,\eps]$ that
    \begin{align}
        \langle 1_{M(\mathcal{Y}_\delta,s + \delta)\backslash M(\mathcal{Y}_\eps,s)} u^f_1(T), u_1^h(T)\rangle_{L^2(N)} &= \langle 1_{M(\mathcal{Y}_\delta,s + \delta)\backslash M(\mathcal{Y}_\eps,s)} u^f_2(T), u_2^h(T)\rangle_{L^2(N)}. \label{eq:inner_product_crecent}
    \end{align}
    Now by Proposition \ref{prop:summary}, the sets $\{Z_\delta\}$ satisfy the conditions for the Lebesgue Differentiation Theorem since the functions being integrated are continuous. Additionally for sufficiently small $\delta$ we have that $1_{M(\mathcal{Y}_\delta,s + \delta)\backslash M(\mathcal{Y}_\eps,s)}$ and $1_{Z_\delta}$ are equal almost everywhere. Therefore, for $k\in \{1,2\}$ we have that
    \begin{align*}
        \lim_{\delta\to 0}  \frac{\langle 1_{M(\mathcal{Y}_\delta,s + \delta)\backslash M(\mathcal{Y}_\eps,s)} u^f_k(T), u_k^h(T)\rangle_{L^2(N)}}{|M(\mathcal{Y}_\delta,s + \delta)\backslash M(\mathcal{Y}_\eps,s)|} 
        % &= \lim_{\delta\to 0} \frac{\langle 1_{Z_\delta} u^f_k(T), u_k^h(T)\rangle_{L^2(N)}}{|Z_\delta|}\\
        % &= \lim_{\delta\to 0}\frac{1}{|Z_\delta|}\int_{Z_\delta} u_k^f(T,x)\overline{u_k^h(T,x)} dV_g(x)\\
        &= u_k^f(T,x_0)\overline{u_k^h(T,x_0)}.
    \end{align*}
    Hence on $\text{int}(M(\cX, T))$, the equality \eqref{product}  follows from $(\ref{eq:inner_product_crecent})$. Using the continuity again, we can extend the equality \eqref{product} on $M(\cX,T)$. 
    
    Finally, the equality  \eqref{product} holds of $N\backslash M(\cX,T)$ due to Finite Speed of Wave Propagation (i.e. Corollary \ref{cor:finite_speed}).
    % that $\text{supp}(u_k^f(T,\cdot)) \subset M(\cX,T)$.
    % Thus $u_k^f(T,x) =0$ for all $x\in N\backslash M(\cX,T)$ and desired equality holds trivially.
    \end{proof}
\begin{comment}
    \begin{cor}
        Assume that $\Lambda_1 = \Lambda_2$, let $T > 0$ and let $f,h\in C_0^\infty((0,T)\times\cX)$. Then for all $x\in N$ we have
        \begin{align*}
        u_1^f(T,x)\overline{u_1^h(T,x)} &= u_2^f(T,x)\overline{u_2^h(T,x)}
        \end{align*}    
    \end{cor}

    \begin{proof}
        As $C_0^\infty((0,T)\times\cX)\subset C_0^\infty(\cX\times (0,T))$, the result immediately holds for all $x\in M(\cX,T)$. Since $\text{supp}(u_k^f(\cdot, T))\subset M(\cX,T)$ holds by Proposition \ref{finite-speed}, we have the $u_k^f(T,x) = 0$ for $x\not\in M(\cX, T)$ so the equality trivially holds. 
    \end{proof}

    \begin{cor}\label{same-magnitude}
        Assume that $\Lambda_1 = \Lambda_2$, let $T > 0$ and let $f\in C_0^\infty((0,T)\times\cX)$. Then $|u_1^f(T,x)| = |u_2^f(T,x)|$. 
    \end{cor}

    \begin{proof}
        Take $f = h$ and then take the square root. 
    \end{proof}
\end{comment}
    \begin{lem}\label{lem:big_result}
        If $\Lambda_1 = \Lambda_2$, $T > 0$, $\Omega\subset M(\cX,T)$ is open and bounded, then there exists a unique smooth unitary function $\kappa:\Omega\to \mathbb{C}$ such that $\kappa|_{\Omega\cap\cX} = 1$ and
        \begin{align}
            u_1^f(T,x) &= \kappa(x)u_2^f(T,x) \hspace{4mm}\text{for all $x\in \Omega$ and for all $f\in C_0^\infty((0,T)\times\cX)$.} \label{gauge-equation}
        \end{align}
    \end{lem}
    \begin{proof}
        We begin with the existence proof. Fix $x_0\in \Omega$, then by Corollary \ref{cor:nonvanishing_condition} there exists $h\in C_0^\infty(\cX\times (0,T))$ such that $u_1^h(T,x_0)\not=0$. By Lemma \ref{lem:product_equal} it follows that $|u_2^h(T,x)| = |u_1^h(T,x)|\not=0$. By continuity of the solution we can then find an open neighborhood $U_{x_0}$ of $x_0$ such that $u_1^h(T,x)\not=0$ for all $x\in U_{x_0}$. Thus, the smooth function 
        \begin{align*}
            \kappa_{x_0}\colon U_{x_0}\to \C, \quad  \kappa(x):=\frac{\overline{u_2^h(T,x)}}{\overline{u_1^h(T,x)}}
        \end{align*}
        is well defined.
        Hence, by Lemma \ref{lem:product_equal} we have on $U_{x_0}$, that 
        $$
            u_1^f(T,x) = \kappa_{x_0}(x)u_2^f(T,x) \,\,\text{for all $f\in C_0^\infty((0,T)\times\cX)$.}
        $$
        
        Let now $\{U_{x_0}\}_{x_0\in \Omega}$ be an open cover for $\Omega$. We can choose a countable subcover $\{U_n\}_{n\in\N}\subset \{U_{x_0}\}_{x\in\Omega}$ and a smooth partition of unity $\{\psi_n\}_{n\in\N}$ subordinate to $\{U_n\}$. Therefore, the function
        \begin{align*}
            \kappa(x) :&= \sum_{n\in\N}\psi_n(x)\kappa_n(x)
        \end{align*}
        is a smooth in $\Omega$. Moreover for any $x\in \Omega$, there are finitely many $\psi_n$ which we shall call $\psi_{n_1},\ldots,\psi_{n_l}$ which are nonzero in an open neighborhood around $x$. By construction of the $\kappa_n$, we have that
        \begin{align*}
            \kappa(x)u_2^f(T,x) &= \sum_{m=1}^l \psi_{n_m}(x)\kappa_{n_m}(x_0)u_2^f(T,x) = \sum_{m=1}^l \psi_{n_m}(x)u_1^f(T,x) = u_1^f(T,x).
        \end{align*}
        This yields \eqref{gauge-equation}. 
        
        Let us now show that $|\kappa(x)| = 1$ for all $x\in\Omega$. We have already shown that $u_1^f(T,x) = \kappa(x)u_2^f(T,x)$ for all $f\in C_0^\infty((0,T)\times\cX)$ and all $x\in\Omega$. For any $x_0\in\Omega$ we have that $d_g(x_0,\cX) < T$, so by Corollary \ref{cor:nonvanishing_condition} we may find $f\in C_0^\infty((0,T)\times\cX)$ such that $u_2^f(T,x_0)\not=0$. Thus, by Lemma \ref{lem:product_equal} we also get that $|u_1^f(T,x_0)| = |u_2^f(T,x_0)|\not=0$. Then 
        \begin{align*}
            |\kappa(x_0)| &= \Big|\frac{u_1^f(T,x_0)}{u_2^f(T,x_0)}\Big| = 1.
        \end{align*}
        This shows that $\kappa\in C^\infty(\Omega,\mathbb{S}^1)$. Now let $x_0\in\cX$, then we find $f\in C_0^\infty((0,T)\times\cX))$ so that $u_1^f(T,x_0) \not= 0$. Since $x_0\in \cX$, then 
        $$
        u_1^f(T,x_0) = \Lambda_1(f)(T,x_0) = \Lambda_2(f)(T,x_0) = u_2^f(T,x_0).
        $$
        This and \eqref{gauge-equation} prove that $\kappa(x_0) = 1$ and we have proven the existence of desired $\kappa$. 
        
        To show uniqueness, let $\kappa_1$ and $\kappa_2$ be two such functions which satisfy for all $x\in\Omega$ and all $f\in C_0^\infty((0,T)\times\cX)$
        \begin{align*}
            u_1^f(T,x) &= \kappa_1(x)u_2^f(T,x)= \kappa_2(x)u_2^f(T,x)
        \end{align*}
        It follows that $(\kappa_1(x)-\kappa_2(x))u_2^f(T,x) = 0$ for all $x\in\Omega$ and $f\in C_0^\infty((0,T)\times\cX)$. Again using Corollary \ref{cor:nonvanishing_condition} we can show that there exists $f\in C_0^\infty((0,T)\times \cX)$ such that $u_1^f(T,x_0)\not=0$ for any $x_0\in \Omega$. It follows that $\kappa_1|_\Omega = \kappa_2|_{\Omega}$.
\end{proof}

To complete the proof of Theorem \ref{thm:main_thm} we give the following Proposition.
% Recall that Lemma \ref{lem:gauge_equivalent} has the equivalent condition if we can show that $u_1^f(t,x) = \kappa(x)u_2^f(t,x)$ for all $(t,x)$. However, Lemma \ref{lem:big_result} is stated only at the time $t = T$. The following lemma shows that the equality at $t = T$ is sufficient to prove that $\mathcal{L}_{g,A_1,V_1} = \mathcal{L}_{g,A_2 + i\kappa^{-1}d\kappa,V_2}$ which is equivalent to proving our gauge for the lower order terms.

\begin{thm}\label{thm:gauge}
    If $\Lambda_1 = \Lambda_2$, then there exists a unique smooth unitary function $\kappa:N\to\mathbb{C}$ such that $\kappa\equiv 1$ on $\cX$ and 
    \begin{align*}
        A_1 &= A_2 + i\kappa^{-1}d\kappa \hspace{4mm}\text{ and } \hspace{4mm}V_1 = V_2
    \end{align*}
\end{thm}
\begin{proof}
    This was originally proven in \cite[Lemma 7]{kian2019unique} on the case of a manifold with boundary. We provide the full proof here for the convenience of the reader.

    Let $T > 0$ and let $\Omega \subset M(\cX,T)$ be open and bounded, then by Lemma \ref{lem:big_result} there exists a unique smooth unitary function $\kappa:\Omega\to \mathbb{C}$ such that $\kappa|_{\Omega\cap \cX} = 1$ and $(\ref{gauge-equation})$ holds. Let $f\in C_0^\infty((0,T)\times\cX)$, then there exists $\eps > 0$ such that $\text{supp}(f)\subset \cX\times[\eps,T)$. For $s\in[0,\eps)$ we can define functions $f_s(t,x)= f(x,t+s)$ and will have that $\text{supp}(f_s)\subset [\eps- s,T-s) \times \cX\subset (0,T)\times\cX$. So for $s\in [0,\eps)$ we have that $f_s\in C_0^\infty((0,T)\times\cX)$. Since our Magnetic-Schr\"odinger operators are time-independent, then our solutions have the following time invariance property for $s\in [0,\eps)$
    \begin{align*}
        u_i^{f_s}(t,x) &= u_i^f(x,t+s), \quad \text{ for both } i \in \{1,2\}\,\,\text{and $(t,x)\in (0,\infty)\times N$}.
    \end{align*}
    Since $f_s\in C_0^\infty((0,T)\times\cX)$ for $s\in [0,\eps)$ we have in combination with $(\ref{gauge-equation})$ that 
    \begin{align}
        u_1^f(T+s,x) &= \kappa(x) u_2^f(T+s,x) \hspace{12mm}\text{for all $x\in N$}. \label{T+s}
    \end{align}
    By differentiating $(\ref{T+s})$ twice with respect to $s$ and evaluating  the second derivative at $s=0$ yields the following equation
    \begin{align*}
        \partial_t^2u_1^f(T,x) &= \kappa(x)\partial_t^2u_2^f(T,x), \quad 
        \text{ for all } x \in \Omega.
    \end{align*}
    Since the functions $u_1^f$ and $u_2^f$ satisfy the Cauchy Problem \eqref{eq:cauchy_problem} we have that
        \begin{align*}
        -\mathcal{L}_1(u_1^f)(T,x) + f(T,x) &= \kappa(x)[-\mathcal{L}_2(u_2^f)(T,x) + f(T,x)] \quad 
        \text{ for all } x \in \Omega.
       \end{align*}
    Here we denoted $\mathcal{L}_i=\mathcal{L}_{g,A_i,V_i}$.
    Since $f\in C_0^\infty((0,T)\times\cX)$, we arrive at
    \begin{align*}
        \mathcal{L}_1(u_1^f)(T,x) &= \kappa(x)\mathcal{L}_2(u_2^f)(T,x) = \kappa \mathcal{L}_2(\kappa^{-1}u^f_1)(T,x)\quad 
        \text{ for all } x \in \Omega.
    \end{align*}
    
    Now let $v\in C_0^\infty(\Omega)$. Since $\Omega$ is bounded, then by Approximate Controllability (i.e. Theorem \ref{thm:approximate_controllability}) we may find waves $\{u_1^{f_j}(T)\}$ such that $u_1^{f_j}(T)\to v$ in $C^2(\overline{\Omega})$. Since the convergence is in $C^2$ and the Magnetic-Schr\"odinger operators are of second order, then $\mathcal{L}_1 = \kappa \mathcal{L}_2(\kappa^{-1}v)$ for all $v\in C_0^\infty(\Omega)$. Similar to the proof of Lemma \ref{lem:gauge_equivalent}, it follows that $\mathcal{L}_1(v) = \kappa \mathcal{L}_2(\kappa^{-1}v)$ for all $v\in C_0^\infty(\Omega)$ is equivalent to $A_1 = A_2 + i\kappa^{-1}d\kappa$ and $V_1 = V_2$ on $\Omega$. Thus for any $T > 0$ and any bounded open set $\Omega\subset M(\cX,T)$ there is a unique smooth unitary function $\kappa_\Omega $ such that 
    \begin{align*}
        A_1 = A_2 + i\kappa^{-1}_\Omega d\kappa_\Omega \,\,\,\,\text{and}\,\,\,\, V_1=V_2 \hspace{4mm}\text{ on $\Omega$.}
    \end{align*}
    
    Thus for $j\in\N$ and any bounded open set $\Omega_1,\Omega_2 \subset M(\cX,j)$ with functions $\kappa|_{\Omega_1}$ and $\kappa_{\Omega_2}$ as above we have by the uniqueness that $\kappa_{\Omega_1}|_{\Omega_1\cap\Omega_2} = \kappa_{\Omega_2}|_{\Omega_1\cap\Omega_2}$. Since the sets $\{M(\cX,j)\}_{j=1}^\infty$ exhaust $N$ it follows that there exists a unique smooth unitary function $\kappa:N\to\mathbb{C}$ with $\kappa|_{\cX} = 1$ and  
    \begin{align*}
        A_1 = A_2 + i\kappa^{-1}d\kappa \,\,\,\,\text{and}\,\,\,\, V_1=V_2 \hspace{4mm}\text{ on $N$.}
    \end{align*}\end{proof}

\begin{proof}[Proof of Theorem \ref{thm:main_thm}]
The claims follow from Theorem \ref{thm:isometry_between_manifolds} and Theorem \ref{thm:gauge}.
\end{proof}

\printbibliography

\end{document}